\documentclass[a4paper,11pt]{article}
\usepackage{pdfsync}

\usepackage[plainpages=false]{hyperref}
\usepackage{amsfonts,amsmath,amssymb,amsthm,enumerate}
\usepackage{latexsym,lscape,rawfonts,mathrsfs,mathtools}
\usepackage{enumitem,pifont}
\usepackage{cases}

\usepackage[all]{xy}
\usepackage{eufrak}
\usepackage{makeidx}
\usepackage{graphicx,psfrag}
\usepackage{pstool}
\usepackage{tikz-cd}
\usepackage{float}
\usepackage{array}
\usepackage{diagbox}

\usepackage{setspace}

\usepackage[titletoc,title]{appendix}

\usepackage{txfonts}

\newcommand{\IF}{\ensuremath{\mathbb{F}}}

\newcommand{\MM}{\ensuremath{\mathcal{M}}}
\newcommand{\NN}{\ensuremath{\mathcal{N}}}
\newcommand{\WW}{\ensuremath{\mathcal{W}}}
\newcommand{\IN}{\ensuremath{\mathbb{N}}}
\newcommand{\IM}{\ensuremath{\mathbb{M}}}

\newcommand{\R}{\ensuremath{\mathbb{R}}}
\newcommand{\RR}{\ensuremath{\mathcal{R}}}
\newcommand{\MS}{\ensuremath{\mathcal{S}}}

\newcommand{\ba}{\begin{align*}}
\newcommand{\ea}{\end{align*}}
\newcommand{\na}{\nabla}
\newcommand{\la}{\langle}
\newcommand{\ra}{\rangle}
\newcommand{\lc}{\left(}
\newcommand{\rc}{\right)}

\newcommand{\ep}{\epsilon}

\newcommand{\XX}{\mathcal{X}}
\newcommand{\tf}{\mathfrak{t}}
\renewcommand{\t}{\mathfrak{t}}
\newcommand{\Var}{\text{Var}}
\newcommand{\CF}{\mathfrak{C}}
\newcommand{\II}{\mathcal{I}}
\newcommand{\td}[1]{\widetilde{#1}}

\makeatletter
\newcommand*\owedge{\mathpalette\@owedge\relax}
\newcommand*\@owedge[1]{%
\mathbin{%
\ooalign{%
$#1\m@th\bigcirc$\cr
\hidewidth$#1\m@th\wedge$\hidewidth\cr
}%
}%
}
\makeatother

\makeatletter
\def\ExtendSymbol#1#2#3#4#5{\ext@arrow 0099{\arrowfill@#1#2#3}{#4}{#5}}

\makeatother

\makeatletter
\def\ExtendSymbol#1#2#3#4#5{\ext@arrow 0099{\arrowfill@#1#2#3}{#4}{#5}}

\makeatother

\def\XXint#1#2#3{{\setbox0=\hbox{$#1{#2#3}{\int}$ }
\vcenter{\hbox{$#2#3$ }}\kern-.55\wd0}}

\numberwithin{equation}{section}

\newtheorem{thm}{Theorem}[section]
\newtheorem{cor}[thm]{Corollary}
\newtheorem{prop}[thm]{Proposition}
\newtheorem{lem}[thm]{Lemma}

\newtheorem{rem}[thm]{Remark}

\newtheorem{defn}[thm]{Definition}

\title{Heat kernel on Ricci shrinkers (II)}
\author{Yu Li \quad and \quad Bing Wang}
\date{\today}

\begin{document}
\maketitle

\begin{abstract}
This paper is the sequel to our study of heat kernels on Ricci shrinkers in \cite{LW20}. In this paper, we improve many estimates in \cite{LW20} and extend the recent progress of Bamler \cite{Bam20a}. In particular, we drop the compactness and curvature boundedness assumptions and show that the theory of $\IF$-convergence holds naturally on any Ricci flows induced by Ricci shrinkers.
\end{abstract}

\tableofcontents

\section{Introduction}

A Ricci shrinker $(M^n, g, f)$ is a complete Riemannian manifold $(M^n,g)$ coupled with a smooth function $f$ satisfying 
\begin{align} 
Rc+\text{Hess}\,f=\frac{1}{2}g \label{E100},
\end{align}
where the potential function $f$ is normalized so that
\begin{align} 
R+|\nabla f|^2&=f \label{E101}.
\end{align}

The study of shrinkers is an essential component of analyzing the singularity formation of solutions to the Ricci flow. For a Ricci flow with type-I curvature bound, it is proved by Enders-M\"uller-Topping \cite{EMT11} that any proper blow-up sequence converges smoothly to a nontrivial Ricci shrinker. For general compact Ricci flows, it is proved by Bamler \cite{Bam20c} that the finite-time singularities are modeled on Ricci shrinkers containing a singular set by using the theory of $\IF$-convergence developed in \cite{Bam20a, Bam20b, Bam20c}.

In dimension $2$ or $3$, all Ricci shrinkers are completely classified (cf.~\cite{Ha95}\cite{Naber}\cite{NW}\cite{CCZ}, etc). We know that $\R^2,S^2,\R^3,S^3,S^2 \times \R$ and their quotients form the complete list. In particular, all low-dimensional Ricci shrinkers have bounded and nonnegative sectional curvature.

In higher dimensions, the complete classification of Ricci shrinkers seems out of reach. Subject to an additional curvature positivity assumption, some partial classifications are also known (cf. \cite{Naber}\cite{MW17}\cite{LNW18}\cite{LN20}\cite{Na19}). In general, it is still unclear if there exists any Ricci shrinker with unbounded sectional curvature. 

On the one hand, Ricci shrinkers can be regarded as critical metrics which generalize the classical positive Einstein manifolds. On the other hand, for any Ricci shrinker, there exists an associated self-similar solution to the Ricci flow (cf. Section 2). As a special class of Ricci flows, Ricci shrinkers have many known important properties of compact Ricci flows. In \cite{LW20}, many fundamental analytic tools, including the maximum principle, optimal log-Sobolev constant estimate, the no-local-collapsing theorems, etc., are established for Ricci flows associated with Ricci shrinkers. Many heat kernel estimates include the differential Harnack inequality and the pseudolocality theorem are also known in \cite{LW20}.

In this paper, we continue to focus on Ricci flows associated with Ricci shrinkers without any curvature assumption. Based on the techniques and results in \cite{LW20} and \cite{Bam20a}, we further obtain results, including a Gaussian bound on the heat kernel, no-local-collapsing and non-expanding estimates, an $\ep$-regularity theorem, etc. All those results are stronger than their counterparts in \cite{LW20}.
It is important to notice that we have no assumption of curvature at all. If we assume bounded curvature on non-compact manifold, then many results are already known (cf.~\cite{Bam21}~\cite{CMZa}). 

The pointed Nash entropy (cf. Definition \ref{def:nash}) plays an important role in \cite{Bam20a}, which first appears in~\cite[Section 5]{Pe1} and is systematically studied in~\cite{HN14}. 
In \cite{LW20}, we use Perelmam's entropy $\boldsymbol{\mu}$ (see \eqref{eq:mu}) to characterize the optimal log-Sobolev constant and the local non-collapsing. The pointed Nash entropy, which is always bounded below by $\boldsymbol{\mu}$, has the advantage of being local in the spacetime of Ricci flows. In \cite{HN14}, it is proved that the Nash entropy is Lipschitz. Moreover, the oscillation of the Nash entropy in the spacetime is established in \cite{Bam20a}. We generalize the Nash entropy and its fundamental estimates to the Ricci flows associated with Ricci shrinkers; see Theorem \ref{thm:T305} and Corollary \ref{cor:404}. 

\begin{thm} \label{thm:101}
Let $(M^n,g(t))_{t<1}$ be the Ricci flow associated with a Ricci shrinker. Then for any $s<t<1$, the Nash entopy $\NN_s^*(x,t):=\NN_{(x,t)}(t-s)$ is smooth and satisfies the following estimates on $M \times (s,1)$.
\begin{align} \label{E102a}
|\na \NN_s^*| \le \sqrt{\frac{n}{2(t-s)}}\quad \text{and} \quad -\frac{n}{2(t-s)} \le \square \NN_s^* \le 0.
\end{align}
\end{thm}

The proof of \eqref{E102a} is based on an integral estimate of the heat kernel (cf. Theorem \ref{thm:T304}), which was initially obtained in \cite{Bam20a} for compact Ricci flows. A key application of Theorem \ref{thm:101} is to estimate the local oscillation of the Nash entropy (cf. Corollary \ref{cor:303}). Using the Nash entropy properties and the heat kernel estimates, we obtain the improved no-local-collapsing and non-expanding result (cf. Theorem~\ref{thm:volume2} and Theorem~\ref{thm:volume4}).

\begin{thm}[\textbf{No-local-collapsing and non-expanding}] \label{thm:102}
Let $(M^n,g(t))_{t<1}$ be the Ricci flow associated with a Ricci shrinker. For any $x \in M$ and $t<1$,
\begin{align*}
|B_t(x,r)|_t \le C(n) \exp\lc \NN_{x,t}(r^2) \rc r^n
\end{align*}
and if $R \le r^{-2}$ on $B_t(x,r)$, then
\begin{align*}
|B_t(x,r)|_t \ge c(n) \exp\lc \NN_{x,t}(r^2) \rc r^n.
\end{align*}
\end{thm}

Note that $\NN_{x,t}(r^2) \leq 0$ (cf. Corollary~\ref{cor:302}),  it is clear that Theorem~\ref{thm:102} provides a uniform volume ratio upper bound, independent of base point and radius.
This clearly improves the known volume upper bounds (cf. ~\cite{CZ10},~\cite{HM11},~\cite{LLW21}). 
On the other hand, as $\boldsymbol{\mu} \leq \NN_{x,t}(r^2)$, the non-collapsing estimate in Theorem~\ref{thm:102} also improves the one in~\cite{LW20}. 

An important concept introduced in \cite{Bam20a} is the $H$-center (cf. Definition \ref{def:Hcenter}). Roughly speaking, an $H$-center is a point around which the conjugate heat kernel is concentrated (cf. Proposition \ref{prop:304}). In addition, for any two conjugate heat kernels, the $W_1$-Wasserstein distance between them can be roughly measured by the distance between two $H$-centers. We prove the existence of an $H_n$-center, where $H_n=(n-1)\pi^2/2+4$, for any conjugate heat kernel, by generalizing the monotonicity of the variance obtained in \cite{Bam20a} to our setting (cf. Proposition \ref{prop:302}, Proposition \ref{prop:303}). By using these concepts and related techniques, we have the following heat kernel estimates (cf. Theorem \ref{thm:lower}, Theorem \ref{thm:heatupper}, Theorem \ref{thm:heatupper3}).

\begin{thm}[\textbf{Heat kernel estimates}] \label{thm:103}
Let $(M^n,g(t))_{t<1}$ be the Ricci flow associated with a Ricci shrinker satisfying $\boldsymbol{\mu} \ge -A$. Then the following properties hold.
\begin{enumerate}[label=(\roman*)]
\item There exists a constant $C=C(n,A,\delta)>1$ such that
\begin{align} \label{E103b}
\frac{C^{-1}}{(t-s)^{\frac n 2}} \exp \lc-\frac{d_s^2(x,y)}{C^{-1}(t-s)} \rc \le H(x,t,y,s) \le \frac{C}{(t-s)^{\frac n 2}} \exp \lc-\frac{d_s^2(x,y)}{C(t-s)} \rc
\end{align}
for any $-\delta^{-1} \le s <t \le 1-\delta$ and $d_t(p,x) \le \delta^{-1}$.

\item For any $\ep>0$, there exists a constant $C = C(n,\ep) >0$ such that
\begin{align} \label{E103c}
H(x,t,y,s) \le \frac{C \exp \lc -\NN_{(x,t)}(t-s) \rc}{(t-s)^{\frac n 2}} \exp \lc -\frac{d_s^2(z,y)}{(4+\ep)(t-s)} \rc,
\end{align}
for any $s<t<1$ and any $H_n$-center $(z,s)$ of $(x,t)$.
\end{enumerate}
\end{thm}

Here, the point $p$ is a minimum point of $f$, regarded as the Ricci shrinker's base point. The Gaussian estimate \eqref{E103c} is previously proved in \cite{Bam20a} for compact Ricci flows, with $4+\ep$ replaced by $8+\ep$. Our proof uses an iteration argument by showing that if \eqref{E103c} fails, one can find a new spacetime point $(x',t')$ with an $H_n$-center $(z',s)$ such that $H(x',t',y,s)$ has a worse bound than \eqref{E103c}. Eventually, we will arrive at a contradiction if $t'$ is sufficiently close to $s$. The proof in our case is more involved since we do not have a global heat kernel bound as \eqref{E103c} when $t$ is close to $s$, which is always available for compact Ricci flows. Therefore, in the iteration process, we must carefully choose the sequence of spacetime points, so they all fall into a compact set. Then the contradiction comes from the local heat kernel estimate (cf. Corollary \ref{cor:401}) since locally the scalar curvature is bounded. 

Once we have the estimate \eqref{E103c}, the upper bound in \eqref{E103b} follows since the distance between $(x,s)$ and $(z,s)$ can be well-controlled. Moreover, the lower bound in \eqref{E103b} is already contained in \cite{LW20} in a different guise. We also obtain the gradient estimate of the heat kernel; see Theorem \ref{thm:gra}.

By the monotonicity of the $W_1$-Wasserstein distance between two conjugate heat kernels (cf. Proposition \ref{prop:301}), it is natural to consider new $P^*$-parabolic neighborhoods in the spacetime of the Ricci flow, as pointed out in \cite{Bam20a} (cf. Definition \ref{def:pnei}, \eqref{E501a}, \eqref{E501aa}). Comparing the $P^*$-parabolic neighborhoods with the conventional ones, we have the following result (cf. Proposition \ref{prop:Hcent1a}, Proposition \ref{prop:com1}, Proposition \ref{prop:Hcent2}, Proposition \ref{prop:Hcent3}).

\begin{thm} \label{thm:104}
Let $(M^n,g(t))_{t<1}$ be the Ricci flow associated with a Ricci shrinker satisfying $\boldsymbol{\mu} \ge -A$. Then the following properties hold.
\begin{enumerate}[label=(\roman*)]
\item Given $\delta \in (0,1)$, $t_0 \in (-\infty,1)$, $T^{\pm} \ge 0$ and $S \ge 0$, there exists a constant $C=C(n,A,\delta)>1$ such that
\begin{align*}
P^*(p,t_0;S,-T^-,T^+) \subset Q(p,t_0; \sqrt{2} S+C,-T^-,T^+) \subset P^*(p,t_0;\sqrt{2}S+2C,-T^{-},T^+)
\end{align*}
provided that $t_0-T^- \ge -\delta^{-1}$.

\item There exists a constant $\rho=\rho(n,A) \in (0,1)$ satisfying the following property.
Given $(x_0,t_0) \in M \times (-\infty,1)$ and $r>0$, suppose that $R \le r^{-2}$ on $P(x_0,t_0;r,-(\rho r)^2,(\rho r)^2)$. Then
\begin{align*} 
P(x_0,t_0; \rho r) &\subset P^*(x_0,t_0;r, -(\rho r)^2,(\rho r)^2) \quad \text{and} \quad P^*(x_0,t_0; \rho r) \subset P(x_0,t_0; r, -(\rho r)^2,(\rho r)^2).
\end{align*}
\end{enumerate}
\end{thm}

The proof of Theorem \ref{thm:104} involves the distance distortion estimates globally with respect to $p$ and locally under the scalar curvature control. Moreover, one needs to locate the $H_n$-center of $(p,t_0)$ or $(x_0,t_0)$. Notice that, if $t_0+T^+<1$, Theorem \ref{thm:104} implies that any $P^*(p,t_0;S,-T^-,T^+)$ is precompact, i.e., its closure is compact.
By using the estimates of the Nash entropy and $P^*$-neighborhoods, one has the following $\ep$-regularity theorem (cf. Theorem \ref{thm:epr}), which is proved in \cite{Bam20a} for compact Ricci flows. Here, $r_{\text{Rm}}$ is the spacetime curvature radius, whose definition can be found in Definition \ref{def:curv}.

\begin{thm}[\textbf{$\ep$-regularity theorem}]\label{thm:105}
There exists a small constant $\ep=\ep(n)>0$ satisfying the following property.

Let $(M^n,g(t))_{t<1}$ be the Ricci flow associated with a Ricci shrinker. Given $(x,t) \in M \times (-\infty,1)$ and $r>0$, suppose that $\NN_{(x,t)}(r^2) \ge -\ep$, then $r_{\emph{Rm}}(x,t) \ge \ep r$.
\end{thm}

Based on the results and techniques generalized (or slightly improved) from \cite{Bam20a}, we can generalize the theory about metric flows and $\IF$-convergence in \cite{Bam20b} and \cite{Bam20c} from compact Ricci flows to the setting of Ricci flows associated with or induced by Ricci shrinkers (cf. Definition~\ref{dfn:RA19_2}). In particular, a pointed Ricci flow induced by a Ricci shrinker can be regarded as a metric flow pair in the sense of \cite[Definition 5.1]{Bam20b}. Therefore, any sequence of pointed Ricci shrinkers induced by Ricci shrinkers with $\boldsymbol{\mu} \ge -A$, by taking a subsequence, will $\IF$-converge to a limit metric flow admitting concrete structure theorems (cf. Theorem \ref{thm:601}, Theorem \ref{thm:602}). As an application of the theory of $\IF$-convergence, we have the following two-sided pseudolocality theorem. Notice that the forward pseudolocality theorem is proved in \cite[Theorem 24]{LW20}. Thus, to obtain a two-sided pseudolocality,  it suffices to obtain a backward pseudolocality, which is proved in Theorem~\ref{thm:603}. 

\begin{thm}[\textbf{Two-sided pseudolocality theorem}] \label{thm:106}
For any $\alpha > 0$, there is an $\ep (n, \alpha) > 0$ such that the following holds.

Let $(M^n,g(t))_{t<1}$ be a Ricci flow associated with a Ricci shrinker. Given $(x_0, t_0) \in M \times (-\infty,1)$ and $r > 0$, if
\begin{equation*}
|B_{t_0}(x_0,r)| \geq \alpha r^n, \qquad |Rm| \leq (\alpha r)^{-2} \quad \text{on} \quad B_{t_0}(x_0,r),
\end{equation*}
then
\[ |Rm| \leq (\ep r)^{-2} \quad \text{on} \quad P(x_0,t_0; (1-\alpha) r, -(\ep r)^2,(\ep r)^2). \]
\end{thm}

Another application of the $\IF$-converge is the following integral estimate of curvature, which originates from the estimate of Cheeger-Naber~\cite{CN13}. 
For more details, see Theorem~\ref{thm:604} and Corollary~\ref{cor:602}. 

\begin{thm} \label{thm:107}
Let $(M^n,g,f,p)$ be a Ricci shrinker in $\MM(A)$. Then
\begin{align*} 
\int_{d(p,\cdot) \le r} |Rm|^{2-\ep} \,dV \le& \int_{d(p,\cdot) \le r} r_{\emph{Rm}}^{-4+2\ep} \,dV \le C r^{n+2\ep-2}, \\
\int_{d(p,\cdot) \ge 1} \frac{|Rm|^{2-\ep}}{d^{n+2\ep-2}(p,\cdot)} \,dV \le& \int_{d(p,\cdot) \ge 1} \frac{r_{\emph{Rm}}^{-4+2\ep}}{d^{n+2\ep-2}(p,\cdot)} \,dV \le C
\end{align*}
for any $\ep>0$ and $r \ge 1$, where $r_{\emph{Rm}}(\cdot)=r_{\emph{Rm}}(\cdot,0)$ and $C=C(n,A,\ep)$.
\end{thm}

This paper is organized as follows. Section 2 discusses some properties of Ricci flows associated with Ricci shrinkers, including the existence of cutoff functions and maximum principles. In Section 3, we prove some estimates and properties regarding the variance, $H$-centers and the Nash entropy. Section 4 focuses on various estimates of the heat kernel. In Section 5, we prove the theorems about the parabolic neighborhoods and the $\ep$-regularity theorem. In the last section, we generalize the theory of $\IF$-convergence in our setting and prove some applications in Ricci shrinkers.

{\bf Acknowledgements}: 

Yu Li is supported by YSBR-001, NSFC-12201597 and research funds from USTC (University of Science and Technology of China) and CAS (Chinese Academy of Sciences). Bing Wang is supported by YSBR-001, NSFC-11971452, NSFC-12026251 and a research fund from USTC.

\section{Preliminaries}
For any Ricci shrinker $(M^n,g,f)$, the scalar curvature $R\ge 0$ from \cite[Corollary $2.5$]{CBL07} and $R>0$ unless $(M^n,g)$ is isometric to the Gaussian soliton $(\R^n,g_E)$, by the strong maximum principle.

With the normalization \eqref{E101}, the entropy is defined as
\begin{align} \label{eq:mu}
\boldsymbol{\mu}=\boldsymbol{\mu}(g)\coloneqq \log \int\frac{e^{-f}}{(4\pi)^{n/2}}\, dV.
\end{align}

Notice that $e^{\boldsymbol{\mu}}$ is uniformly comparable to the volume of the unit ball $B(p,1)$ (cf. \cite[Lemma $2.5$]{LLW21}). It was proved in \cite[Theorem $1$]{LW20} that $\boldsymbol{\mu}$ is the optimal log-Sobolev constant for all scales. Following \cite{LWs1}, we have the following definition. 

\begin{defn}\label{dfn:201}
Let $\mathcal M(A)$ be the family of Ricci shrinkers $(M^n,g,f)$ satisfying
\begin{align}\label{eq:mubound}
\boldsymbol{\mu}(g) \ge -A.
\end{align}
\end{defn}

Recall that any Ricci shrinker $(M^n,g,f)$ can be considered a self-similar solution to the Ricci flow. Let ${\psi^t}: M \to M$ be a family of diffeomorphisms generated by $\dfrac{1}{1-t}\nabla f$ and $\psi^{0}=\text{id}$.
In other words, we have
\begin{align} 
\frac{\partial}{\partial t} {\psi^t}(x)=\frac{1}{1-t}\nabla f\left({\psi^t}(x)\right). \label{E201a}
\end{align}
It is well known that the rescaled pull-back metric $g(t)\coloneqq (1-t) (\psi^t)^*g$ satisfies the Ricci flow equation for any $-\infty <t<1$,
\begin{align} 
\partial_t g=-2 Rc_{g(t)} \quad \text{and} \quad g(0)=g. \label{E201ax}
\end{align}
Sometimes we encounter Ricci flow obtained from the above Ricci flow through time-shifting and rescaling.
We emphasize whether there exist extra time-shifting and rescaling by the following definition. 

\begin{defn}
For any Ricci shrinker, the Ricci flow defined in \eqref{E201ax} is called the \textbf{associated Ricci flow}. 
Any Ricci flow obtained from the associated Ricci flow via time-shifting and rescaling is called the \textbf{Ricci flow induced by a Ricci shrinker}. 
\label{dfn:RA19_2}
\end{defn}

Clearly, a Ricci flow associated to a Ricci shrinker must be a Ricci flow induced by a Ricci shrinker, but the reverse is generally not true. 
In this article, if not mentioned explicitly, the associated Ricci flow is the default one.

Next, we recall the function $F(x,t):=\bar \tau f(x,t)$, where $\bar \tau:=1-t$ and $f(x,t):=(\psi^t)^*f$, satisfies the following identities (see \cite[Section 2]{LW20} for proofs):
\begin{align} 
&\partial_t f=|\na f|^2, \label{E202xa}\\
&\partial_tF=-\bar \tau R,
\label{E202a} \\
&\bar \tau R+\Delta F=\frac{n}{2},
\label{E202b}\\
&\bar \tau^2R+|\nabla F|^2=F,
\label{E202c} \\
& \square F=-\frac{n}{2}. \label{E202d}
\end{align}
Here, we define $\square:=\partial_t-\Delta_t$ and have dropped the subscript $g(t)$ or $t$ if there is no confusion. Based on these identities, we have the following estimates of $F$.

\begin{lem}[Lemma 1 of \cite{LW20}]
\label{L201}
There exists a point $p \in M$ where $F$ attains its infimum and $F$ satisfies the quadratic growth estimate
\begin{align}
\frac{1}{4}\left(d_t(x,p)-5n\bar \tau-4 \right)^2_+ \le F(x,t) \le \frac{1}{4} \left(d_t(x,p)+\sqrt{2n\bar \tau} \right)^2
\label{E203}
\end{align}
for all $x\in M$ and $t < 1$, where $a_+ :=\max\{0,a\}$.
\end{lem}

Thanks to \eqref{E203}, $F(x,t)$ grows like $d^2_t(x,p)/4$ and hence one can obtain a family of cutoff functions by composing $F$ with a cutoff function on $\R$. More precisely, we fix a function $\eta \in C^{\infty}([0,\infty))$ such that $0\le \eta \le 1$, $\eta=1$ on $[0,1]$ and $\eta=0$ on $[2,\infty)$. Furthermore, $-C \le \eta'/\eta^{\frac{1}{2}}\le 0$ and $|\eta''| \le C$ for a universal constant $C>0$.
For each $r \ge 1$, we define 
\begin{align}
\phi^r \coloneqq \eta\left(\frac{F}{r} \right). 
\label{E204}
\end{align}
Then $\phi^r$ is a smooth function on $M \times (-\infty,1)$. The following estimates of $\phi^{r}$ are proved in \cite[Lemma 3]{LW20}:
\begin{align}
(\phi^r)^{-1} |\nabla \phi^r|^2 & \le Cr^{-1}, \label{E205a}\\
|\phi^r_t|&\le C{\bar \tau}^{-1}, \label{E205b}\\
|\Delta \phi^r|&\le C(\bar \tau^{-1}+r^{-1}), \label{E205c} \\
|\square \phi^r|& \le Cr^{-1}, \label{E205d}
\end{align}
where the constant $C$ depends only on the dimension $n$.

For later applications, we recall the following volume estimate proved in \cite[Lemma 2]{LW20}.
\begin{lem}
\label{L202}
There exists a constant $C=C(n)>0$ such that for any Ricci shrinker $(M^n,g,f)$ with $p \in M$ a minimum point of $f$,
\begin{align*}
|B_t(p,r)|_t \le Cr^n.
\end{align*}
\end{lem}

Next, we recall the following version of the maximum principle on Ricci shrinkers, which is proved in \cite[Theorem 6]{LW20} and will be frequently used.

\begin{thm}[Maximum principle on Ricci shrinkers I]
\label{T201}
Let $(M,g(t))_{t<1}$ be the Ricci flow associated with a Ricci shrinker. Given any closed interval $[a,b] \subset (-\infty,1)$ and a function $u$ which satisfies $\square u \le 0$ on $M \times [a,b]$, suppose that 
\begin{align} 
\int_a^b \int_M u^2_{+}(x,t)e^{-2f(x,t)}\,dV_t(x)\,dt < \infty.
\label{E206}
\end{align}
If $u(\cdot,a) \le c$, then $u(\cdot, b) \le c$.
\end{thm} 

We also need the following version of the maximum principle, which is proved in \cite[Theorem 12.14]{CCG2} for Ricci flows with bounded curvature. Notice that if $X \equiv 0$, Theorem \ref{T202} follows from Theorem \ref{T201}.

\begin{thm}[Maximum principle on Ricci shrinkers II]
\label{T202}
Let $(M,g(t))_{t<1}$ be the Ricci flow associated with a Ricci shrinker. Given any closed interval $[a,b] \subset (-\infty,1)$ and a function $u$ which satisfies 
\begin{equation*}
Lu:=\square u-\left\langle\nabla u, X(t)\right\rangle\leq0
\end{equation*}
on $M \times [a,b]$, suppose that $X(t)$ is a bounded vector field on $M \times [a,b]$ and
\begin{align} \label{E207a}
u(x,t) \le K e^{kf(x,t)}
\end{align}
on $M \times [a,b]$ for some constants $K>0$ and $k<1$. If $u(\cdot,a) \le c$, then $u(\cdot, b) \le c$.
\end{thm} 

\begin{proof}
We first construct a barrier function
\begin{align*} 
\phi(x,t):=Ke^{B(t-a)+(1-\ep)f(x,t)},
\end{align*}
where $1-\ep>k$ and $B$ is a constant determined later.

\textbf{Claim:} There exists a constant $B>0$ such that
\begin{align} \label{E207b}
L\phi \geq \phi.
\end{align}

\emph{Proof of Claim}: By direct computations, we have
\begin{align*}
L\phi&=\phi \lc B+(1-\ep)f_t-(1-\ep)^2|\na f|^2-(1-\ep) \Delta f-(1-\ep) \la \na f,X \ra \rc \\
&= \phi \lc B+\ep(1-\ep)|\na f|^2-\frac{n(1-\ep)}{2\bar \tau}+(1-\ep)R-(1-\ep) \la \na f,X \ra \rc \\
&\ge \phi \lc B+\ep(1-\ep)|\na f|^2-C_1|\na f|-\frac{n(1-\ep)}{2(1-b)} \rc,
\end{align*}
where we have used \eqref{E202a}, \eqref{E202b} and the assumption that $|X| \le C_1$. Therefore, \eqref{E207b} holds if we choose
$$
B=\frac{C_1^2}{4\ep(1-\ep)}+\frac{n(1-\ep)}{2(1-b)}+1.
$$

Now, we assume $c=0$ by considering $u-c$ instead of $u$. To complete the proof, we only need to verify that for any $\delta>0$, $u \leq \delta\phi$ on $M\times [a,b]$. Otherwise, then there exists $(x',t')\in M\times [a,b]$ such that $\left(u-\delta\phi\right)(x',t')>0$. Due to the estimate \eqref{E207a} and our definition of $\phi$, we know that $\left(u-\delta\phi\right)(x,t)\longrightarrow-\infty$ as $d_t(x,p)\longrightarrow+\infty$ uniformly in $t$, i.e., $u-\delta\phi<0$ for $d_t(x,p)$ large enough independent of $t$. Moreover, $\left(u-\delta\phi\right)(x,a)<0$ for all $x\in M$. Consequently, there exists $(x'',t'') \in M\times (a,t')$ such that $\left(u-\delta\phi\right)(x,t)\leq0$ for all $(x,t)\in M\times[a,t'']$ and $\left(u-\delta\phi\right)(x'',t'')=0$. At $(x'',t'')$, we compute
\begin{align*}
0\leq L\left(u-\delta\phi\right) \le -\delta \phi <0,
\end{align*}
which is a contradiction. In sum, our proof is complete.
\end{proof}

\section{Variance, $H$-center and Nash entropy}

Let $(M^n,g(t))_{t<1}$ be the Ricci flow associated with a Ricci shrinker. It is proved in \cite[Theorem 7]{LW20} that there exists a positive heat kernel function $H(x,t,y,s)$ for $x,y \in M$ and $s<t<1$. More precisely,
\begin{align*}
\square H(\cdot,\cdot,y,s)=0, \quad \lim_{t \searrow s}H(\cdot,t,y,s)=\delta_y 
\end{align*}
and
\begin{align*}
\square^* H(x,t,\cdot,\cdot)=0, \quad \lim_{s \nearrow t}H(x,t,\cdot,s)=\delta_x, 
\end{align*}
where $\square:=\partial_t-\Delta$ and $\square^*:=-\partial_t-\Delta+R$. Furthermore, the heat kernel $H$ satisfies the semigroup property
\begin{align}
H(x,t,y,s)=&\int_{M}H(x,t,z,\rho)H(z,\rho,y,s)\, dV_{\rho}(z), \quad \forall \; x, y \in M, \; \rho \in (s,t) \subset (-\infty, 1), \label{E301} 
\end{align}
and the following integral relationships
\begin{align} 
&\int_{M}H(x,t,y,s)\,dV_t(x) \le 1, \label{E301a}\\
&\int_{M}H(x,t,y,s)\,dV_s(y) =1. \label{E301b}
\end{align}

For any $(x,t) \in M \times (-\infty,1)$, we define the conjugate heat kernel measure $v_{x,t;s}$ by $dv_{x,t;s}(y)=K(x,t,y,s)\,dV_s(y)$. It follows immediately from \eqref{E301b} that $v_{x,t;s}$ is a probability measure on $M$. In particular, $v_{x,t;t}=\delta_x$. 

With the help of the heat kernel, one can solve the (conjugate) heat solution from the given initial condition. More precisely, it follows from \cite[Lemma 5, Lemma 6]{LW20} that

\begin{thm}\label{thm301}
Suppose $[a,b] \subset (-\infty, 1)$ and $u_{a}$ is a bounded function on the time slice $(M, g(a))$. Then
\begin{align}
u(x,t) \coloneqq \int_{M}H(x,t,y,a)u_a(y)\,dV_{a}(y), \quad \forall \; t \in [a,b] \label{E302a}
\end{align}
is the unique bounded heat solution with the initial value $u_a$. Similarly, suppose $w_{b}$ is an integrable function on the time slice $(M, g(b))$. Then
\begin{align}
w(y,s) \coloneqq \int_{M}H(x,b,y,s)w_b(x) \, dV_{b}(x) \label{E302b}
\end{align}
is the unique conjugate heat solution with initial value $w_{b}$ such that
\begin{align}
\sup_{s \in [a,b]} \int |w| \,dV_s< \infty. \label{E302c}
\end{align}
\end{thm}

Next, we recall the following gradient estimate, which slightly strengthens \cite[Corollary 1]{LW20}.

\begin{lem} \label{lem:301}
Let $u$ be a bounded heat solution on $M \times [a,b]$ such that $\sup_{M} |\nabla u(\cdot, a)| <\infty$. Then
\begin{enumerate}[label=\textnormal{(\roman{*})}]
\item We have 
\begin{align}
\sup_{M} |\nabla u(\cdot, b)| \leq \sup_{M} |\nabla u(\cdot, a)|. \label{E303a}
\end{align}

\item Assume $w$ is a nonnegative conjugate heat solution on $M \times [a,b]$ such that 
\begin{align} \label{E304aa}
\sup_{t \in [a,b]} \int_M w\,dV_t <\infty,
\end{align}
then we have
\begin{align}
2\int_a^b \int_M |\emph{Hess } u|^2 w\,dV_t dt =\left. \int_{M} |\na u|^2 w \,dV \right|_b^a<\infty. \label{E304a}
\end{align}
\end{enumerate}
\end{lem}

\begin{proof}
(i) From $\square u=0$ and direct computation, we have
\begin{align}
\square |\nabla u|^2 =-2|\text{Hess }u|^2 \le 0. \label{E303c}
\end{align}
Therefore, \eqref{E303a} follows from Theorem \ref{T201} provided that
\begin{align}
\int_a^b \int_M |\na u|^2 e^{-2f} \,dV_tdt<\infty. \label{E303cc}
\end{align}
Now, we fix $r \gg1$ and multiply both sides of $\square u=0$ by $u (\phi^r)^2e^{-2f}$. By integrating on $M \times [a,b]$, we obtain 
\begin{align*} 
&\left. \frac{1}{2}\int_{M} u^2(\phi^r)^2 e^{-2f} \,dV \right|_a^b \notag \\
&-\int_a^b \int_{M} u^2 \phi^r\phi^r_t e^{-2f} \,dV_tdt+\int_a^b \int_{M} u^2 (\phi^r)^2 f_t e^{-2f} \,dV_tdt+\frac{1}{2}\int_a^b \int_{M} u^2 (\phi^r)^2Re^{-2f} \,dV_tdt \notag \\
=&\int_a^b \int_{M} \left\{-|\na(u \phi^r)|^2+|\na \phi^r|^2u^2+\la \na u^2,\na f \ra(\phi^r)^2 \right\} e^{-2f} dV_tdt \\
= & \int_a^b \int_{M} \left\{-|\na(u \phi^r)|^2+|\na \phi^r|^2u^2+ (2|\na f|^2-\Delta f) u^2(\phi^r)^2-2 u^2 \phi^r \la \na \phi^r,\na f \ra \right\} e^{-2f} dV_tdt.
\end{align*}

Since $R \ge 0$ and $f_t=|\na f|^2$ by \eqref{E202xa}, we have
\begin{align*} 
&\int_a^b \int_{M} |\na(u \phi^r)|^2 e^{-2f} dV_tdt+\left. \frac{1}{2}\int_{M} u^2(\phi^r)^2 e^{-2f} \,dV \right|_a^b \\
\le &\int_a^b \int_{M} \left\{|\na \phi^r|^2u^2+ (|\na f|^2-\Delta f) u^2(\phi^r)^2+ u^2 \phi^r(\phi^r_t-2 \la \na \phi^r,\na f \ra) \right\} e^{-2f} dV_tdt \\
= &\int_a^b \int_{M} \left\{|\na \phi^r|^2u^2+ \frac{1}{1-t} \lc f-\frac{n}{2} \rc u^2(\phi^r)^2+ u^2 \phi^r(\phi^r_t-2 \la \na \phi^r,\na f \ra) \right\} e^{-2f} dV_tdt,
\end{align*}
where we have used the identity $\Delta f-|\na f|^2=\bar \tau(f-n/2)$ from \eqref{E202b} and \eqref{E202c}.

Since $u$ is bounded on $M \times [a,b]$ and $|\na f|^2 \le f/(1-b)$, it follows from \eqref{E205a}, \eqref{E205b}, Lemma \ref{L201} and Lemma \ref{L202} that by letting $r \to +\infty$,
\begin{align*} 
&\int_a^b \int_{M} |\na u|^2 e^{-2f} dV_tdt \le \left. \frac{1}{2}\int_{M} u^2 e^{-2f} \,dV \right|_b^a+\int_a^b \int_{M} \frac{1}{1-t} \lc f-\frac{n}{2} \rc u^2e^{-2f} dV_tdt <\infty
\end{align*}
and hence \eqref{E303cc} holds.

(ii) Fix $r \gg 1$ and $\ep \ll 1$. We calculate 
\begin{align} 
\partial_t\int_{M} |\na u|^2 \phi^r w\,dV&=\int_{M} \left\{ \square (|\na u|^2\phi^{r}) w- (|\na u|^2 \phi^{r}) \square^{*} w \right\} dV \notag\\
&=\int_{M} \left\{ |\na u|^2 \square \phi^{r} + \phi^{r} \square |\na u|^2 - 2\langle \nabla |\na u|^2, \nabla \phi^{r}\rangle \right\} w \,dV \notag\\
&\le \int_{M} \left\{ |\na u|^2 \square \phi^{r} -2|\text{Hess }u|^2\phi^r+4|\text{Hess }u||\nabla u||\nabla \phi^{r}| \right\} w \,dV \notag \\
&\le \int_{M} \left\{ |\na u|^2 |\square \phi^{r}| -(2-\ep^2)|\text{Hess }u|^2\phi^r+4\ep^{-2}|\nabla u|^2|\nabla \phi^{r}|^2 (\phi^r)^{-1} \right\} w \,dV. \label{E304b}
\end{align}
Since $|\na u|$ is uniformly bounded by \eqref{E303a}, it follows from \eqref{E304b}, \eqref{E205a} and \eqref{E205d} that
\begin{align*}
2\int_a^b \int_M |\text{Hess } u|^2 w\,dV_t dt \le \left. \int_{M} |\na u|^2 w \,dV \right|_b^a
\end{align*}
if we let $r \to +\infty$ and $\ep \to 0$. The other inequality can be proved similarly and hence \eqref{E304a} holds.
\end{proof}

Next, we prove

\begin{prop} \label{prop:iden1}
For any $[a,b]\subset (-\infty,1)$, suppose $u$ and $w$ are two smooth functions on $M \times [a,b]$ satisfying $\square u=\square^* w=0$.
Then, the identity
\begin{align} 
\int_M uw \,dV_a=\int_M uw \,dV_b \label{E304xxb}
\end{align}
holds under one of the following additional assumptions:
\begin{enumerate}[label=\textnormal{(\roman{*})}]
\item $\displaystyle \sup_{t \in [a,b]} \int_M |wu|\,dV_t+\int_a^b \int_M |w||\na u| \,dV_tdt<\infty$.

\item $\displaystyle \sup_{t \in [a,b]} \int_M |wu|\,dV_t+\int_a^b \int_M |u||\na w| \,dV_tdt<\infty$.
\end{enumerate}
\end{prop}

\begin{proof}
(i) We take $r \gg1$ and calculate 
\begin{align*} 
\partial_t\int_{M} wu\phi^r \,dV&=\int_{M} \left\{ w \square (u\phi^{r}) - (u \phi^{r}) \square^{*} w \right\} dV \notag\\
&=\int_{M} w\left\{ u \square \phi^{r} + \phi^{r} \square u - 2\langle \nabla u, \nabla \phi^{r}\rangle \right\} dV \notag\\
&=\int_{M} w\left\{ u \square \phi^{r} - 2\langle \nabla u, \nabla \phi^{r}\rangle \right\}dV.
\end{align*}
By using \eqref{E205a} and \eqref{E205d}, we conclude
\begin{align*}
\left| \left. \int_{M} wu \phi^{r} dV \right|_{a}^{b} \right| \leq C(r^{-1}+r^{-\frac{1}{2}}) \int_a^b \int_M |w|(|u|+|\na u|) \,dV_tdt.
\end{align*}
By taking $r \to \infty$, we arrive at $\eqref{E304xxb}$.

(ii) Similarly, we have
\begin{align*} 
\partial_t\int_{M} u w\phi^r\,dV&=\int_{M} \left\{ (\square u) w\phi^r- u \square^{*} (w\phi^r) \right\} dV \notag\\
&=\int_{M} u\left\{-(\square^* w)\phi^r+w(\Delta \phi^r+\phi^r_t)+2\la \na w,\na \phi^r \ra \right\} \,dV \notag\\
&=\int_{M} u\left\{w(\Delta \phi^r+\phi^r_t)+2\la \na w,\na \phi^r \ra \right\} \,dV. 
\end{align*}
Therefore, by \eqref{E205a}, \eqref{E205b} and \eqref{E205c}, we have
\begin{align*}
\left| \left. \int_{M} wu \phi^{r} dV \right|_{a}^{b} \right| \leq C(r^{-1}+r^{-\frac{1}{2}}+(1-a)^{-1}) \iint_{K_r} |u|(|w|+|\na w|) \,dV_tdt,
\end{align*}
where $K_r:= \{r \le F(x,r) \le 2r,\, a\le t\le b\}$. Consequently, by our assumption, \eqref{E304xxb} holds if $r \to \infty$.
\end{proof}

\begin{rem} \label{rem:iden2}
Suppose $\square u=\square^* w=0$.
\begin{enumerate}[label=\textnormal{(\roman{*})}]
\item If $\displaystyle \sup_M |\na u(\cdot,a)|+\sup_{M \times [a,b]}|u|+\sup_{t \in [a,b]} \int_{M} |w|\,dV_t<\infty$, then assumption (i) holds by \eqref{E303a}. If $\displaystyle \sup_{M \times [a,b]}|u|+\sup_{t \in [a,b]} \int_{M} |w|\,dV_t<\infty$ and $u$ is positive, then $|\na u| \le C/\sqrt{t-a}$ by \emph{\cite[Lemma 18]{LW20}}. Therefore, \eqref{E304xxb} also holds by taking the limit for $t \searrow a$.

\item If $\displaystyle \sup_{M \times [a,b]}|u|+\sup_{t \in [a,b]} \int_{M} |w|\,dV_t<\infty$ and $w(\cdot,b)$ is a nonnegative function with compact support, then assumption (ii) holds. Indeed, it follows from \emph{\cite[Lemma 9]{LW20}} that $\displaystyle \int_a^b \int_M \frac{|\na w|^2}{w} \,dV_tdt<\infty$ and hence
\begin{align*} 
\int_a^b \int_M |u||\na w| \,dV_tdt \le C(\sup_{M \times [a,b]}|u|) \lc \int_a^b \int_M \frac{|\na w|^2}{w} \,dV_tdt \rc^{\frac 1 2} \lc \int_a^b \int_M w \,dV_tdt \rc^{\frac 1 2}<\infty.
\end{align*}
\end{enumerate}
\end{rem}

For later applications, we prove the following estimate of the heat kernel.

\begin{lem} \label{lem:302}
For any $y \in M$ and $s<t<1$, we set $u(x,t):=H(x,t,y,s)$ and $\bar w(x,t):=(4\pi\bar \tau)^{-\frac n 2}e^{-f(x,t)}$. Then
\begin{align}\label{E304xa}
\int_M u(x,t) \bar w(x,t) \,dV_t(x)=\bar w(y,s).
\end{align}
\end{lem}

\begin{proof}
It is clear from the definition of $\bar w$ that $\square^* \bar w=0$; see \cite[Equation (28)]{LW20}. Moreover, for any $[a,b] \subset (s,t]$, the assumption (ii) of Proposition \ref{prop:iden1} holds since
\begin{align*}
\int_a^b \int_M u |\na \bar w| \,dV_tdt \le& \int_a^b \int_M u \bar w (1+|\na f|) \,dV_tdt \\
\le& (4\pi(1-b))^{-\frac n 2}\int_a^b \int_M u (1+(1-b)^{-\frac 1 2} f^{\frac 1 2}) e^{-f} \,dV_t dt \\
\le& C \int_a^b\int u\,dV_t dt\le C(b-a)
\end{align*}
where we have used \eqref{E202c} and \eqref{E301a} and the constant $C$ depends only on $n$ and $b$.

By choosing $b=t$ and letting $a \searrow s$, the proof of Proposition \ref{prop:iden1} yields
\begin{align*} 
\int_{M} u \bar w \phi^r \,dV_t-\bar w(y,s)\phi^r(y,s)=o(r)
\end{align*}
where $o(r) \to 0$ as $r \to \infty$. Therefore, we immediately obtain \eqref{E304xa} by letting $r \to \infty$.
\end{proof}

Next, we recall the definition of $W_1$-Wasserstein distance.

\begin{defn} \label{def:W1}
Let $(X,d)$ be a complete metric space and $\mu_1,\mu_2$ two probability measures on $X$. Then the $W_1$-Wasserstein distance between $\mu_1$ and $\mu_2$ is defined by 
\begin{align*} 
d_{W_1}(\mu_1,\mu_2):=\sup_{f} \lc \int f\,d\mu_1-\int f \,d\mu_2 \rc
\end{align*}
where the supremum is taken for all bounded $1$-Lipschitz functions $f$. We also use $d^t_{W_1}$ to denote the $W_1$-distance with respect to $g(t)$.
\end{defn}

We prove the following monotonicity of the Wasserstein distance as \cite[Lemma 2.7]{Bam20a}.

\begin{prop} \label{prop:301}
Let $(M^n, g(t))_{t<1}$ be a Ricci flow associated with a Ricci shrinker. For $[a,b] \subset (-\infty,1)$, let $w_1,w_2 \in C^\infty(M \times [a,b])$ be two nonnegative conjugate heat solutions such that $\int_M w_i \, dV_t=1$ for any $t \in [a,b]$ and $i=1,2$. We define the probability measures with $d\mu_{i,t} = w_i(\cdot, t) \, dV_t$, $i=1,2$.
Then
\begin{align*}
d^t_{W_1}(\mu_{1,t}, \mu_{2,t})
\end{align*}
is increasing for $t \in [a,b]$. In particular, if $t_1 \le t_2<1$, then for any $x_1,x_2 \in M$ and $t \le t_1$,
\begin{align*}
d^t_{W_1}(v_{x_1,t_1;t},v_{x_2,t_2;t})
\end{align*}
is increasing and 
\begin{align*}
d^t_{W_1}(v_{x_1,t_1;t},v_{x_2,t_2;t}) \le d_{t_1}(x_1,x_2).
\end{align*}
\end{prop}

\begin{proof}
Let $t_1 \leq t_2$, $t_1, t_2 \in [a,b]$ and consider a bounded function $u_1 \in C^\infty(M)$ with $\sup_M |\nabla u_1(\cdot, t_1)| \le 1$. Suppose $u$ is the unique bounded heat solution on $M \times [t_1,t_2]$ starting from $u_1$. Then it follows from Lemma \ref{lem:301} (i) that 
\begin{align*}
\sup_M |\na u(\cdot, t)| \le 1
\end{align*}
for any $t \in [t_1,t_2]$. Clearly, we have
\begin{align*}
\int_M u \, d\mu_{1,t_1} - \int_M u \, d\mu_{2,t_1}
&= \int_M u(x,t_1) w_1(x,t_1) \, dV_{t_1}(x) - \int_M u(x,t_1) w_2(x,t_1) \, dV_{t_1}(x) \\
&= \int_M u(x,t_2) w_1(x,t_2) \, dV_{t_2}(x) - \int_M u(x,t_2) w_2(x,t_2) \, dV_{t_2}(x) \\
&= \int_M u \, d\mu_{1,t_2} - \int_M u \, d\mu_{2,t_2} \leq d^{t_2}_{W_1} (\mu_{1,t_2}, \mu_{2,t_2}). 
\end{align*}
Here, we have used \cite[Proposition 1]{LW20} for the second equality. By taking the supremum over all such $u_1$, one obtains
\[ d^{t_1}_{W_1} (\mu_{1,t_1}, \mu_{2,t_1}) \leq d^{t_2}_{W_1} (\mu_{1,t_2}, \mu_{2,t_2}).\]
\end{proof}

Next, we recall the following definition from \cite[Definition 3.1]{Bam20a}.

\begin{defn}[Variance] \label{Def_Variance}
The variance between two probability measures $\mu_1, \mu_2$ on a Riemannian manifold $(M,g)$ is defined as
\[ \emph{Var} (\mu_1, \mu_2) := \int_M \int_M d^2 (x_1, x_2) d\mu_1 (x_1) d\mu_2 (x_2). \]
In the case $\mu_1 = \mu_2 = \mu$, we write
\[ \emph{Var} (\mu) = \emph{Var} (\mu, \mu) = \int_M \int_M d^2 (x_1, x_2) d\mu (x_1) d\mu (x_2). \]
We also define $\emph{Var}_t$ as the variance with respect to the metric $g(t)$.
\end{defn}

For some basic properties of the variance, we refer the readers to~\cite[Lemma 3.2]{Bam20a}. Next, we prove the following results which originate from~\cite[Corollary 3.7, Corollary 3.8]{Bam20a}. Before that, we first prove the following maximum principle on the product manifold (cf.~\cite{Andrews}~\cite{Brendle} for related survey).  

\begin{thm}[Maximum principle on the product]
\label{T302}
Let $(M^n, g(t))_{t<1}$ be a Ricci flow associated with a Ricci shrinker. Given any closed interval $[a,b] \subset (-\infty,1)$ and a function $u$ on $M \times M \times [a,b]$ such that
\begin{align} 
(\partial_t -\Delta_x-\Delta_y) u(x,y,t) \le 0.
\label{E305xa}
\end{align}
Suppose that
\begin{align} 
\int_a^b \int_{M \times M} u^2_{+}(x,y,t)e^{-2f(x,t)-2f(y,t)}\,dV_t(x)dV_t(y)\,dt < \infty.
\label{E305xb}
\end{align}
If $u(\cdot,a) \le c$, then $u(\cdot, b) \le c$.
\end{thm} 
\begin{proof}
The proof follows almost verbatim from \cite[Theorem 6]{LW20}, except that we multiply \eqref{E305xa} by $u_+(x,y,t) (\phi^r(x)\phi^r(y))^2e^{-2f(x,t)-2f(y,t)}$ and do the integration. 
Since no other new ingredient is needed, we omit the details here.
\end{proof}

\begin{prop} \label{prop:302}
Under the same assumptions as in Proposition \ref{prop:301}, if we further assume $w_1(\cdot, b)$ and $w_2(\cdot, b)$ have compact supports, then
\begin{align*}
\emph{Var}_t(\mu_{1,t}, \mu_{2,t})+H_nt
\end{align*}
is increasing for $t \in [a,b]$, where $H_n:=(n-1) \pi^2/2+4$. Moreover, for any $x_1,x_2 \in M$,
\begin{align*}
\emph{\text{Var}}_t(v_{x_1,b;t},v_{x_2,b;t})+H_nt
\end{align*}
is increasing for $t \le b$. In particular,
\begin{align*}
\emph{\text{Var}}_t(v_{x_1,b;t},v_{x_2,b;t}) \le d_{b}^2(x_1,x_2)+H_n(b-t) \quad \text{and} \quad \emph{\text{Var}}_t(v_{x,b;t}) \le H_n(b-t).
\end{align*}
\end{prop}

\begin{proof}
For any $[c,d] \subset [a,b]$, we set $u \in C^0(M \times M \times [c,b]) \cap C^{\infty}(M \times M \times (c,b])$ be the solution to the following heat equation
\begin{align*}
(\partial_t-\Delta_x-\Delta_y) u=-H_n, \quad u(\cdot, c)=d_c^2.
\end{align*}
Indeed, by the existence of the heat kernel, one may define
\begin{align}\label{E305a}
u(x,y,t):= \int_M \int_M H(x,t,z,c) H(y,t,w,c) d_c^2(z,w)\,dV_c(z)dV_c(w)-H_n(t-c). 
\end{align}
We first show \eqref{E305a} is well-defined. In fact, it is clear that
\begin{align}\label{E305b}
&\int_M \int_M H(x,t,z,c) H(y,t,w,c) d_c^2(z,w)\,dV_c(z)dV_c(w) \notag \\
\le& 2 \lc \int_M H(x,t,z,c) d_c^2(z,p) \,dV_c(z)+ \int_M H(y,t,w,c) d_c^2(w,p) \,dV_c(w) \rc
\end{align}
and the convergence of the last two integrals follows from \cite[Corollary 5]{LW20}.

On the other hand, it follow from \cite[Theorem 3.5]{Bam20a} that
\begin{align}\label{E305c}
(\partial_t-\Delta_x-\Delta_y) d_t^2(x,y) \ge -H_n.
\end{align}

Combining \eqref{E305a} and \eqref{E305c}, we claim that $u(x,y,t) \le d_t^2(x,y)$ for any $t \in [c,b]$. Indeed, this follows from the maximum principle Theorem \ref{T302} as long as the condition \eqref{E305xb} is satisfied. First, notice that
\begin{align}
&\int_c^b \int_{M \times M} d^4_t(x,y)e^{-2f(x,t)-2f(y,t)}\,dV_t(x)dV_t(y)\,dt \notag \\
\le& 8 \int_c^b \int_{M \times M} \lc d_t^4(x,p)+d_t^4(y,p) \rc e^{-2f(x,t)-2f(y,t)}\,dV_t(x)dV_t(y)\,dt. \label{E305d}
\end{align}

From Lemma \ref{L201} and Lemma \ref{L202}, it is clear that \eqref{E305d} is bounded. In addition, it follows from \eqref{E305b} that
\begin{align}
& \int_c^b \int_{M \times M} \lc u(x,y,t)+H_n(t-c) \rc^2 e^{-2f(x,t)-2f(y,t)}\,dV_t(x)dV_t(y)\,dt \notag\\
\le & 8 \int_c^b \int_{M \times M} \lc \int_M H(x,t,z,c) d_c^2(z,p) \,dV_c(z) \rc^2 e^{-2f(x,t)-2f(y,t)}\,dV_t(x)dV_t(y)\,dt \notag\\
&+8 \int_c^b \int_{M \times M} \lc \int_M H(y,t,w,c) d_c^2(w,p) \,dV_c(w) \rc^2 e^{-2f(x,t)-2f(y,t)}\,dV_t(x)dV_t(y)\,dt \notag\\
\le & 8 \int_c^b \int_{M \times M} \int_M H(x,t,z,c) d_c^4(z,p) e^{-2f(x,t)-2f(y,t)}\,dV_c(z)dV_t(x)dV_t(y)\,dt \notag \\
&+8 \int_c^b \int_{M \times M} \int_M H(y,t,w,c) d_c^4(w,p) e^{-2f(x,t)-2f(y,t)}\,dV_c(w)dV_t(x)dV_t(y)\,dt, \label{E305e}
\end{align}
where we have used Cauchy-Schwarz inequality for the last inequality. From Lemma \ref{lem:302}, we obtain
\begin{align*}
&\int_c^b \int_{M \times M} \int_M H(x,t,z,c) d_c^4(z,p) e^{-2f(x,t)-2f(y,t)}\,dV_c(z)dV_t(x)dV_t(y)\,dt \\
\le & \int_c^b \int_{M} \int_M \int_M H(x,t,z,c) d_c^4(z,p) e^{-f(x,t)-2f(y,t)}\,dV_t(x) dV_c(z)dV_t(y)\,dt \\
\le & \int_c^b \int_{M} \int_M \lc\frac{1-t}{1-c} \rc^{\frac n 2}d_c^4(z,p) e^{-f(z,c)-2f(y,t)}\, dV_c(z)dV_t(y)\,dt\\
\le & \int_c^b \int_{M} \int_M d_c^4(z,p) e^{-f(z,c)-2f(y,t)}\, dV_c(z)dV_t(y)\,dt <\infty
\end{align*}
by Lemma \ref{L201} and Lemma \ref{L202}. Similarly, the second term in \eqref{E305e} is also bounded. Therefore, we have proved that $u(x,y,t) \le d_t^2(x,y)$ for any $t \in [c,b]$.

By our assumption, $w_1(\cdot, b)$ and $w_2(\cdot, b)$ have compact supports, then it follows \cite[Lemma 8, Lemma 9]{LW20} that
\begin{align} \label{E306a}
w_i(x,t) \le C \bar w(x,t)
\end{align}
for any $c \le t \le b$ and 
\begin{align}\label{E306b}
\int_c^b \int_M \frac{|\na w_i|^2}{w_i} \,dV_t dt \le C
\end{align}
for some constant $C>0$.

Next, we set $w_1=w_1(x,t)$, $w_2=w_2(y,t)$, $\phi^r_x=\phi^r(x)$ and $\phi^r_y=\phi^r(y)$, then we compute
\begin{align}
& \partial_t \int_M \int_M u w_1 w_2 \phi^r_x \phi^r_y\,dV_t(x)dV_t(y) \notag \\
=& \int_M \int_M (\partial_t-\Delta_x-\Delta_y)(u\phi^r_x \phi^r_y) w_1 w_2\,dV_t(x)dV_t(y) \notag\\
=& \int_M \int_M \lc -H_n \phi^r_x \phi^r_y-u(\Delta_x \phi^r_x \phi^r_y+\Delta_y \phi^r_y \phi^r_x)\rc w_1 w_2\,dV_t(x)dV_t(y) \notag \\
&+2 \int_M \int_M (\Delta_x \phi^r_x+ \la \na \phi^r_x, \na w_1 \ra) u\phi^r_y w_2+(\Delta_y \phi^r_y+ \la \na \phi^r_y, \na w_2 \ra) u\phi^r_y w_1 \,dV_t(x)dV_t(y). \label{E306c}
\end{align}

From \eqref{E306b}, we have
\begin{align}
&\lc \int_c^b \int_M \int_M |\na \phi^r_x||\na w_1||u|\phi^r_y w_2 \,dV_t(x) dV_t(y) dt \rc^2 \notag\\
\le& \lc \int_c^b \int_M \int_M |\na \phi^r_x|^2 u^2 (\phi^r_y)^2 w_1 w_2 \,dV_t(x) dV_t(y) dt \rc \lc \int_c^b \int_M \frac{|\na w_1|^2}{w_1} \,dV_t dt \rc \notag\\
\le& C \int_c^b \int_M \int_M |\na \phi^r_x|^2 u^2 w_1 w_2 \,dV_t(x) dV_t(y) dt. \label{E306d}
\end{align}
Similarly, we have
\begin{align}
&\lc \int_c^b \int_M \int_M |\na \phi^r_y||\na w_2||u|\phi^r_x w_1 \,dV_t(x) dV_t(y) dt \rc^2 \notag\\
\le& C \int_c^b \int_M \int_M |\na \phi^r_y|^2 u^2 w_1 w_2 \,dV_t(x) dV_t(y) dt. \label{E306e}
\end{align}

Combining \eqref{E306a}, \eqref{E306c}, \eqref{E306d}, \eqref{E306e} and the fact that $-H_n(t-c)\le u \le d_t^2(x,y)$, we conclude by letting $r \to \infty$ that
\begin{align} \label{E306f}
\int_M \int_M u w_1 w_2 \,dV_d(x)dV_d(y)-\int_M \int_M u w_1 w_2 \,dV_c(x)dV_c(y)=-H_n(d-c).
\end{align}
Since $u \le d_t^2(x,y)$, it follows from \eqref{E306f} and the definition of the variance that
\begin{align*}
\text{Var}_d(\mu_{1,d}, \mu_{2,d})+H_nd \ge \text{Var}_c(\mu_{1,c}, \mu_{2,c})+H_nc.
\end{align*}

Now, we assume $w_i=H(x_i,b,\cdot,\cdot)$ for $i=1,2$. Then it follows from \cite[Lemma 23]{LW20} that 
\begin{align}\label{E307a}
\int_a^{b-\ep} \int_M \frac{|\na w_i|^2}{w_i} \,dV_t dt \le C \log \ep^{-1}.
\end{align}
Therefore, one can use the same arguments as above, thanks to \eqref{E307a} and \cite[Corollary 5]{LW20}, to conclude that \eqref{E306f} still holds if $[c,d] \subset [a,b-\ep]$. Since $\ep$ is arbitrary, we immediately show that
\begin{align*}
\text{Var}_t(v_{x_1,b;t},v_{x_2,b;t})+H_n t
\end{align*}
is increasing for any $t \le b$.
\end{proof}

Next, we recall the definition of $H$-center, where the conjugate heat kernel measure is concentrated.

\begin{defn}[$H$-center] \label{def:Hcenter}
Given a constant $H>0$, a point $(z,t) \in M \times (-\infty,1)$ is called an $H$-center of $(x_0,t_0)\in M \times (-\infty,1)$ if $t \le t_0$ and
\begin{align*}
\emph{\text{Var}}_t(\delta_z,v_{x_0,t_0;t})\le H(t_0-t).
\end{align*}
In particular, we have
\begin{align} \label{eq:hdis}
d_{W_1}^t(\delta_z,v_{x_0,t_0;t})\le \sqrt{H(t_0-t)}.
\end{align}

\end{defn}

From Proposition \ref{prop:302}, the following result is immediate; see \cite[Proposition 3.12]{Bam20a}.

\begin{prop} \label{prop:303}
Let $(M^n,g(t))_{t<1}$ be the Ricci flow associated with a Ricci shrinker. Given $(x_0, t_0) \in M \times (-\infty,1)$ and $t \leq t_0$ there is at least one point $z \in M$ such that $(z,t)$ is an $H_n$-center of $(x_0, t_0)$ and for any two such points $z_1, z_2 \in M$ we have $d_t (z_1, z_2) \leq 2 \sqrt{H_n (t_0 -t)}$.
\end{prop}

The following result ensures that the conjugate heat kernel measure is concentrated around an $H$-center; see \cite[Proposition 3.13]{Bam20a}.

\begin{prop} \label{prop:304}
If $(z,t)$ is an $H$-center of $(x_0,t_0)$, then for any $L>0$,
\begin{align} \label{E308a}
v_{x_0,t_0;t} \lc B_t(z,\sqrt{LH (t_0-t)}) \rc \ge 1-\frac{1}{L}.
\end{align}
\end{prop}

Combining the above Proposition with~\cite[Theorem 14]{LW20}, we obtain the following integral bound for the conjugate heat kernel; see also \cite[Theorem 3.14]{Bam20a}.

\begin{prop} \label{prop:305}
If $(z,t)$ is an $H_n$-center of $(x_0, t_0)$, then for all $r \geq 0$ and $\ep >0$ we have
\begin{align*}
v_{x_0,t_0;t}\big( M \setminus B_t(z, r) \big) \le C(n,\ep) \exp \bigg( - \frac{r^2}{(4+\ep)(t_0 - t)} \bigg).
\end{align*}
\end{prop}

\begin{proof}
We apply \cite[Theorem 14]{LW20} for $A=M \setminus B_t(z, r)$, $B=B_t(z,\sqrt{2H_n(t_0-t)})$ and $\sigma=\ep/8$ to obtain
\begin{align*}
v_{x_0,t_0;t}\big( M \setminus B_t(z, r) \big) \le& v^{-\frac{8}{\ep}}_{x_0,t_0;t}\big( B_t(z,\sqrt{2H_n(t_0-t)}) \big) \exp \bigg( {- \frac{\big(r - \sqrt{2H_n (t_0-t)} \big)_+^2}{(4+\ep/2)(t_0-t)} }\bigg) \\
\le& C(n,\ep) \exp \bigg( - \frac{r^2}{(4+\ep)(t_0 - t)} \bigg),
\end{align*}
where we have used \eqref{E308a} for $L=2$ and $H=H_n$.
\end{proof}

In order to obtain the estimates on the Nash entropy, we first generalize the improved gradient estimate \cite[Theorem 4.1]{Bam20a} to our setting. We define the following antiderivative of the $1$-dimensional heat kernel:
\begin{equation} \label{eq_erf}
\Phi(x) = \int_{-\infty}^x (4\pi)^{-1/2} e^{-t^2/4} \,dt.
\end{equation}
Notice that $\Phi_t(x):=\Phi ( t^{-1/2} x)$ is a solution to the $1$-dimensional heat equation with initial condition $\chi_{[0, \infty)}$.

\begin{thm} \label{thm:T303}
Let $(M^n,g(t))_{t<1}$ be the Ricci flow associated with a Ricci shrinker. Given $[a,b] \subset (-\infty,1)$ and a solution $u \in C^\infty (M \times [a,b])$ to the heat equation $\square u =0$ and a constant $T \geq 0$, suppose that $u$ only takes values in $(0,1)$ and $| \nabla (\Phi_T^{-1} ( u (\cdot , a) ))| \leq 1$ if $T > 0$. Then $| \nabla (\Phi^{-1}_{T+t - a} ( u(\cdot, t) ))| \leq 1$ for all $t \in [a, b]$.
\end{thm}

\begin{proof}
We may assume that $u$ takes values in $(\ep,1-\ep)$. Indeed, we can consider $(1-2\ep)u+\ep$ instead and let $\ep \searrow 0$. With the extra assumption, it follows from \cite[Lemma 18]{LW20} that
\begin{align} \label{E309a}
|\na u| \le \frac{C_1}{\sqrt{t-a}}
\end{align}
on $M \times (a,b]$. It is clear from the definition of $\Phi_t$ that $\sup_M |\na (\Phi_T^{-1} ( u (\cdot , a+\ep) ))| \to 0$ if $T \searrow 0$. Therefore, we only need to prove the case for $T>0$ and then let $T \searrow 0$ and $\ep \searrow 0$.

Now, we set $u(x,t)=\Phi_{T+t-a} \circ h(x,t)$. It follows from the definition of $\Phi_t$ that
\begin{align} \label{E309b}
|h| \le C_2
\end{align}
on $M \times [a,b]$. Moreover, since $|\na h(\cdot, a)| \le 1$, it follows from \eqref{E309b} and Lemma \ref{lem:301}(i) that
\begin{align} \label{E309c}
|\na h| \le C_3
\end{align}
on $M \times [a,b]$. By direct computation, see \cite[Theorem 4.1]{Bam20a} for details, we have
\begin{align} \label{E309d}
\square |\na h|^2= -2|\text{Hess }h|^2-\frac{1}{T+t-a} \la \na h^2, \na |\na h|^2 \ra+\frac{1}{2(T+t-a)}(1-|\na h|^2) |\na h|^2.
\end{align}
Therefore, if we set $v=(|\na h|^2-1)_+$, then it follows from \eqref{E309d} that
\begin{align*} 
\square v +\frac{1}{T+t-a} \la \na h^2, \na v \ra \le 0.
\end{align*}
Since $|\na h^2|$ and $v$ are uniformly bounded on $M \times [a,b]$ by \eqref{E309b} and \eqref{E309c}, it follows from Theorem \ref{T202} that $v \le 0$ on $M \times [a,b]$. In other words,  $|\na h| \le 1$ on $M \times [a,b]$. 
Thus the proof is complete.
\end{proof}

With the help of Theorem \ref{thm:T303}, one can follow verbatim as \cite[Proposition 4.2]{Bam20a} and \cite[Proposition 3.4]{MZ21} to obtain the following estimate.

\begin{thm}\label{thm:T304}
Let $(M^n,g(t))_{t<1}$ be the Ricci flow associated with a Ricci shrinker and $[s,t] \subset (-\infty,1)$. Then for any $x\in M$, $1\leq p<\infty$ and measurable subset $X\subset M$, we have
\begin{equation*}
(t-s)^{\frac{p}{2}}\int_{X}\left(\frac{\left|\nabla_xH(x,t,\cdot,s)\right|}{H(x,t,\cdot,s)}\right)^p\,dv \leq C(n,p)v\left(X\right)\left(-\log\left(\frac{v\left(X\right)}{2}\right)\right)^{\frac{p}{2}},
\end{equation*}
where $dv=H(x,t,\cdot,s)\,dV_s$ is the conjugate heat kernel measure. Moreover, for any $x\in M$ and $w\in T_xM$ with $|w|_{t}=1$, there holds that
\begin{equation} \label{eqn:int}
(t-s)\int_{M}\left(\frac{\partial_w H(x,t,\cdot,s)}{H(x,t,\cdot,s)}\right)^2\,dv \leq\frac{1}{2}.
\end{equation}
In particular, we have
\begin{equation} \label{eqn:int2}
(t-s)\int_{M}\left|\frac{\na_x H(x,t,\cdot,s)}{H(x,t,\cdot,s)}\right|^2\,dv \leq\frac{n}{2}.
\end{equation}
\end{thm}

Another application of Theorem \ref{thm:T303} is the following $L^p$-Poincar\'e inequality; see \cite[Theorem 11.1]{Bam20a}.

\begin{thm}[$L^p$-Poincar\'e inequality] \label{thm:poin2}
Let $(M^n, g(t))_{t<1}$ be a Ricci flow associated with a Ricci shrinker. Then for $p \ge 1$ and any $[s,t] \subset (-\infty,1)$ we have
\begin{align*} 
\int_M u^p \,dv_s \le C(p) (t-s)^{\frac{p}{2}} \int_M |\na u|^p\,dv_s,
\end{align*}
for any $u \in W^{1,p}(M,dv_s)$ with $\int_M u\, dv_s=0$. Here, $dv_s(y)=H(x,t,y,s)dV_s(y)$. One may choose $C(1)=\sqrt{\pi}$ and $C(2)=2$.
\end{thm}
\begin{proof}
The proof for $p \ne 2$ follows verbatim from \cite[Theorem 11.1]{Bam20a}. Only the last statement for $p=2$ needs to be proved. It follows from \cite[Theorem 13]{LW20} that the probability measure $dv_s$ satisfies the log-Sobolev inequality with the constant $\frac{1}{2(t-s)}$. It is a standard fact that the log-Sobolev condition implies the Poincar\'e inequality; see \cite[Theorem 22.17]{Vil08}.
\end{proof}

Next, we recall the definitions of the Nash entropy and $\WW$-entropy based at $(x_0,t_0)$.

\begin{defn}\label{def:nash}
Given a Ricci flow $(M^n,g(t))_{t<1}$ associated with a Ricci shrinker and a point $(x_0,t_0) \in M \times (-\infty,1)$, let
\begin{align*} 
dv=dv_{x_0,t_0;t}(x)=(4\pi \tau)^{-\frac n 2} e^{-b(x,t)}\,dV_t=H(x_0,t_0,x,t)\,dV_t
\end{align*}
where $\tau=t_0-t$. Then Perelman’s $\WW$-entropy and the Nash entropy based at $(x_0,t_0)$ are respectively defined as
\begin{align} 
\WW_{(x_0,t_0)}(\tau)&=\int_M \lc \tau(2\Delta b-|\na b|^2+R)+b-n \rc\,dv, \label{E310a} \\
\NN_{(x_0,t_0)}(\tau)&=\int_M b\,dv-\frac{n}{2}. \label{E310b}
\end{align}
\end{defn}

Now, we prove some basic properties of $\NN$ and $\WW$.

\begin{prop} \label{prop:nash1}
The following properties hold with Definition \ref{def:nash}.
\begin{enumerate}[label=\textnormal{(\alph{*})}]
\item $\WW_{(x_0,t_0)}(0)=0$ and for any $\tau_0>0$,
\begin{align} 
\WW_{(x_0,t_0)}(\tau_0)=-2\int_0^{\tau_0} \tau\int_M\left|Rc+\emph{Hess }b-\frac{g}{2\tau} \right|^2 \,dvd\tau. \label{E311a}
\end{align}
In particular, $\WW_{(x_0,t_0)}(\tau)$ is nonpositive and decreasing.

\item $\NN_{(x_0,t_0)}(0)=0$ and for any $\tau_0>0$,
\begin{align} 
\NN_{(x_0,t_0)}(\tau_0)&=\frac{1}{\tau_0} \int_0^{\tau_0} \WW_{(x_0,t_0)}(\tau)\,d\tau \ge \WW_{(x_0,t_0)}(\tau_0). \label{E311b}
\end{align}

\item For any $0<\tau_1 \le \tau_2$,
\begin{align} 
\NN_{(x_0,t_0)}(\tau_1)-\frac{n}{2}\log\lc \frac{\tau_2}{\tau_1} \rc \le \NN_{(x_0,t_0)}(\tau_2) \le \NN_{(x_0,t_0)}(\tau_1). \label{E311c}
\end{align}
\end{enumerate}
\end{prop}

\begin{proof}
Given $(x_0,t_0)$ and $\tau$, we first prove that $\NN_{(x_0,t_0)}(\tau)$ and $\WW_{(x_0,t_0)}(\tau)$ are well-defined. In the following, all constants $C_i>1$ depend on $(x_0,t_0)$, $\tau$ and the given Ricci shrinker. 

It follows from \cite[Theorem 19]{LW20} that for any $r \ge 1$,
\begin{align} 
\int_{d_t(x_0,x) \ge r\sqrt{\tau}} dv_t(x) \le C_1 e^{-\frac{r^2}{8}}. \label{E312a}
\end{align}
Therefore, there exists $C_2>1$ such that
\begin{align} 
\int_{d_t(p,x) \ge r} dv_t(x) \le C_2 e^{-\frac{r^2}{C_2}} \label{E312aa}
\end{align}
if $r \ge C_2$. In addition, it follows from \cite[Theorem 15, Formula (203)]{LW20} that
\begin{align} 
\boldsymbol{\mu} \le b(x,t) \le -3\boldsymbol{\mu}+\frac{d_{t_0}^2(x_0,x)}{3\tau}+\frac{4\tau}{3(1-t_0)^2}F(x,t_0). \label{E312ab}
\end{align}
From \eqref{E312ab} and Lemma \ref{L201}, there exists $C_3>1$ such that
\begin{align} 
-C_3 \le b(x,t) \le C_3(1+F(x,t_0)). \label{E312ac}
\end{align}
Since $F$ is decreasing with respect to $t$ by \eqref{E202a}, it follows from \eqref{E312ac} and Lemma \ref{L201} that
\begin{align*} 
b(x,t) \le C_3(1+F(x,t)) \le C_4(1+d_t^2(p,x)). 
\end{align*}
for some $C_4 \ge C_3$. Consequently, we obtain
\begin{align} 
|b(x,t)| \le C_4(1+d_t^2(p,x)). \label{E312b}
\end{align}
Combining \eqref{E312aa} and \eqref{E312b}, we can estimate
\begin{align} 
\int_M |b(x,t)|\,dv_t(x) \le& C_4+C_4 \int_M d_t^2(p,x)\,dv_t(x) \notag \\
=& C_4+C_4 \int_{d_t(p,x) \le C_2} d_t^2(p,x)\,dv_t(x)+C_4 \sum_{k=1}^{\infty} \int_{2^{k-1}C_2\le d_t(p,x) \le 2^kC_2} d_t^2(p,x)\,dv_t(x) \notag \\
\le & C_4+C_4C_2^2+C_4\sum_{k=1}^{\infty }(2^k C_2)^2 C_2e^{-2^{2k-2}C_2} <\infty. \label{E312bb}
\end{align}
Therefore, it follows from the definition \eqref{E310b} that $\NN_{(x_0,t_0)}(\tau)$ is finite. Now, the fact that $\WW_{(x_0,t_0)}(\tau)$ is well-defined follows from Perelman's differential Harnack inequality \cite[Theorem 21]{LW20}.

(a): The identity \eqref{E311a} follows from \cite[Remark 6]{LW20}. Notice that the integral in \eqref{E311a} is always finite by \cite[Lemma 30]{LW20}. In particular, $\WW_{(x_0,t_0)}(0)=\lim_{\tau \searrow 0}\WW_{(x_0,t_0)}(\tau)=0$.

(b): We fix $r \gg 1$ and compute
\begin{align} 
& \partial_{\tau} \lc \tau \int_M b \phi^r \,dv \rc-\frac{n}{2} \notag \\
=& \int_M b \phi^r \,dv-\tau \int_M \square(b\phi^r)\,dv-\frac{n}{2} \notag \\
=& \int_M \lc \tau(2\Delta b-|\na b|^2+R)\phi^r+b\phi^r+\tau b\square \phi^r-2\tau \la \na b, \na \phi^r \ra-\frac{n}{2}(1+\phi^r) \rc\,dv, \label{E312c}
\end{align}
where we have used the fact that $\square b=-2\Delta b+|\na b|^2-R+\dfrac{n}{2\tau}$. For $\tau_0>0$, we integrate \eqref{E312c} from $0$ to $\tau_0$ and obtain
\begin{align} 
& \tau_0 \lc \int_M b \phi^r \,dv -\frac{n}{2} \rc \notag \\
=& \int_0^{\tau_0} \int_M \lc \tau(2\Delta b-|\na b|^2+R)\phi^r+b\phi^r+\tau b\square \phi^r-2\tau \la \na b, \na \phi^r \ra-\frac{n}{2}(1+\phi^r) \rc\,dv d\tau, \label{E312d}
\end{align}
where we have used \eqref{E312a} and \eqref{E312b}. On the one hand, it follows from \eqref{E205d}, \eqref{E312a} and \eqref{E312b} that
\begin{align} \label{E312e}
\lim_{r \to \infty} \int_0^{\tau_0}\int_M \tau |b||\square \phi^r| \, dv d\tau=0.
\end{align}
On the other hand, we estimate
\begin{align} \label{E312f}
\int_0^{\tau_0}\int_M \tau |\na b||\na \phi^r| \, dv d\tau \le Cr^{-\frac{1}{2}} \tau_0^2 \lc \int_0^{\tau_0}\int_M \tau^2 |\na b|^2 \, dv d\tau \rc^2.
\end{align}
Since the last integral is finite by \cite[Lemma 25]{LW20}, it follows from \eqref{E312f} that
\begin{align} \label{E312g}
\lim_{r \to \infty} \int_0^{\tau_0}\int_M \tau |\na b||\na \phi^r| \, dv d\tau=0.
\end{align}
Combining \eqref{E312d}, \eqref{E312e} and \eqref{E312g}, if we let $r \to \infty$, then
\begin{align*}
\tau_0 \lc \int_M b \,dv -\frac{n}{2} \rc=\int_0^{\tau_0} \int_M \lc \tau(2\Delta b-|\na b|^2+R)+b-n \rc\,dv d\tau,
\end{align*}
which is exactly \eqref{E311b}. Notice that the last inequality in \eqref{E311b} follows from the fact that $\WW_{(x_0,t_0)}(\tau)$ is decreasing. Moreover, it follows from \eqref{E311b} and $\WW_{(x_0,t_0)}(0)=0$ that $\NN_{(x_0,t_0)}(0)=0$.

(c): The inequality \eqref{E311c} follows exactly the same as \cite[Proposition 5.2 (5.7)]{Bam20a} and we omit the proof.
\end{proof}

\begin{cor} \label{cor:301}
Under the same assumptions, we have
\begin{align} 
&\int_M (|\na b|^2+R)dv \le \frac{n}{2\tau}. \label{E313a} \\
&\int_M \lc b-\NN_{(x_0,t_0)}(\tau)-\frac{n}{2} \rc^2\,dv \le n. \label{E313b} 
\end{align}
\end{cor}

\begin{proof}
From the fact that $\NN_{(x_0,t_0)}(\tau) \ge \WW_{(x_0,t_0)}(\tau)$, we conclude that
\begin{align}\label{E314a} 
\lim_{r \to \infty}\int_M (2\Delta b-|\na b|^2)\phi^r+R\,dv \le \frac{n}{2\tau},
\end{align}
where we have used the differential Harnack inequality \cite[Theorem 21]{LW20}. From integration by parts, we have
\begin{align}\label{E314b} 
\int_M (2\Delta b-|\na b|^2)\phi^r\,dv =\int_M |\na b|^2\phi^r-2\la \na b, \na \phi^r \ra\,dv.
\end{align}
In addition, we can estimate
\begin{align} \label{E314c} 
2\int_M |\na b||\na \phi^r|\,dv \le \int_M |\na b||\na \phi^r|\,dv \le \int_M \ep|\na b|^2 \phi^r+\ep^{-1} \frac{|\na \phi^r|^2}{\phi^r}\,dv
\end{align}
Therefore, it follows from \eqref{E205a}, \eqref{E314a}, \eqref{E314b} and \eqref{E314c} that
\begin{align*}
\int_M (1-\ep)|\na b|^2+R\,dv \le \frac{n}{2\tau}.
\end{align*}
By letting $\ep \searrow 0$, we obtain \eqref{E313a}. 

Now, it follows from the Poincar\'e inequality Theorem \ref{thm:poin2} and \eqref{E313a} that
\begin{align*}
\int_M \lc b-\NN_{(x_0,t_0)}(\tau)-\frac{n}{2} \rc^2\,dv \le 2\tau \int |\na b|^2 \,dv \le n
\end{align*}
and \eqref{E313b} is proved.
\end{proof}

\begin{rem} \label{rem:ibp}
From the proof of \eqref{E313a}, $\WW$ can be rewritten as
\begin{align*} 
\WW_{(x_0,t_0)}(\tau)=\int_M \lc \tau(|\na b|^2+R)+b-n \rc\,dv,
\end{align*}
which agrees with the original definition of Perelman \emph{\cite[Formula (3.1)]{Pe1}}.
\end{rem}

\begin{cor} \label{cor:302}
Let $(M^n,g(t))_{t<1}$ be the Ricci flow associated with a Ricci shrinker $(M^n,g,f) \in \MM(A)$, then
\begin{align} \label{E315a}
0 \ge \NN_{(x_0,t_0)}(\tau) \ge \WW_{(x_0,t_0)}(\tau) \ge \boldsymbol{\mu} \ge -A
\end{align}
for any $(x_0,t_0) \in M \times (-\infty,1)$ and $\tau>0$. In particular, given a Ricci shrinker, the Nash entropy is always uniformly bounded.
\end{cor}

\begin{proof}
For fixed $(x_0,t_0)$ and $\tau>0$, it follows from \cite[Theorem 20]{LW20} that $b$ increases quadratically. Therefore, it is easy to see the function $u$, defined by $u^2=(4\pi \tau)^{-\frac n 2} e^{-b}$, belongs to $W_*^{1,2}(M)$ defined in \cite[(92)]{LW20}. From \cite[Theorem 1]{LW20}, we immediately conclude that
\begin{align*}
\WW_{(x_0,t_0)}(\tau) \ge \boldsymbol{\mu}(g(t_0-\tau),\tau) \ge \boldsymbol{\mu} \ge -A.
\end{align*}
\end{proof}

Following \cite{Bam20a}, we use the notation
\begin{align*}
\NN_s^*(x,t):=\NN_{(x,t)}(t-s).
\end{align*}
Similar to \cite[Theorem 5.9]{Bam20a}, we have

\begin{thm}\label{thm:T305}
Let $(M^n,g(t))_{t<1}$ be the Ricci flow associated with a Ricci shrinker. Then for any $s<t<1$, the following properties hold.
\begin{enumerate}[label=\textnormal{(\roman{*})}]
\item $\NN_s^*$ is a Lipschitz function with Lipschitz constant $\sqrt{\dfrac{n}{2(t-s)}}$.

\item In the distribution sense, we have
\begin{align} \label{E316a}
-\frac{n}{2(t-s)} \le \square \NN_s^* \le 0.
\end{align}
\end{enumerate}
\end{thm}

\begin{proof}
Without loss of generality, we assume $s=0$ and consider $t \in (0,1)$. We first define the following modified Nash entropy:
\begin{align} \label{E316b}
\NN^r=\NN^r(x,t):=\int b \phi^r \,dv-\frac{n}{2},
\end{align}
where, as before, $b=b_{(x,t)}(y,0)=-\frac{n}{2}\log(4\pi t)-\log H(x,t,y,0)$ and $dv=H(x,t,y,0)\,dV_0(y)$.

\textbf{Claim}: $\NN^r$ converges to $\NN_0^*$ in $C^0_{\text{loc}}$ on $M \times (0,1)$, as $r \to \infty$.

\emph{Proof of Claim}: Given a spacetime compact set $K \subset M \times (0,1)$, all constants $C_i>1$ below depends only $K$ and the Ricci shrinker.

Similar to \eqref{E312aa}, there exists $C_1>1$ such that
\begin{align} \label{E316c}
\int_{d_0(p,y) \ge r} dv(y) \le C_1 e^{-\frac{r^2}{C_1}}
\end{align}
for any $r \ge C_1$. From the same argument leading to \eqref{E312b}, we have
\begin{align} \label{E316d}
|b_{(x,t)}(y,0)| \le C_2(1+d_0^2(p,y)).
\end{align}

Combining \eqref{E316c}, \eqref{E316d} and the fact that $\text{supp}(\phi^r) \cap M \times \{0\} \subset \{ C_3 r \le d_0^2(p,\cdot) \le C_4 r\}$, it is easy to show as \eqref{E312bb} that
\begin{align} \label{E316e}
\lim_{r \to \infty} \int_M |b_{(x,t)}(y,0)|(1-\phi^r(y))\,dv(y)=0
\end{align}
uniformly for $(x,t) \in K$. From \eqref{E316e},  the Claim is proved.

Next, for any vector $w \in T_x M$ with $|w|_t=1$ we compute
\begin{align} 
\partial_w \NN^r(x,t)=& \int_M \left\{ (\partial_w b)H \phi^r+b(\partial_w H)\phi^r \right\} dV_0 \notag \\
=& \int_M \left\{ -(\partial_w H) \phi^r+b(\partial_w H)\phi^r \right\} dV_0 =: I+II, \label{E317a}
\end{align}
where $H=H(x,t,y,0)$. Notice that 
\begin{align*}
\int_M H \phi^r \,dV_0=\int_M H(x,t,y,0) \phi^r(y)\,dV_0(y)
\end{align*}
is the heat solution starting from $\phi^r$. Therefore, it follows from \eqref{E303a} and \eqref{E205a} that
\begin{align} \label{E317b}
|I| \le \left| \na_x \int_M H(x,t,y,0) \phi^r(y)\,dV_0(y) \right| \le Cr^{-\frac 1 2}.
\end{align}
Next, we estimate
\begin{align*}
II=\int_M b \frac{\partial_w H}{H} \phi^r \,dv =\int_M \lc b-\NN_0^*-\frac{n}{2} \rc \frac{\partial_w H}{H} \phi^r \,dv- \lc \NN_0^*+\frac{n}{2} \rc I.
\end{align*}
Therefore, we have
\begin{align} \label{E317c}
|II| \le& \lc \int_M \lc b-\NN_0^*-\frac{n}{2} \rc^2\,dv \rc^{\frac 1 2} \lc \int_M \lc \frac{\partial_w H}{H} \rc^2 \,dv \rc^{\frac 1 2}+Cr^{-\frac{1}{2}} \le \sqrt{\frac{n}{2t}}+Cr^{-\frac{1}{2}},
\end{align}
where we have used \eqref{eqn:int}, \eqref{E313b} and \eqref{E317b}. Combining \eqref{E317a}, \eqref{E317b} and \eqref{E317c}, we obtain
\begin{align} \label{E317d}
\left| \na_x \NN^r(x,t) \right| \le \sqrt{\frac{n}{2t}}+Cr^{-\frac{1}{2}}.
\end{align}

Since $\NN^r$ converges to $\NN_0^*$ locally uniformly by the Claim, we immediately conclude from \eqref{E317d} that $\NN_0^*$ is $\sqrt{\frac{n}{2t}}$-Lipschitz.

Next, by direct computation, we have
\begin{align} 
\square \NN^r(x,t)=\int_M \left| \frac{\na_x H}{H} \right |^2 \phi^r \,dv-\frac{n}{2t} \int_M \phi^r\,dv. \label{E318a}
\end{align}
Combining \eqref{eqn:int2}, \eqref{E318a} and the Claim, it follows immediately that
$$
-\frac{n}{2t} \le \square \NN_0^* \le 0
$$
in the distribution sense.
\end{proof}

\begin{rem}
Later, we will show that the conclusions in Theorem \ref{thm:T305} hold in the classical sense once we know the decay of the conjugate heat kernel; see Corollary \ref{cor:404}.
\end{rem}

As an application of Theorem \ref{thm:T305}, we prove the following oscillation of the Nash entropy.

\begin{cor} \label{cor:303}
For any $x_1,x_2 \in M$ and $s<t^* \le t_1,t_2 <1$, we have
\begin{align} 
\NN_s^*(x_1,t_1)-\NN_s^*(x_2,t_2) \le \sqrt{\frac{n}{2(t^*-s)}}d^{t^*}_{W_1}\lc v_{x_1,t_1;t^*},v_{x_2,t_2;t^*} \rc+\frac{n}{2} \log \lc\frac{t_2-s}{t^*-s} \rc. \label{E319}
\end{align}
In particular, if $s<t^*=t_2 \le t_1<1$, then
\begin{align} 
\NN_s^*(x_1,t_1)-\NN_s^*(x_2,t_2) \le \sqrt{\frac{n}{2(t_2-s)}}d^{t_2}_{W_1}\lc v_{x_1,t_1;t_2}, \delta_{x_2} \rc. \label{E319a}
\end{align}
If we further assume $(x_2,t_2)$ is an $H_n$-center of $(x_1,t_1)$, then
\begin{align} 
\NN_s^*(x_1,t_1)-\NN_s^*(x_2,t_2) \le \sqrt{\frac{nH_n(t_1-t_2)}{2(t_2-s)}}. \label{E319b}
\end{align}
\end{cor}

\begin{proof}
The proof follows verbatim from \cite[Corollary 5.11]{Bam20a}. The only difference is that we consider $\NN^r$ as defined in \eqref{E316b} instead and let $r \to \infty$.
\end{proof}

\section{Heat kernel estimates}
Throughout this section, we assume $(M^n,g(t))_{t<1}$ is the Ricci flow associated with a Ricci shrinker in $\MM(A)$. First, we recall the following no-local-collapsing theorem proved in \cite[Theorem 22]{LW20}.

\begin{thm}\label{thm:volume1}
For any $x \in M$ and $t<1$, if $R \le r^{-2}$ on $B_t(x,r)$, then
\begin{align}\label{E400a}
|B_t(x,r)|_t \ge c e^{\boldsymbol{\mu}} r^n
\end{align}
for some constant $c=c(n)>0$.
\end{thm}

One can improve \eqref{E400a} by using the Nash entropy. Based on the Lipschitz property of the Nash entropy, we can follow the same proof of \cite[Theorem 6.1]{Bam20a} to obtain the following result. Notice that by \eqref{E315a}, \eqref{E400b} is stronger than \eqref{E400a}

\begin{thm}\label{thm:volume2}
For any $x \in M$ and $t<1$, if $R \le r^{-2}$ on $B_t(x,r)$, then
\begin{align}\label{E400b}
|B_t(x,r)|_t \ge c \exp\lc \NN_{x,t}(r^2) \rc r^n
\end{align}
for some constant $c=c(n)>0$.
\end{thm}

By using \eqref{E313b}, we also have the following volume estimate around an $H_n$-center by following the same proof of \cite[Theorem 6.2]{Bam20a}.

\begin{thm}\label{thm:volume 3}
For any $x \in M$ and $t<1$, if $(z,t-r^2)$ is an $H_n$-center of $(x,t)$, then
\begin{align}\label{E400c}
|B_{t-r^2}(z,r)|_{t-r^2} \ge c \exp\lc \NN_{x,t}(r^2) \rc r^n
\end{align}
for some constant $c=c(n)>0$ and any $r \ge 0$.
\end{thm}

Next, we recall the following upper bound estimate of the heat kernel proved in \cite[Theorem 15]{LW20}, which has already been used in the last section. 

\begin{thm}\label{thm:heatupper1}
For any $x,y \in M$ and $s<t<1$, we have
\begin{align}\label{E401}
H(x,t,y,s) \le \frac{e^{-\boldsymbol{\mu}}}{\lc 4\pi (t-s)\rc^{\frac n 2}}.
\end{align}
\end{thm}

Instead of using the entropy $\boldsymbol{\mu}$, one can include the Nash entropy and obtain the following result; see \cite[Theorem 7.1]{Bam20a}.

\begin{thm}\label{thm:heatupper2}
For any $x,y \in M$ and $s<t<1$, we have
\begin{align}\label{E401xa}
H(x,t,y,s) \le \frac{C(n)}{(t-s)^{\frac n 2}} \exp \lc -\NN_{x,t}(t-s) \rc.
\end{align}
\end{thm}

\begin{proof}
The proof follows almost the same as \cite[Theorem 7.1]{Bam20a}. The main idea is to improve the bound $Z$ of the estimate
\begin{align*}
H(x,t,y,s) \le \frac{Z}{(t-s)^{\frac n 2}} \exp \lc -\NN_{x,t}(t-s) \rc.
\end{align*}
Notice that such $Z$ always exists by \eqref{E401} and \eqref{E315a}, which may depend on the Ricci shrinker. Thanks to \eqref{E319} and \eqref{E400c}, we can follow the same argument as in \cite[Theorem 7.1]{Bam20a} to improve $Z$ to be $Z/2$, if $Z \ge \bar Z(n)$. 
\end{proof}

With the help of Theorem \ref{thm:T304}, Corollary \ref{cor:303} and Theorem \ref{thm:heatupper2}, we obtain the following gradient estimate of the heat kernel as \cite[Theorem 7.5]{Bam20a}, which improves \cite[Lemma 18]{LW20}.

\begin{thm}\label{thm:gra}
For any $x,y \in M$ and $s<t<1$, then
\begin{align}\label{E401xb}
\frac{|\na_x H|(x,t,y,s)}{H(x,t,y,s)} \le \sqrt{\frac{C}{t-s}} \sqrt{\log \lc \frac{C \exp \lc -\NN_{x,t}(t-s) \rc}{(t-s)^{\frac n 2} H(x,t,y,s)} \rc}
\end{align}
for some constant $C=C(n)>0$. 
\end{thm}

With the gradient estimate \eqref{E401xb}, one obtains the following non-expanding estimate as \cite[Theorem 8.1]{Bam20a}. Notice that \eqref{E401xc} generalizes the global volume estimate Lemma \ref{L202}.

\begin{thm}\label{thm:volume4}
For any $x \in M$, $t<1$ and $r \ge 0$, we have
\begin{align}\label{E401xc}
|B_t(x,r)|_t \le C(n) \exp\lc \NN_{x,t}(r^2) \rc r^n \le C(n)r^n.
\end{align}
\end{thm}

Before we prove more refined heat kernel estimates, we first prove a series of lemmas.

\begin{lem}[Distance comparison]\label{lem:401}
For any $\delta \in (0,1)$, there exists a constant $L_1=L_1(n,\delta)>1$ such that
\begin{align}\label{E402a}
d_t(x,p) \le d_s(x,p)+L_1\le \frac{L_1}{\sqrt{1-t}}(d_t(x,p)+1)+L_1^2
\end{align}
for any $x \in M$ and $-\delta^{-1} \le s \le t<1$.
\end{lem}

\begin{proof}
From \eqref{E202a} and \eqref{E202c}, we have
\begin{align*}
-\frac{F}{1-t} \le \partial_t F=-(1-t) R \le 0.
\end{align*}
Therefore, for any $x \in M$, 
\begin{align}\label{E402b}
\frac{1-t}{1-s}F(x,s) \le F(x,t) \le F(x,s).
\end{align}
Consequently, \eqref{E402a} follows from the combination of Lemma \ref{L201} and \eqref{E402b}.
\end{proof}

As an application of the distance comparison, we have the following lower bound of the heat kernel.

\begin{thm}\label{thm:lower}
For any $K>1$, $\delta \in (0,1)$ and $A>0$, there exists a constant $C=C(n,K,\delta,A)>1$ satisfying the following property.

Suppose $-\delta^{-1} \le s <t \le 1-\delta$ and $d_t(x,p) \le K$, then 
\begin{align}\label{RA18a}
H(x,t,y,s) \ge \frac{C^{-1}}{(t-s)^{\frac n 2}} \exp \lc-\frac{d_s^2(x,y)}{C^{-1}(t-s)} \rc.
\end{align}
\end{thm}

\begin{proof}
In the proof, all constants $C_i>1$ depend on $n,K,\delta$ and $A$.

It follows from \cite[Formula (203)]{LW20} that
\begin{align}\label{RA18b}
H(x,t,y,s) \ge \frac{C_1^{-1}}{(t-s)^{\frac n 2}} \exp \lc-\frac{d_t^2(x,y)}{3(t-s)}-\frac{4(t-s)}{3(1-t)^2}F(y,t) \rc.
\end{align}
From \eqref{E402a}, we have
\begin{align}\label{RA18c}
d_t^2(x,y) \le 2(d_t^2(x,p)+d_t^2(p,y)) \le C_2(d_s^2(x,p)+d_s^2(p,y)+1) \le C_3(d_s^2(x,y)+1),
\end{align}
where we have used $d_s(x,p) \le C(d_t(x,p)+1) \le C(K+1)$ by \eqref{E402a}.

In addition, since $F$ is decreasing with respect to $t$,
\begin{align}\label{RA18d}
F(y,t) \le F(y,s) \le C_4(d_s^2(p,y)+1) \le C_5(d_s^2(x,y)+1),
\end{align}
by Lemma \ref{L201}. Combining \eqref{RA18b}, \eqref{RA18c} and \eqref{RA18d}, it is easy to see \eqref{RA18a} holds for some $C$.
\end{proof}

\begin{lem}\label{lem:402}
For any $K>1$, $\delta \in (0,1)$ and $A>0$, there exist constants $L_2=L_2(n,K,\delta,A)>1$ and $L_3=L_3(n,\delta, A)>1$ satisfying the following property.

Suppose $-\delta^{-1} \le s <t \le 1-\delta$ and $d_t(p,x) \le K$, then for any $H_n$-center $(z,s)$ of $(x,t)$, we have
\begin{align}\label{E403a}
d_s(x,z) \le L_2 \sqrt{t-s}
\end{align}
and 
\begin{align}\label{E403aa}
d_s(z,p) \le d_t(x,p)+L_3 \sqrt{t-s}. 
\end{align}
\end{lem}

\begin{proof}
Since $d_t(p,x) \le K$, it follows from \cite[Theorem 19]{LW20} that 
\begin{align}\label{E403b}
v_{x,t;s} \lc M \setminus B_s(x,r\sqrt{t-s}) \rc \le C_2 \exp \lc -\frac{r^2}{5} \rc
\end{align}
for any $r \ge 1$ and $C_2=C_2(n, K,\delta, A)>0$. On the other hand, by Proposition \ref{prop:305}, we have
\begin{align}\label{E403c}
v_{x,t;s} \lc M \setminus B_s(z, r\sqrt{t-s}) \rc \le C(n) \exp \lc - \frac{r^2}{5} \rc
\end{align}
for any $r \ge 0$.
Combining \eqref{E403b} and \eqref{E403c}, \eqref{E403a} follows immediately. 

If we assume $(z',s)$ to be an $H_n$-center of $(p,t)$, then \eqref{E403a} indicates that
\begin{align}\label{E403ab}
d_s(z',p) \le C_3(n, \delta, A) \sqrt{t-s}.
\end{align}
Then it follows from Proposition \ref{prop:301} and \eqref{eq:hdis} that
\begin{align*}
d_s(z,p) \le& d_s(z,z')+d_s(z',p) \\
\le& d_{W_1}^s(\delta_z,\delta_{z'})+C_3 \sqrt{t-s} \\
\le& d_{W_1}^s(v_{x,t;s},v_{p,t;s})+d_{W_1}^s(\delta_z,v_{x,t;s})+d_{W_1}^s(\delta_{z'},v_{p,t;s})+C_3 \sqrt{t-s} \notag\\
\le& d_t(x,p)+2\sqrt{H_n(t-s)}+C_3 \sqrt{t-s}.
\end{align*}
Therefore, \eqref{E403aa} holds for $L_3=C_3+2\sqrt{H_n}$.
\end{proof}

Next, we prove the following rough heat kernel estimate. 

\begin{prop}\label{prop:401}
For any $K>1$, $\delta \in (0,1)$ and $A>0$, there exists a constant $L_4=L_4(n,K,\delta,A)>1$ satisfying the following property.

Suppose $-\delta^{-1} \le s <t \le 1-\delta$ and $d_s(x,p)+d_s(y,p) \le K$, then 
\begin{align}\label{E404a}
H(x,t,y,s) \le \frac{L_4}{(t-s)^{\frac n 2}} \exp \lc-\frac{d_s^2(x,y)}{L_4(t-s)} \rc.
\end{align}
\end{prop}

\begin{proof}
Without loss of generality, we assume $s=0$. In the proof, all constants $C_i$ depend on $n,K,\delta$ and $A$. 

Given $0<t \le 1-\delta$ and $x,y\in M$ with $d_0(x,p)+d_0(y,p) \le K$, we set $d:=d_0(x,y)$. It follows from Lemma \ref{lem:401} that
\begin{align}\label{E404b}
d_l(x,p)+d_l(y,p) \le C_1
\end{align}
for any $l\in [0,t]$. Therefore, it follows from the local distance distortion estimate \cite[Theorem 18]{LW20} that there exists $C_2>1$ such that if $d \ge C_2 \sqrt{t}$,
\begin{align}\label{E404c}
C_2^{-1} d \le d_l(x,y) \le C_2 d
\end{align}
for any $l \in [0,t]$. Notice that if $d \le C_2 \sqrt{t}$, \eqref{E404a} follows immediately from \eqref{E401}. Consequently, we may assume $d \ge C_2 \sqrt{t}$ and hence \eqref{E404c} holds.

For any $l \in [0,t/2]$, we apply \cite[Theorem 14]{LW20} for sets $B_l(x,\sqrt{t})$, $B_l(y,\sqrt{t})$ and parameter $\sigma=1$ to obtain
\begin{align}\label{E404d}
v_{x,t;l} \lc B_l(x,\sqrt{t}) \rc v_{x,t;l} \lc B_l(y,\sqrt{t}) \rc \le \exp \lc -\frac{(d_l(x,y)-2\sqrt{t})_+^2}{16t} \rc \le C_3\exp \lc -\frac{d^2}{C_3 t} \rc
\end{align}
for some $C_3>1$, where we have used \eqref{E404c}. In addition, for any $l \in [0,t/2]$ and $d_l(x,z) \le \sqrt{t}$, it follows from \cite[Theorem 18]{LW20} that $d_t(x,z) \le C_4 \sqrt{t}$. Therefore, it follows from \cite[Theorem 17]{LW20} that
\begin{align*}
H(x,t,z,l) \ge C_5^{-1} t^{-\frac{n}{2}}
\end{align*}
and hence
\begin{align}\label{E404e}
v_{x,t;l} \lc B_l(x,\sqrt{t}) \rc \ge C_5^{-1} t^{-\frac{n}{2}} |B_l(x,\sqrt{t})| \ge C_6^{-1},
\end{align}
where we have used the fact that $R$ is bounded on $B_l(x,\sqrt{t})$ and the no-local-collapsing Theorem \ref{thm:volume1}.

Combining \eqref{E404d} and \eqref{E404e}, we obtain for any $l \in [0,t/2]$ that
\begin{align}\label{E404f}
\int_{B_l(y,\sqrt{t})} H(x,t,z,l) \,dV_l(z) \le C_7\exp \lc -\frac{d^2}{C_3 t} \rc.
\end{align}
In light of \eqref{E401},   for any $l \in [0,t/2]$, the above inequality implies that
\begin{align}\label{E404g}
\int_{B_l(y,\sqrt{t})} H^2(x,t,z,l) \,dV_l(z) \le \frac{C_8}{t^{\frac n 2}}\exp \lc -\frac{d^2}{C_3 t} \rc.
\end{align}
Integrating $l$ from $0$ to $t/2$, we have
\begin{align}\label{E404h}
\int_0^{\frac t 2} \int_{B_l(y,\sqrt{t})} H^2(x,t,z,l) \,dV_l(z)ds \le \frac{C_9}{t^{\frac n 2-1}}\exp \lc -\frac{d^2}{C_3 t} \rc.
\end{align}
Consequently, the desired heat kernel estimate \eqref{E404a} follows from \eqref{E404h} and a parabolic mean value inequality \cite[Lemma 4.2]{BZ17}. Here, \cite[Lemma 4.2]{BZ17} can be applied in our setting since the key ingredient is the existence of a nice local cutoff function, which is constructed in \cite[Theorem 1.3]{BZ17} (see also Proposition \ref{prop:cutoff}). Once the existence of the local cutoff function is guaranteed, one can follow verbatim the proof of \cite[Lemma 4.2]{BZ17} to obtain the mean value inequality.
\end{proof}

We immediately obtain the following result by combining Lemma \ref{lem:402} and Proposition \ref{prop:401}.

\begin{cor}\label{cor:401}
For any $K>1$, $\delta \in (0,1)$ and $A>0$, there exists a constant $L_5=L_5(n,K,\delta,A)>1$ satisfying the following property.

Suppose $-\delta^{-1} \le s <t \le 1-\delta$ and $d_t(x,p)+d_s(y,p) \le K$, then for any $H_n$-center $(z,s)$ of $(x,t)$, we have
\begin{align}\label{E405a}
H(x,t,y,s) \le \frac{L_5}{(t-s)^{\frac n 2}} \exp \lc-\frac{d_s^2(z,y)}{L_5(t-s)} \rc.
\end{align}
\end{cor}

\begin{proof}
From Lemma \ref{L201} and \eqref{E402a}, we have $d_s(x,p)+d_s(y,p) \le C$ for some $C=C(n, K, \delta)>0$. Then \eqref{E405a} follows from \eqref{E403a} and \eqref{E404a}.
\end{proof}

Next, we prove the following technical result.

\begin{lem}\label{lem:tech}
There exists a positive constant $\bar Q=\bar Q(n) > 0$ satisfying the following property.

Suppose $x,y \in M$, $T \in (0,1)$ and there exists an $H_n$-center $(z,0)$ of $(x,T)$ such that
\begin{align}\label{E406a}
H(x,T,y,0) \ge Q\frac{\exp \lc -\NN_0^*(x,T) \rc}{T^{\frac n 2}} \exp \lc -\frac{d_0^2(z,y)}{QT} \rc
\end{align}
for some $Q \ge \bar Q$. Then for any $H_n$-center $(z',T_1)$ of $(x,T)$, there exist a point $x_1 \in M$ and an $H_n$-center $(z_1,0)$ of $(x_1,T_1)$ such that 
\begin{align}\label{E406b}
d_{T_1}(x_1,z') \le \frac{10}{\sqrt{Q}} d_0(z,y) 
\end{align}
and
\begin{align}\label{E406c}
H(x_1,T_1,y,0) \ge Q_1 \frac{\exp \lc -\NN_0^*(x_1,T_1) \rc}{T_1^{\frac n 2}} \exp \lc -\frac{d_0^2(z_1,y)}{Q_1 T_1} \rc,
\end{align}
where $T_1=T/8$ and $Q_1=2Q$.
\end{lem}

\begin{figure}
\begin{center}
\psfrag{A}[c][c]{\color{green}{$(x,T)$}}
\psfrag{B}[c][c]{$(y,0)$}
 \psfrag{C}[c][c]{\color{green}{$(z,0)$}}
 \psfrag{D}[c][c]{\color{green}$(z', T_1)$}
 \psfrag{E}[c][c]{\color{red}$(x_1,T_1)$}
 \psfrag{F}[c][c]{\color{red}$(z_1,0)$}
 \psfrag{TA}[c][c]{$t=T$}
 \psfrag{TB}[c][c]{$t=T_1$}
 \psfrag{TC}[c][c]{$t=0$}
 \psfrag{BB}[c][c]{\color{brown}$B$}
\includegraphics[width=0.8 \columnwidth]{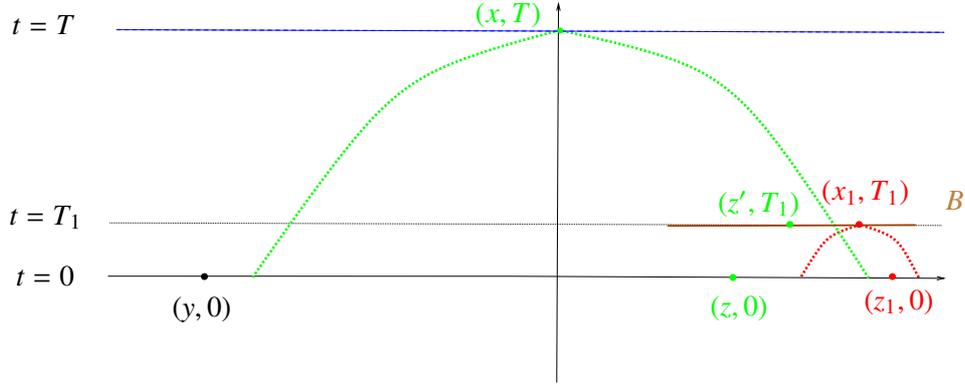}
\caption{Find a new point with improved lower bound}
\label{fig:pointselecting}
\end{center}
\end{figure}

\begin{proof}
In the proof, all constants $C_i>1$ depend only on $n$. We set
\begin{align*}
d:=d_0(z,y), \quad a:=H(x,T,y,0), \quad v_t:=v_{x,T;t},\quad \text{and}\quad V:=\left \{w\in M \mid H(w,T_1,y,0) \ge \frac{a}{2} \right\}.
\end{align*}
Notice that by \eqref{E406a} and \eqref{E401xa}, we have
\begin{align*}
C_1 \frac{\exp \lc -\NN_0^*(x,T) \rc}{T^{\frac n 2}} \ge a \ge \sqrt{Q} \cdot \sqrt{Q} \frac{\exp \lc -\NN_0^*(x,T) \rc}{T^{\frac n 2}} \exp \lc -\frac{d^2}{QT} \rc.
\end{align*}
Thus if  $\bar Q$ is sufficiently large, we have $Q>\bar{Q}>C_1^2$ and derive from the above inequality that
\begin{align} \label{E407aa}
\frac{d^2}{T} \ge \frac{Q \log Q}{2}.
\end{align}

It follows from the semigroup property \eqref{E301} that
\begin{align}
a=&\int_M H(w,T_1,y,0)\,dv_{T_1}(w) \notag \\
=& \int_{M \setminus V} H(w,T_1,y,0)\,dv_{T_1}(w)+\int_{V} H(w,T_1,y,0)\,dv_{T_1}(w) \notag\\
\le & \frac{a}{2} v_{T_1}(M \setminus V)+\int_{V} H(w,T_1,y,0)\,dv_{T_1}(w) \notag\\
\le& \frac{a}{2}+C_1 T_1^{-\frac{n}{2}} \int_{V} \exp\lc -\NN_0^*(w,T_1) \rc\,dv_{T_1}(w), \label{E407a}
\end{align}
where we have used \eqref{E401xa} for the last inequality. Moreover, it follows from \eqref{E319b} and the Lipschitz property of $\NN_0^*$ that
\begin{align} \label{E407xa}
-\NN_0^*(z',T_1)\le -\NN_0^*(x,T)+ C_2\sqrt{\frac{T-T_1}{T_1}} \le -\NN_0^*(x,T)+ C_3
\end{align}
and
\begin{align} \label{E407xb}
-\NN_0^*(w,T_1)\le -\NN_0^*(z',T_1)+\sqrt{\frac{n}{2T_1}}d_{T_1}(w,z') \le -\NN_0^*(x,T)+C_4 T^{-\frac 1 2} d_{T_1}(w,z')+C_4.
\end{align}

Now, we define $B:=B_{T_1}(z',10 Q^{-\frac 1 2} d)$. Then it follows from \eqref{E407xb}that
\begin{align}
&\int_{V} \exp\lc -\NN_0^*(w,T_1) \rc\,dv_{T_1}(w) \notag\\
\le & e^{C_4}\exp\lc -\NN_0^*(x,T) \rc \int_{V} e^{C_4T^{-\frac 1 2}d_{T_1}(w,z')} \,dv_{T_1}(w) \notag\\
\le & e^{C_4}\exp\lc -\NN_0^*(x,T) \rc \lc e^{10C_4T^{-\frac 1 2}Q^{-\frac 1 2}d} v_{T_1}(V\cap B)+\int_{M \setminus B} e^{C_4T^{-\frac 1 2}d_{T_1}(w,z')} \,dv_{T_1}(w) \rc. \label{E407xc}
\end{align}

For a small constant $\beta>0$ to be determined later, it follows from Proposition \ref{prop:305} that
\begin{align}
& \int_{M \setminus B} e^{C_4T^{-\frac 1 2}d_{T_1}(w,z')} \,dv_{T_1}(w) \notag\\
=& \sum_{k=1}^{\infty} \int_{2^{k-1} (10 Q^{-\frac 1 2} d)\le d_{T_1}(w,z') \le 2^{k}(10 Q^{-\frac 1 2} d)} e^{C_4T^{-\frac 1 2}d_{T_1}(w,z')} \,dv_{T_1}(w) \notag\\
\le & \sum_{k=1}^{\infty} e^{C_42^{k} T^{-\frac 1 2}10 Q^{-\frac 1 2} d} \int_{d_{T_1}(w,z') \ge 2^{k-1}(10 Q^{-\frac 1 2} d)} \,dv_{T_1}(w) \notag\\
\le & C_5 \sum_{k=1}^{\infty} \exp \lc C_42^{k} T^{-\frac 1 2}10 Q^{-\frac 1 2} d-\frac{(2^{k-1}10 Q^{-\frac 1 2} d)^2}{5(T-T_1)} \rc \notag\\
\le & C_5 \sum_{k=1}^{\infty} \exp \lc -\frac{(2^{k-1}10 Q^{-\frac 1 2} d)^2}{5T} +C_6\rc \le C_7 \exp \lc -\frac{ 20 d^2}{QT} \rc, \label{E407xd}
\end{align}
where we have used the fact that $\exp \lc -\frac{ 20 d^2}{QT} \rc \le Q^{-10} \ll 1$ by \eqref{E407aa}.

Combining \eqref{E406a}, \eqref{E407a}, \eqref{E407xc} and \eqref{E407xd}, we have
\begin{align} \label{E407c}
QT^{-\frac{n}{2}} \exp \lc -\frac{d^2}{QT} \rc \le C_8 T_1^{-\frac{n}{2}} \lc e^{10C_4T^{-\frac 1 2}Q^{-\frac 1 2}d} v_{T_1}(V\cap B)+\exp \lc -\frac{ 20 d^2}{QT} \rc \rc.
\end{align}
Since $Q$ is large, by \eqref{E407aa} we have
\begin{align*}
Q\exp \lc\frac{ 19 d^2}{QT} \rc \ge Q^{\frac{21}{2}} \ge 2C_8 8^{\frac{n}{2}}.
\end{align*}
Then it is not hard to see from \eqref{E407c} that $v_{T_1}(V\cap B) >0$.
Thus there exists a point $x_1 \in V \cap B$ which satisfies \eqref{E406b}. Then we take an $H_n$-center $(z_1,0)$ of $(x_1,T_1)$. 
The point selecting process is illustrated in Figure~\ref{fig:pointselecting}. 

It follows from Proposition \ref{prop:301} and \eqref{eq:hdis} that
\begin{align}
d_0(z,z_1) =&d_{W_1}^0(\delta_z,\delta_{z_1}) \notag \\
\le& d_{W_1}^0(v_{x,T;0},v_{x_1,T_1;0})+d_{W_1}^0(\delta_z,v_{x,T;0})+d_{W_1}^0(\delta_{z_1},v_{x_1,T_1;0}) \notag\\
\le& d_{W_1}^{T_1}(v_{x,T;T_1}, \delta_{x_1})+\sqrt{H_nT}+\sqrt{H_nT_1}\notag \\
\le &d_{W_1}^{T_1}(v_{x,T;T_1}, \delta_{z'})+d_{T_1}(z',x_1)+\sqrt{H_nT}+\sqrt{H_nT_1}\notag \\
\le &\sqrt{H_n(T-T_1)}+\sqrt{H_nT}+\sqrt{H_nT_1}+10 Q^{-\frac 1 2} d \notag\\
\le & 3\sqrt{H_nT}+10 Q^{-\frac 1 2} d. \label{E407d}
\end{align}
Therefore, we conclude
\begin{align} \label{E407e}
d_0(z_1,y) \ge d-d_0(z,z_1) \ge (1-10 Q^{-\frac 1 2})d-3\sqrt{H_nT}
\end{align}
and hence
\begin{align} \label{E407f}
d^2_0(z_1,y) \ge \frac{(1-10 Q^{-\frac 1 2})^2}{2 }d^2-9H_nT.
\end{align}

Since $x_1 \in V$,  from the definition of $V$ and \eqref{E406a} we have
\begin{align} \label{E407g}
H(x_1,T_1,y,0) \ge \frac{a}{2} \ge Q \frac{\exp \lc -\NN_0^*(x,T) \rc}{2T^{\frac n 2}} \exp \lc -\frac{d^2}{QT} \rc, 
\end{align}
which enables us to claim
\begin{align} \label{E408aa}
Q \frac{\exp \lc -\NN_0^*(x,T) \rc}{2T^{\frac n 2}} \exp \lc -\frac{d^2}{QT} \rc \ge Q_1\frac{\exp \lc -\NN_0^*(x_1,T_1) \rc}{T_1^{\frac n 2}} \exp \lc -\frac{d_0^2(z_1,y)}{Q_1 T_1} \rc.
\end{align}
Indeed, it follows from \eqref{E407xb} that
\begin{align} \label{E408xa}
-\NN_0^*(x_1,T_1) \le -\NN_0^*(x,T)+C_9(1+d(QT)^{-\frac 1 2}).
\end{align}
On the other hand, by \eqref{E407f} we have
\begin{align}
&\exp \lc\frac{d_0^2(z_1,y)}{Q_1 T_1}-\frac{d^2}{QT}-C_9\sqrt{\frac{d^2}{QT}}-C_9 \rc \notag \\
\ge& \exp \lc \frac{1}{QT} \lc 2(1-10Q^{-\frac 1 2})^2-0.9 \rc d^2-\frac{36H_n}{Q}-C_{10}\rc \notag \\
\ge& \exp \lc \frac{d^2}{QT}-\frac{36H_n}{Q}-C_{10}\rc \ge \sqrt{Q} \exp \lc-\frac{36H_n}{Q}-C_{10}\rc \ge 4 \cdot 8^{\frac n 2}, \label{E408xb}
\end{align}
where we have used \eqref{E407aa} for the last inequality. 
As $\bar Q$ is sufficiently large, it is clear that \eqref{E408aa} follows from the combination of  \eqref{E408xa} and \eqref{E408xb}.
Consequently, we obtain \eqref{E406c}. 
\end{proof}

\begin{prop}\label{prop:weak}
For any $x,y \in M$ and $t \in (0,1)$,
\begin{align}\label{E408a}
H(x,t,y,0) \le \bar Q \frac{\exp \lc -\NN_0^*(x,t) \rc}{t^{\frac n 2}} \exp \lc -\frac{d_0^2(z,y)}{\bar Q t} \rc
\end{align}
where $(z,0)$ is any $H_n$-center of $(x,t)$ and $\bar Q$ is the same constant in Proposition \ref{lem:tech}.
\end{prop}

\begin{proof}
Suppose otherwise, there exist $x,y \in M$, $T \in (0,1)$ and an $H_n$-center $(z,0)$ of $(x,T)$ such that
\begin{align}\label{E408b}
H(x,T,y,0) \ge \bar Q \frac{\exp \lc -\NN_0^*(x,T) \rc}{T^{\frac n 2}} \exp \lc -\frac{d_0^2(z,y)}{\bar Q T} \rc.
\end{align}
Now, we define $Q_k=2^k \bar Q$ and $T_k:=8^{-k}T$ for $k \in \mathbb N$. If we set $x_0=x$ and $z_0=z$, then we claim there are sequences $x_k,z'_k$ and $z_k$ satisfying
\begin{enumerate}[label=\textnormal{(\alph{*})}]
\item $(z_k',T_k)$ is an $H_n$-center of $(x_{k-1},T_{k-1})$.

\item $(z_k,0)$ is an $H_n$-center of $(x_k,T_k)$.

\item $d_{T_k}(x_k,z'_k) \le 10Q_{k-1}^{-\frac{1}{2}} d_0(z_{k-1},y)$.

\item $d_0(z_k,z_{k-1}) \le 3\sqrt{H_n T_{k-1}}+10 Q_{k-1}^{-\frac 1 2} d_0(z_{k-1},y)$.

\item We have the heat kernel estimate
\begin{align}\label{E408c}
H(x_k,T_k,y,0) \ge Q_k \frac{\exp \lc -\NN_0^*(x_k,T_k) \rc}{T_k^{\frac n 2}} \exp \lc -\frac{d_0^2(z_k,y)}{Q_k T_k} \rc.
\end{align}
\end{enumerate}

The existence of $x_k,z'_k$ and $z_k$ satisfying (a)-(e) is obtained by Lemma~\ref{lem:tech} and an inductive argument. Notice that (d) is guaranteed by \eqref{E407d}.

\textbf{Claim}: $b_k:=d_{T_k}(x_k,p)$ is uniformly bounded.

\emph{Proof of the Claim}: We set $d_k:=d_0(z_k,y)$ for $k \in \mathbb N$. It follows from (d) that
\begin{align}\label{E408d}
d_k \le d_{k-1}+d_0(z_k,z_{k-1}) \le \lc 1+10 Q_{k-1}^{-\frac 1 2} \rc d_{k-1}+3\sqrt{H_n T_{k-1}}.
\end{align}
Therefore, it is easy to derive from \eqref{E408d} and the definitions of $Q_k$ and $T_k$ that
\begin{align}\label{E408e}
d_k \le K_1<\infty
\end{align}
for some constant $K_1$ depending on $d_0(z,y),T,\bar Q$ and $n$. From (c) and \eqref{E408e}, we have
\begin{align}\label{E408f}
d_{T_k}(x_k,z_k') \le 10Q_{k-1}^{-\frac{1}{2}} d_{k-1} \le 10K_1 Q_{k-1}^{-\frac{1}{2}}.
\end{align}
Moreover, since $(z_k',T_k)$ is an $H_n$-center of $(x_{k-1},T_{k-1})$, it follows from \eqref{E403aa} that
\begin{align}\label{E408g}
d_{T_k}(z_k',p) \le d_{T_{k-1}}(x_{k-1},p)+L_3 \sqrt{T_{k-1}-T_k} \le b_{k-1}+L_3 T_{k-1}^{\frac 1 2},
\end{align}
where $L_3=L_3(n,\delta, A)>0$ for some fixed constant $\delta \in (0,1)$ with $T \le 1-\delta$. Combining \eqref{E408f} and \eqref{E408g}, we obtain
\begin{align}\label{E408h}
b_k \le b_{k-1}+10K_1 Q_{k-1}^{-\frac{1}{2}}+L_3 T_{k-1}^{\frac 1 2}.
\end{align}
From \eqref{E408h}, it is clear that $b_k$ is uniformly bounded, and the Claim is proved.

Thanks to the Claim, we can apply Corollary \ref{cor:401} to obtain an upper bound of heat kernel, which contradicts the lower bound \eqref{E408c} when $k$ is sufficiently large. 
\end{proof}

Now, we state the main theorem of this section regarding the heat kernel upper bound, which generalizes and slightly improves \cite[Theorem 7.2]{Bam20a}.

\begin{thm}\label{thm:heatupper}
$(M^n,g(t))_{t<1}$ is the Ricci flow associated with a Ricci shrinker. For any $\ep>0$, there exists a constant $C = C(n,\ep) >0$ such that
\begin{align}\label{E409a}
H(x,t,y,s) \le \frac{C \exp \lc -\NN_{(x,t)}(t-s) \rc}{(t-s)^{\frac n 2}} \exp \lc -\frac{d_s^2(z,y)}{(4+\ep)(t-s)} \rc,
\end{align}
for any $s<t<1$ and any $H_n$-center $(z,s)$ of $(x,t)$.
\end{thm}

\begin{proof}
Without loss of generality, we assume $s=0$. The proof is a modification of the proof of Lemma~\ref{lem:tech} and all constants $C_i>1$ depend on $n$ and $\ep$.

Suppose otherwise, there exist $x,y \in M$, $T \in (0,1)$, $\ep>0$ and an $H_n$-center $(z,0)$ of $(x,T)$ such that 
\begin{align}\label{E410a}
H(x,T,y,0) \ge Q\frac{\exp \lc -\NN_0^*(x,T) \rc}{T^{\frac n 2}} \exp \lc -\frac{d_0^2(z,y)}{(4+\ep)T} \rc,
\end{align}
where $Q$ is a large constant determined later. We also set $\theta \in (0,1)$ as a small parameter and $\theta_1^3=\theta$.
Define 
\begin{align*}
&d:=d_0(z,y), \quad a:=H(x,T,y,0), \quad v_t:=v_{x,T;t},\quad T_{\theta}:=\theta T,\\
&V:=\left \{w\in M \mid H(w,T_{\theta},y,0) \ge \frac{a}{2} \right\}.
\end{align*}
From \eqref{E410a} and \eqref{E401xa}, we have
\begin{align}\label{E410b}
\exp \lc \frac{d^2}{T} \rc \ge \left\{ C_1^{-1} Q \right\}^{4+\epsilon}.
\end{align}
Now, we assume $(z',T_{\theta})$ is an $H_n$-center of $(x,T)$ and set $B:=B_{T_{\theta}}(z',(1-\theta_1)d)$. Similar to \eqref{E407xb}, we have
\begin{align} \label{E410xa}
-\NN_0^*(w,T_{\theta}) \le -\NN_0^*(x,T)+C_2 \theta^{-\frac 1 2}T^{-\frac 1 2} d_{T_{\theta}}(w,z')+C_2 \theta^{-\frac 1 2}.
\end{align}
By the same argument as \eqref{E407xd}, we apply Proposition \ref{prop:305} for $\ep/4$ to obtain if $\theta<\bar \theta(\ep)$,
\begin{align}
& \int_{M \setminus B} e^{C_2 \theta^{-\frac 1 2} T^{-\frac 1 2}d_{T_{\theta}}(w,z')} \,dv_{T_{\theta}}(w) \notag\\
\le & C_3 \sum_{k=1}^{\infty} \exp \lc C_2 2^{k} \theta^{-\frac 1 2} T^{-\frac 1 2}(1-\theta_1) d-\frac{(2^{k-1}(1-\theta_1) d)^2}{(4+\ep/4)(1-\theta)T} \rc \notag\\
\le & C_3 \sum_{k=1}^{\infty} \exp \lc -\frac{(2^{k-1}(1-\theta_1) d)^2}{(4+\ep/3)(1-\theta)T} +C_4\theta^{-1}\rc \notag\\ 
\le & C_5 \exp \lc -\frac{((1-\theta_1) d)^2}{(4+\ep/3)(1-\theta)T} +C_4\theta^{-1}\rc. \label{E410xb}
\end{align}
Similar to \eqref{E407c}, we obtain
\begin{align}
QT^{-\frac{n}{2}} \exp \lc -\frac{d^2}{(4+\ep)T} \rc \le C_6 T_{\theta}^{-\frac{n}{2}} \lc e^{C_2 \theta^{-\frac 1 2}T^{-\frac 1 2}(1-\theta_1)d} v_{T_{\theta}}(V\cap B)+e^{C_6 \theta^{-1}}\exp \lc -\frac{((1-\theta_1) d)^2}{(4+\ep/3)(1-\theta)T}\rc \rc. \label{E410c}
\end{align}
We claim that $v_{T_{\theta}}(V\cap B) >0$. Indeed, it follows from \eqref{E410b} that
\begin{align*}
Q\exp \lc \lc \frac{ (1-\theta_1)^2}{(4+\ep/3)(1-\theta)}-\frac{1}{(4+\ep)} \rc \frac{d^2}{T} \rc \ge Q \exp \lc \frac{c(\ep) d^2}{T} \rc \ge Q^{1+c(\ep)} C_1^{-c(\ep)} \ge 2 C_6 \theta^{-\frac n 2}e^{C_6 \theta^{-1}},
\end{align*}
where $c(\ep)>0$ depends only on $\ep>0$ and we choose $\theta<\bar \theta(\ep)$ and $Q$ sufficiently large. Therefore, the claim follows from \eqref{E410c}.

We choose a point $x_1 \in V \cap B$ and an $H_n$-center $(z_1,0)$ of $(x_1,T_{\theta})$. Similar to \eqref{E407d}, we have
\begin{align}\label{E410d}
d_0(z,z_1) \le 3\sqrt{H_nT}+(1-\theta_1)d
\end{align}
and hence
\begin{align}\label{E410e}
d_0(z_1,y) \ge \theta_1 d-3\sqrt{H_nT}.
\end{align}
Moreover, as \eqref{E408xa}, we have by \eqref{E410xa},
\begin{align} \label{E408ea}
-\NN_0^*(x_1,T_{\theta}) \le -\NN_0^*(x,T)+C_2 \theta^{-\frac 1 2}(T^{-\frac 1 2}d+1).
\end{align}

Now, by virtue of Proposition~\ref{prop:weak} and the definition of $V$, we have
\begin{align} \label{E410f}
\bar Q  \exp \lc -\frac{d_0^2(z_1,y)}{\bar Q T_{\theta}} \rc \ge H(x_1,T_{\theta},y,0) \cdot T^{\frac{n}{2}} \cdot \exp \lc \NN_0^*(x_1,T_{\theta}) \rc \ge \frac{1}{2}Q  \exp \lc -\frac{d^2}{(4+\ep)T} \rc.
\end{align}
Since $d_0^2(z_1,y) \ge \theta_1^2 d^2/2-9H_nT$ from \eqref{E410e}, it follows from \eqref{E408ea} and \eqref{E410f} that
\begin{align} 
& Q\exp \lc \lc \frac{1}{2\theta_1 \bar Q}-\frac{1}{(4+\ep)}-1 \rc \frac{d^2}{T}-C_7 \theta^{-1} \rc \notag \\
\le& Q\exp \lc \lc \frac{1}{2\theta_1 \bar Q}-\frac{1}{(4+\ep)} \rc \frac{d^2}{T}-C_2 \theta^{-\frac 1 2}T^{-\frac 1 2}d \rc \le 2 \bar Q \theta^{-\frac n 2} \exp \lc \frac{9H_n }{\bar Q \theta} \rc, \label{E410g}
\end{align}
provided that $\theta \le \bar \theta(\ep, \bar Q)$. However, \eqref{E410g} is impossible by \eqref{E410b} if $Q$ is sufficiently large.

In sum, we obtain a contradiction and \eqref{E409a} holds.
\end{proof}

Combining Lemma \ref{lem:402} and Theorem \ref{thm:heatupper}, we have the following estimate, which improves \cite[Theorem 20]{LW20}.

\begin{thm}\label{thm:heatupper3}
For any $K>1$, $\delta \in (0,1)$ and $A>0$, there exists a constant $C=C(n,K,\delta,A)>1$ satisfying the following property.

Suppose $-\delta^{-1} \le s <t \le 1-\delta$ and $d_t(x,p) \le K$, then 
\begin{align}\label{E411a}
H(x,t,y,s) \le \frac{C}{(t-s)^{\frac n 2}} \exp \lc-\frac{d_s^2(x,y)}{C(t-s)} \rc.
\end{align}
\end{thm}

\begin{rem} \label{rem:poten}
Given $(x_0,t_0) \in M \times (-\infty,1)$, if we set $H(x_0,t_0,y,s)=(4\pi(t_0-s))^{-\frac n 2}e^{-b(y,s)}$, then it follows from Theorem \ref{thm:lower} and Theorem \ref{thm:heatupper3} that $b(y,s)$ increases quadratically.
\end{rem}

Combining \eqref{E411a} and the standard regularity theory of the parabolic equation (cf. \cite{Fri08}), we have the following derivative estimate of higher orders.

\begin{cor}\label{cor:403}
Given $(x_0,t_0) \in M \times (-\infty,1)$ and $s_0<t_0$, there exists a small parabolic neighborhood $P=B_{t_0}(x_0,r) \times [t_0-r^2,t_0+r^2]$ such that for any $m_1,m_2 \in \mathbb N$
\begin{align}\label{E411aa}
|\partial_t^{m_1} \na_x^{m_2}H(x,t,y,s_0)| \le \frac{1}{r^{2m_1+m_2}} \cdot \frac{Q}{(t_0-s_0)^{\frac n 2}} \cdot \exp \lc-\frac{d_{s_0}^2(x_0,y)}{Q(t_0-s_0)} \rc
\end{align}
for some constant $Q>1$ and any $(x,t) \in P$ and $y \in M$.
\end{cor}

Note that when $(y,s_0)$ is fixed, $H(x,t,y,s_0)$ is a heat solution.  The scale $r$ in the above Corollary is small constant much less than the curvature radius at $(x_0,t_0)$. 
Then inequality (\ref{E411aa}) can be obtained by dominated convergence theorem. 
It indicates that one can take differentiation under the integral sign if the integrand involves the heat kernel in many cases. As an application, we can follow the same proof as in Theorem \ref{thm:T305} to estimate $|\na \NN_s^*|$ and $\square \NN_s^*$ without using $\phi^r$. Therefore, one obtains

\begin{cor}\label{cor:404}
The Nash entropy $\NN_s^*(x,t)$ is smooth on $M \times (s,1)$ satisfying
\begin{align*}
|\na \NN_s^*| \le \sqrt{\frac{n}{2(t-s)}}\quad \text{and} \quad -\frac{n}{2(t-s)} \le \square \NN_s^* \le 0
\end{align*}
in the classical sense.
\end{cor}

We end this section by proving the following hypercontractivity; see \cite[Theorem 12.1]{Bam20a}.

\begin{thm} \label{thm:ultra}
Suppose that $(x_0, t_0) \in M \times (-\infty,1)$ and $0 < \tau_1 < \tau_2$. Let $u \in C^2(M \times [t_0 - \tau_2 ,t_0 - \tau_1])$ be a nonnegative function satisfying $\square u \le 0$ and having at most polynomial spatial growth in the sense that
\begin{align} \label{E412a}
|u(x,t)| \le m(d^m_t(p,x)+1)
\end{align}
for some $m \in \mathbb N$. If $ 1< q_0 \leq p_0 < \infty$ with
\[ \frac{\tau_2}{\tau_1} \geq \frac{p_0-1}{q_0-1}, \]
then for $dv_t:=dv_{x_0,t_0;t}$,
\begin{equation} \label{E412xa}
\lc \int_Mu^{p_0} dv_{t_0 - \tau_1} \rc^{1/p_0} \leq \lc \int_M u^{q_0} dv_{t_0 - \tau_2} \rc^{1/q_0}. 
\end{equation}
\end{thm}

\begin{proof}
Without loss of generality, we assume $t_0=0$. We set $p=p(t)=1+\tau_2(q_0-1)|t|^{-1}$ for $t<0$. Notice that $p(-\tau_2)=q_0$ and $p(-\tau_1) \ge p_0$ by our assumption.

When $t<0$, direct calculation shows that
\begin{align}
&\partial_t \int_M u^{p} \phi^r \,dv_t -\int_M u^{p} \square \phi^r dv_t  \notag \\
=& \int_M \lc \dot{p} u^p \log u+pu^{p-1} \square u-p(p-1)|\na u|^2 u^{p-2} \rc\phi^r-2 \la \na u^p,\na \phi^r \ra \,dv_t \notag \\
\le& \frac{\dot{p}}{p}\int_M \phi^r u^p \log u^p\,dv_t-\frac{p-1}{p}\int_M \frac{|\na u^p|^2}{u^p} \phi^r \,dv_t-2\int_M \la \na u^p,\na \phi^r \ra \,dv_t.\label{E412b}
\end{align}
Moreover, we have for any $\ep \ll 1$,
\begin{align}
2 \int_M |\na u^p|| \na \phi^r| \,dv_t \le \ep \int_M \frac{|\na u^p|^2}{u^p} \phi^r \,dv_t+\ep^{-1} \int_M u^p |\na \phi^r|^2 \,dv_t. \label{E412c}
\end{align}
Combining \eqref{E412b} and \eqref{E412c}, we have
\begin{align}
&\partial_t \lc \int_M u^{p} \phi^r \,dv_t \rc^{\frac 1 p} \notag \\
=&\frac{1}{p}\lc \int_M u^{p} \phi^r \,dv_t \rc^{\frac 1 p-1} \lc \partial_t \int_M u^{p} \phi^r \,dv_t-\frac{\dot{p}}{p} \lc\int_M u^{p} \phi^r \,dv_t \rc \log \lc\int_M u^{p} \phi^r \,dv_t \rc \rc \notag \\
\le&\frac{1}{p}\lc \int_M u^{p} \phi^r \,dv_t \rc^{\frac 1 p-1} \lc \frac{\dot{p}}{p}\int_M \phi^r u^p \log u^p\,dv_t- \frac{\dot{p}}{p} \lc\int_M u^{p} \phi^r \,dv_t \rc\log \lc\int_M u^{p} \phi^r \,dv_t \rc \rc \notag \\
&+\frac{1}{p}\lc \int_M u^{p} \phi^r \,dv_t \rc^{\frac 1 p-1} \lc \lc \ep-\frac{p-1}{p} \rc \int_M \frac{|\na u^p|^2}{u^p} \phi^r \,dv_t+\ep^{-1} \int_M u^p \left\{ |\na \phi^r|^2 + \square \phi^r \right\} dv_t   \rc. \label{E412d}
\end{align}

We integrate \eqref{E412d} from $-\tau_2 $ to $-\tau_1$, let $r \to \infty$ and then let $\ep \to 0$. By Theorem~\ref{thm:heatupper1},~\eqref{E412a},~\eqref{E205a} and~\eqref{E205d}, we obtain
\begin{align}
&\left. \lc \int_M u^{p} \,dv_t \rc^{\frac 1 p} \right |_{-\tau_2}^{-\tau_1} \notag \\
\le&\int_{-\tau_2}^{-\tau_1}\frac{1}{p}\lc \int_M u^{p} \,dv_t \rc^{\frac 1 p-1} \lc \frac{\dot{p}}{p}\int_M u^p \log u^p\,dv_t- \frac{\dot{p}}{p} \lc\int_M u^{p} \,dv_t \rc \log \lc\int_M u^{p} \,dv_t \rc \rc \,dt \notag \\
&+\int_{-\tau_2}^{-\tau_1}\frac{1}{p}\lc \int_M u^{p} \,dv_t \rc^{\frac 1 p-1} \lc -\frac{p-1}{p} \int_M \frac{|\na u^p|^2}{u^p} \,dv_t \rc\,dt. \label{E412e}
\end{align}
Note that the log-Sobolev inequality \cite[Theorem 13]{LW20} implies that
\begin{align*}
\frac{\dot{p}}{p}\int_M u^p \log u^p\,dv_t- \frac{\dot{p}}{p} \lc\int_M u^{p} \,dv_t \rc \log \lc\int_M u^{p} \,dv_t \rc \le \frac{\dot{p}}{p}|t| \int_M \frac{|\na u^p|^2}{u^p} \,dv_t=\frac{p-1}{p} \int_M \frac{|\na u^p|^2}{u^p} \,dv_t.
\end{align*}
Therefore, it follows from \eqref{E412e} that
\begin{align*}
\lc \int_M u^{p_0} \,dv_{-\tau_1} \rc^{\frac{1}{p_0}} \le \lc \int_M u^{p(-\tau_1)} \,dv_{-\tau_1} \rc^{\frac{1}{p(-\tau_1)}} \le \lc \int_M u^{q_0} \,dv_{-\tau_2} \rc^{\frac{1}{q_0}}
\end{align*}
and the proof is complete.
\end{proof}

\begin{rem}
If $u \in C^2(M \times [t_0-\tau_2,t_0-\tau_1])$ and satisfies $\square u=0$ and \eqref{E412a}, then
\begin{equation*}
\lc \int_M|u|^{p_0} dv_{t_0 - \tau_1} \rc^{1/p_0} \leq \lc \int_M |u|^{q_0} dv_{t_0 - \tau_2} \rc^{1/q_0}. 
\end{equation*}
Indeed, one can apply \eqref{E412xa} to $\sqrt{u^2+\ep}$ since $\square \sqrt{u^2+\ep} \le 0$, and let $\ep \to 0$.
\end{rem}

\section{Parabolic neighborhoods and $\ep$-regularity theorem}

In this section, we assume $(M^n,g(t))_{t<1}$ is the Ricci flow associated with a Ricci shrinker in $\MM(A)$.

Given $(x_0,t_0) \in M \times (-\infty,1)$, we first recall the conventional parabolic neighborhoods are defined by
\begin{align}
P(x_0,t_0;S,-T^-,T^+):=&B_{t_0}(x_0,S) \times \lc [t_0-T^-,t_0+T^+]\cap (-\infty,1)\rc \label{E501a} \\
Q(x_0,t_0;S,-T^-,T^+):=&\left\{d_t(x,x_0) \le S,\,t \in [t_0-T^-,t_0+T^+]\cap (-\infty,1) \right\} \label{E501aa}
\end{align}
for any $S,T^{\pm} \ge 0$. Based on the monotonicity of $W_1$-distance in Proposition \ref{prop:301}, we follow \cite{Bam20a} to define the following new parabolic neighborhoods.

\begin{defn}[$P^*$-parabolic neighborhoods] \label{def:pnei}
Suppose that $(x_0, t_0) \in M \times (-\infty,1)$ and $S, T^{\pm} \geq 0$.
The $P^*$-parabolic neighborhood $P^* (x_0, t_0; S, -T^-, T^+) \subset M \times (-\infty,1)$ is defined as the set of points $(x,t) \in M \times (-\infty,1)$ with $t \in [t_0-T^-, t_0+T^+]$ and 
\[ d^{t_0 - T^-}_{W_1} (v_{x_0, t_0; t_0 - T^-}, v_{x,t; t_0 - T^-}) < S. \]
For any $r>0$, we also define
\begin{align*}
P^*(x_0,t_0;r):=&P^*(x_0,t_0;r,-r^2,r^2) \\
P^{*+}(x_0,t_0;r):=&P^*(x_0,t_0;r,0,r^2) \\
P^{*-}(x_0,t_0;r):=&P^*(x_0,t_0;r,-r^2,0).
\end{align*}
Similar definitions are also made for $P^{\pm}$.
\end{defn}

Some basic properties of $P^*$-parabolic neighborhoods can be found in \cite[Proposition 9.4, Corollary 9.6]{Bam20a}. We state the following containment result from \cite[Proposition 9.4 (d)]{Bam20a}.

\begin{lem} \label{lem:ball}
If $A_1, A_2, T_1^\pm, T_2^\pm \geq 0$ and $(x_1, t_1) \in P^* (x_2, t_2; A_2, -T_2^-, T_2^+)$, then
\[ P^*(x_1, t_1; A_1, -T_1^-, T_1^+) \subset P^*(x_2, t_2; A_1 + A_2, - (T_1^- + T_2^-), T_1^+ + T_2^+). \]
\end{lem}

We immediately have the following result from the distance comparison Lemma \ref{lem:401}.

\begin{lem} \label{lem:dis2}
Given $\delta \in (0,1)$, $t_0 \in (-\infty,1)$, $T^{\pm} \ge 0$ and $S \ge 0$, there exists a constant $C=C(n,A,\delta)>1$ such that
\begin{align*}
P(p,t_0;S,-T^-,T^+) \subset& Q(p,t_0;C(S+1),-T^{-},T^+) \\
Q(p,t_0;S,-T^-,T^+) \subset& P(p,t_0;C(S+1),-T^{-},T^+)
\end{align*}
provided that $-\delta^{-1} \le t_0-T^- \le t_0+T^+ \le 1-\delta$.
\end{lem}

In order to investigate the relation between $P^*$-parabolic neighborhoods and conventional ones, we first prove

\begin{prop} \label{prop:Hcent1}
Given $(x_0,t_0) \in M \times (-\infty,1)$ and $r>0$, suppose $R(x_0,t) \le r^{-2}$ for any $t \in [t_0-r^2,t_0]$. Then
\begin{align}
d_{W_1}^{t_0-r^2}(v_{x_0,t_0;t_0-r^2},\delta_{x_0}) \le C(n,A) r. \label{E502a}
\end{align}
\end{prop}

\begin{proof}
It follows from \cite[Theorem 16]{LW20} that
\begin{align}\label{E502b}
H(x_0,{t_0},x_0,{t_0-r^2}) \ge \frac{1}{(4\pi r^2)^{\frac n 2}} \exp \lc -l_{(x_0,{t_0})}(x_0,{t_0-r^2}) \rc,
\end{align}
From the definition of $l_{(x_0,{t_0})}(x_0,{t_0-r^2})$ we have
\begin{align} \label{E502c}
l_{(x_0,{t_0})}(x_0,{t_0-r^2}) &\le \frac{1}{2r}\int_{t_0-r^2}^{t_0} \sqrt{{t_0}-s}\,R(x_0,s) \,ds \le \frac{1}{3}.
\end{align}

Combining \eqref{E409a} for $\ep=1$, \eqref{E502b} and \eqref{E502c}, it is clear that
\begin{align*}
d_{t_0-r^2}^2(x_0,z) \le C_1 r^2
\end{align*}
for some constant $C_1=C_1(n,A)$, where $(z,t_0-r^2)$ is an $H_n$-center of $(x_0,t_0)$. Therefore,
\begin{align*}
d_{W_1}^{t_0-r^2}(v_{x_0,t_0;t_0-r^2},\delta_{x_0}) \le d_{W_1}^{t_0-r^2}(v_{x_0,t_0;t_0-r^2},\delta_{z})+d_{t_0-r^2}(x_0,z) \le C_2 r,
\end{align*}
where we have used \eqref{eq:hdis} and $C_2:=\sqrt{H_n}+\sqrt{C_1}$.
\end{proof}

\begin{rem} \label{rem:Hcent}
From the proof, we conclude that \eqref{E502a} also holds for a constant $C=C(n,A,\alpha)$ if we assume
\begin{align*}
R(x_0,t) \le \frac{\alpha}{r^2(t_0-t)}
\end{align*}
for some $\alpha>0$ and any $t \in [t_0-r^2,t_0]$.
\end{rem}

\begin{cor} \label{cor:Hcent1}
For any $s_0<t_0<1$, we have
\begin{align}
d_{W_1}^{s_0}(v_{p,t_0;s_0},\delta_{p}) \le C(n,A) \sqrt{t_0-s_0}. \label{E502xa}
\end{align}
\end{cor}
\begin{proof}
From the self-similarity of the flow, we know that
\begin{align*}
R(p,t)=\frac{R(p,0)}{1-t} \le \frac{n}{2(1-t)} \le \frac{n}{2(t_0-t)}
\end{align*}
for any $t<t_0$. Therefore, the conclusion follows from Proposition \ref{prop:Hcent1} and Remark \ref{rem:Hcent}.
\end{proof}

\begin{prop} \label{prop:Hcent1a}
Given $\delta \in (0,1)$, $t_0 \in (-\infty,1)$, $T^{\pm} \ge 0$ and $S \ge 0$, there exists a constant $C=C(n,A,\delta)>1$ such that
\begin{align*}
Q(p,t_0;S,-T^-,T^+) \subset P^*(p,t_0;S+C,-T^{-},T^+)
\end{align*}
provided that $t_0-T^- \ge -\delta^{-1}$.
\end{prop}

\begin{proof}
For any $(x,t) \in Q(p,t_0;S,-T^-,T^+)$, we have
\begin{align}
& d^{t_0-T^-}_{W_1}(v_{p,t_0;t_0-T^-},v_{x,t;t_0-T^-}) \notag \\
\le & d^{t_0-T^-}_{W_1}(v_{p,t;t_0-T^-},v_{x,t;t_0-T^-})+d^{t_0-T^-}_{W_1}(v_{p,t_0;t_0-T^-},v_{p,t;t_0-T^-}) \notag \\
\le & d_t(x,p)+d^{t_0-T^-}_{W_1}(v_{p,t_0;t_0-T^-},v_{p,t;t_0-T^-}) \le S+d^{t_0-T^-}_{W_1}(v_{p,t_0;t_0-T^-},v_{p,t;t_0-T^-}), \label{E502xxa}
\end{align}
where we have used Proposition \ref{prop:302}. In addition, it follows from Corollary \ref{cor:Hcent1} that
\begin{align}
d^{t_0-T^-}_{W_1}(v_{p,t_0;t_0-T^-},v_{p,t;t_0-T^-}) \le d^{t_0-T^-}_{W_1}(v_{p,t_0;t_0-T^-},\delta_p)+d^{t_0-T^-}_{W_1}(v_{p,t;t_0-T^-},\delta_p) \le C(n,A,\delta). \label{E502xxb}
\end{align}
Therefore, the conclusion follows from \eqref{E502xxa} and \eqref{E502xxb}.
\end{proof}

Next, we recall the following version of the local distance distortion estimate, which can be proved almost exactly as~\cite[Theorem 18]{LW20}; see also~\cite[Section 4.3]{CW17B},~\cite[Theorem 3.1]{CW19} and~\cite[Theorem 1.1]{BZ17}.

\begin{lem} \label{lem:discom}
Given $(x_0,t_0) \in M \times (-\infty,1)$ and $r>0$, suppose $R \le r^{-2}$ on $P^-(x_0,t_0;r)$ \emph{(}resp. $P(x_0,t_0;r)$\emph{)}. Then
\begin{align*}
\rho_1 d_t(x,x_0) \le d_{t_0}(x,x_0) \le \rho_1^{-1} d_t(x,x_0)
\end{align*}
if $d_{t_0}(x,x_0) \le \rho_1 r$ and $t \in [t_0-(\rho_1 r)^2,t_0)$ \emph{(}resp. $t \in [t_0-(\rho_1 r)^2,t_0+(\rho_1 r)^2] \cap (-\infty,1)$\emph{)}, where $\rho_1=\rho_1(n,A) \in (0,1)$. In particular, 
\begin{align} \label{E503a}
\quad P^-(x_0,t_0;\rho_1^2 r)\subset Q^-(x_0,t_0; \rho_1 r) \subset P^-(x_0,t_0;r)
\end{align}
\begin{align} \label{E503aa}
\lc \text{resp.} \quad P(x_0,t_0;\rho_1^2 r)\subset Q(x_0,t_0; \rho_1 r) \subset P(x_0,t_0;r) \rc.
\end{align}
\end{lem}

Thanks to Proposition \ref{prop:Hcent1} and Lemma \ref{lem:discom}, we have the following result.

\begin{prop} \label{prop:com1}
There exists a constant $\rho_2=\rho_2(n,A) \in (0,1)$ satisfying the following property.
Given $(x_0,t_0) \in M \times (-\infty,1)$ and $r>0$, suppose that $R \le r^{-2}$ on $P(x_0,t_0;r,-(\rho_2 r)^2)$ \emph{(}resp. $P(x_0,t_0;r,-(\rho_2 r)^2,(\rho_2 r)^2)$\emph{)}. Then
\begin{align} \label{E503b}
P^-(x_0,t_0; \rho_2 r) \subset P^{*}(x_0,t_0;r, -(\rho_2 r)^2,0)
\end{align}
\begin{align} \label{E503bb}
\lc \text{resp.} \quad P(x_0,t_0; \rho_2 r) \subset P^*(x_0,t_0;r, -(\rho_2 r)^2,(\rho_2 r)^2) \rc.
\end{align}
\end{prop}

\begin{proof}
In the proof, all constants $C_i>1$ depend on $n$ and $A$ and $\rho_1$ is from Lemma \ref{lem:discom}. We only prove \eqref{E503b}, and the proof of \eqref{E503bb} is similar. Moreover, we set $0<\tau \ll 1$ to be determined later.

For any $(y,s) \in P^-(x_0,t_0; \tau r)$, it follows from Lemma \ref{lem:discom} that
\begin{align} \label{E503bc}
d_t(y,x_0) \le C_1 \tau r
\end{align}
for any $t \in [t_0-(\tau r)^2,s]$. In particular, $R(y,t) \le r^{-2}$ for any $t \in [t_0-(\tau r)^2,s]$. Therefore, it follows from Proposition \ref{prop:Hcent1} that 
\begin{align} \label{E503bd}
d_{W_1}^{t_0-(\tau r)^2}(v_{y,s;t_0-(\tau r)^2},\delta_{y}) \le C_2 \tau r.
\end{align}
It follows from \eqref{E503bc} and \eqref{E503bd} that
\begin{align*}
&d_{W_1}^{t_0-(\tau r)^2}(v_{y,s;t_0-(\tau r)^2},v_{x_0,t_0;t_0-(\tau r)^2}) \\
\le& d_{W_1}^{t_0-(\tau r)^2}(v_{x_0,t_0;t_0-(\tau r)^2},\delta_{x_0})+d_{W_1}^{t_0-(\tau r)^2}(v_{y,s;t_0-(\tau r)^2},\delta_{y})+d_{t_0-(\tau r)^2}(y,x_0) \le C_3 \tau r<r,
\end{align*}
if $\tau$ is sufficiently small. From this, it is immediate that \eqref{E503b} holds for small $\rho_2$.
\end{proof}
Now, we prove

\begin{prop} \label{prop:Hcent2}
Given $\delta \in (0,1)$, $t_0 \in (-\infty,1)$, $T^{\pm} \ge 0$ and $S \ge 0$, there exists a constant $C=C(n,A,\delta)>1$ such that
\begin{align} \label{E503xxa}
P^*(p,t_0;S,-T^-,T^+) \subset Q(p,t_0; \sqrt{2} S+C,-T^{-},T^+)
\end{align}
provided that $t_0-T^- \ge -\delta^{-1}$. In particular, it implies that $P^*(p,t_0;S,-T^-,T^+)$ is precompact in $M \times (-\infty,1)$ if $t_0+T^+<1$.
\end{prop}
\begin{proof}
In the proof, all constants $C_i>1$ depend on $n, A$ and $\delta$.
It follows from \eqref{E502xa} that
\begin{align}
d^{t_0-T^-}_{W_1} (\delta_p ,v_{p,t_0;t_0-T^-}) \le C_1. \label{E503xa}
\end{align}
For any $(x_1,t_1) \in P^*(p,t_0;S,-T^-,T^+)$, we assume $(z,t_0-T^-)$ to be an $H_n$-center of $(x_1,t_1)$. By \eqref{E503xa} and the definition of $P^*$ neighborhood, we have
\begin{align}
d_{t_0-T^-} (p,z) \le & d^{t_0-T^-}_{W_1} (\delta_p ,v_{p,t_0;t_0-T^-})+d^{t_0-T^-}_{W_1} (v_{p,t_0;t_0-T^-},\delta_z) \notag \\
\le& d^{t_0-T^-}_{W_1} (\delta_z,v_{x_1,t_1;t_0-T^-})+d^{t_0-T^-}_{W_1} (v_{p,t_0;t_0-T^-},v_{x_1,t_1;t_0-T^-})+C_1 \notag \\
\le& S+C_2. \label{E503xb}
\end{align}

Set $v_t=v_{x_1,t_1;t}$ and compute
\begin{align*}
\partial_t \int_M \phi^r \,dv_t=\int_M \square\phi^r \,dv_t.
\end{align*}
By \eqref{E205d}, we have
\begin{align*}
\phi^r(x_1,t_1) \ge& \int_M \phi^r \,dv_{t_0-T^-}-C(n)r^{-1}(t_1-t_0+T^-) 
\ge \int_{F \le r} 1 \,dv_{t_0-T^-}-C_3 r^{-1}.
\end{align*}
Note that $\phi^r=1$ if $F \leq r$ and $r$ is large.  In light of \eqref{E503xb} and Lemma \ref{L201},  the set $\{F \leq r\}$ contains a large geodesic ball centered at $z$.
Thus by Proposition \ref{prop:304}, the above inequality implies that
\begin{align}
  \phi^r(x_1,t_1) \ge \int_{F \le r} 1 \,dv_{t_0-T^-}-C_3 r^{-1}  \geq \frac12  \label{E503xc}
\end{align}
if $2\sqrt r = S+C_4$. Since $\phi^r$ is supported on $F \le 2r$, we conclude from \eqref{E503xc} that
\begin{align*}
F(x_1,t_1) \le 2r=\frac{(S+C_4)^2}{2}.
\end{align*}
From Lemma \ref{L201} and Lemma \ref{lem:401}, we immediate conclude that
\begin{align*}
d_{t_1}(p,x_1) \le \sqrt{2}S+C_5.
\end{align*}
Now, the last conclusion follows from \eqref{E503xxa} and Lemma \ref{lem:dis2}.
\end{proof}

\begin{cor} \label{cor:501}
Given $(x_0, t_0) \in M \times (-\infty,1)$ and $S, T^{\pm} \geq 0$, $P^* (x_0, t_0; S, -T^-, T^+)$ is precompact in $M \times (-\infty,1)$ if $t_0+T^+<1$.
\end{cor}

\begin{proof}
It is clear that $(x_0,t_0) \in P^* (p, t_0; S', -1,0)$ for some large $S'>0$. Therefore, it follows from Lemma \ref{lem:ball}
\begin{align*}
P^* (x_0, t_0; S, -T^-, T^+) \subset P^* (p, t_0; S+S', -(1+T^-),T^+).
\end{align*}
Therefore, the conclusion follows from Proposition \ref{prop:Hcent2}.
\end{proof}
Next, we recall the following existence of the local cutoff function from \cite[Theorem 1.3]{BZ17}.
\begin{prop} \label{prop:cutoff}
Given $(x_0,t_0) \in M \times (-\infty,1)$ and $r>0$, there exists a constant $\rho_3=\rho_3(n,A) \in (0,1)$ satisfying the following property.

Suppose $R \le r^{-2}$ on $P(x_0,t_0;r,0,-\tau)$ with $0<\tau \le (\rho_3 r)^2$. Then there exists a function $\varphi \in C^{\infty}(M \times [t_0-\tau,t_0])$ with the following properties:
\begin{enumerate}[label=\textnormal{(\alph{*})}]
\item $0 \le \varphi<1$ on $M \times [t_0-\tau,t_0]$.

\item $\varphi>\rho_3$ on $P(x_0,t_0;\rho_3 r,0,-\tau)$.

\item $\varphi=0$ outside $P(x_0,t_0;r,0,-\tau)$.

\item $|\na \varphi| \le r^{-1}$ and $|\partial_t \varphi|+|\Delta \varphi| \le r^{-2}$.

\item $\square \varphi \le 0$ on $M \times [t_0-\tau,t_0]$.
\end{enumerate}
\end{prop}
\begin{proof}
We sketch the proof for readers' convenience. In \cite[Theorem 1.3]{BZ17}, $\varphi$ is constructed as the smoothing of $\psi^2$, where 
\begin{align*}
\psi(x,t):=c_1\max \{K(x,t)-c,0\}
\end{align*}
for some constants $c,c_1>0$ on $U \times [t_0-\tau,t_0]$ for some open set $U \subset B_{t_0}(x_0,r)$, where $\psi=0$ on $\partial U \times [t_0-\tau,t_0]$ and can be extended to be $0$ outside $U \times [t_0-\tau,t_0]$. Here, $K(x,t)=H(x,t,y,s)$ for some appropriate $(y,s)$ such that $K(x_0,t_0) \ge (4\pi(t_0-s))^{-\frac n 2}e^{-n/2}$ and $t_0-s$ is sufficiently small. 

The estimates of (a)-(d) follow from \cite[Lemma 20]{LW20}. From the definitions of $\psi$ and $\varphi$, it is clear that (e) also holds. 
\end{proof}
Next, we prove
\begin{prop} \label{prop:Hcent3}
There exists a constant $\rho_4=\rho_4(n,A) \in (0,1)$ satisfying the following property.
Given $(x_0,t_0) \in M \times (-\infty,1)$ and $r>0$, suppose that $R \le r^{-2}$ on $P(x_0,t_0;r,-(\rho_4 r)^2,0)$ \emph{(}resp. $P(x_0,t_0;r,-(\rho_4 r)^2, (\rho_4r)^2)$\emph{)}. Then
\begin{align} \label{E506a}
P^{*-}(x_0,t_0; \rho_4 r) \subset P(x_0,t_0; r, -(\rho_4 r)^2,0)
\end{align}
\begin{align} \label{E506b}
\lc \text{resp.} \quad P^*(x_0,t_0; \rho_4 r) \subset P(x_0,t_0; r, -(\rho_4 r)^2,(\rho_4 r)^2) \rc.
\end{align}
\end{prop}
\begin{proof}
In the proof, all positive constants $C_i >1$ depend only on $n$ and $A$. Moreover, we set $0<\tau \ll 1$ to be determined later.

For any $(y,s) \in P^{*-}(x_0,t_0; \tau r)$, we assume $(z,t_0-(\tau r)^2)$ to be its $H_n$-center. From Proposition \ref{prop:Hcent1}, we have
\begin{align} 
&\quad d_{t_0-(\tau r)^2}(z,x_0)    \notag\\
&\le d^{t_0-(\tau r)^2}_{W_1}(\delta_{x_0},v_{x_0,s;t_0-(\tau r)^2})+ d^{t_0-(\tau r)^2}_{W_1}(v_{y,s;t_0-(\tau r)^2},v_{x_0,s;t_0-(\tau r)^2})+d^{t_0-(\tau r)^2}_{W_1}(\delta_z,v_{y,s;t_0-(\tau r)^2})  \notag \\
&\le C_1 \tau r. \label{E506c}
\end{align}

We assume $\tau<\rho_3$ and consider the cutoff function $\varphi$ constructed in Proposition \ref{prop:cutoff}. If we set $v_t=v_{y,s;t}$, then by direct computation,
\begin{align*}
\partial_t \int_M \varphi \,dv_t=\int_M \square\varphi \,dv_t \ge -r^{-2},
\end{align*}
where we have used Proposition \ref{prop:cutoff}(d). By integration, we have
\begin{align} \label{E506d}
\varphi(y,s) \ge \int_M \varphi \,dv_{t_0-(\tau r)^2}-\tau.
\end{align}
Notice that $\varphi>\rho_3$ on $P(x_0,t_0;\rho_3 r,0,(\tau r)^2)$. Combining this fact with \eqref{E506c} and Proposition \ref{prop:304}, we conclude that if $\tau$ is sufficiently small,
\begin{align*}
\varphi(y,s) \ge \int_M \varphi \,dv_{t_0-(\tau r)^2}-\tau \ge \frac{\rho_3}{2}>0.
\end{align*}
On the other hand, since $\varphi=0$ outside $P(x_0,t_0;r,-(\tau r)^2)$, we conclude that
\begin{align*}
d_{t_0}(x_0,y) \le r
\end{align*}
and hence \eqref{E506a} holds.
\end{proof}
Next, we recall the definition of the curvature radius.
\begin{defn}[Curvature radius] \label{def:curv}
For any $(x, t) \in M \times (-\infty,1)$, the curvature radii at $(x,t)$ are defined as
\begin{align*}
r_{\emph{Rm}}(x,t):=&\sup \left\{r>0 \mid |Rm| \le r^{-2} \quad \text{on} \quad P(x,t;r) \right \}, \\
r_{\emph{Rm}}^{-}(x,t):=&\sup \left\{r>0 \mid |Rm| \le r^{-2} \quad \text{on} \quad P^-(x,t;r)\right \}, \\
r_{\emph{Rm}}^{s}(x,t):=&\sup \left \{r>0 \mid |Rm| \le r^{-2} \quad \text{on} \quad B_t(x,r)\right \}. \\
\end{align*}
\end{defn}
It is clear from the definition that $r_{\text{Rm}}(x,t) \le r_{\text{Rm}}^-(x,t) \le r_{\text{Rm}}^s(x,t)$. In addition, it follows from Theorem \ref{thm:volume1} and the pseudolocality theorem \cite[Theorem 24]{LW20} on Ricci shrinkers that there exists a constant $C=C(n,A)>1$ such that
\begin{align} \label{E507a}
r_{\text{Rm}}^-(x,t) \le C r_{\text{Rm}}(x,t).
\end{align}

We are in a position to obtain the following $\ep$-regularity theorem; see \cite[Theorem 10.2]{Bam20a}.

\begin{thm}[$\ep$-regularity] \label{thm:epr}
There exists a small constant $\ep=\ep(n)>0$ satisfying the following property.
Given $(x,t) \in M \times (-\infty,1)$ and $r>0$, suppose that $\NN_{(x,t)}(r^2) \ge -\ep$, then $r_{\emph{Rm}}(x,t) \ge \ep r$.
\end{thm}

\begin{proof}
We only sketch the proof as the details can be found in \cite[Theorem 10.2]{Bam20a}. The key step is a point-picking argument in the spacetime with respect to the curvature radius $r_{\text{Rm}}$. More precisely, one needs to show that for any $A>0$ with $10Ar_{\text{Rm}}(x,t) \le 1/2$, there exists a point $(x',t') \in P^{*-}(x',t';10Ar_{\text{Rm}}(x,t))$ such that $r_{\text{Rm}}(x',t') \le r_{\text{Rm}}(x,t)$ and $r_{\text{Rm}} \ge r_{\text{Rm}}(x',t')/10$ on $P^{*-}(x',t';Ar_{\text{Rm}}(x',t'))$. Otherwise, one can iteratively pick a sequence of spacetime points $(x_i,t_i)$ in a compact set of $M \times (-\infty,1)$ satisfying $r_{\text{Rm}}(x_i,t_i) \to 0$. 
In light of Lemma~\ref{lem:ball}, all $(x_i,t_i)$ fall into a given  $P^{*-}$-parabolic neighborhood, which is precompact by Corollary \ref{cor:501}. 
Note that the curvature radius of $(x_i, t_i)$ shrinks by a definite portion in each step, the bounded geometry of a compact set implies that the process must terminate in finite steps, say $(x_k, t_k)=(x',t')$.
Such choice of $(x', t')$ guarantees that it has almost maximal curvature radius in spacetime neighborhood. 
Notice that similar point-picking arguments can be found in \cite[Theorem 10.1]{Pe1} and \cite[Proposition 3.43]{CW17B}.

If the $\ep$-regularity theorem fails, we could obtain a sequence of pointed Ricci flows such that $r_{\text{Rm}}=1$ at the base points after the point-picking and appropriate rescalings.  
Since nearby points have curvature radii uniformly bounded from below, the sequence converges smoothly to a limit Ricci flow which is the Euclidean spacetime by the assumption of the Nash entropy. 
Therefore, $r_{\text{Rm}}=1$ must be violated and we obtain a contradiction.
\end{proof}

Using the $\ep$-regularity theorem, one immediately has the following gap property, following the same proof of \cite[Theorem 3]{LW20}.
\begin{cor} \label{cor:502}
Suppose $(M^n,g,f,p)$ is a non-flat Ricci shrinker. Then
\begin{align*}
\NN_{(p,0)}(\ep^{-2}) < -\ep,
\end{align*}
where $\ep$ is the same constant in Theorem \ref{thm:epr}.
\end{cor}
\begin{proof}
Suppose $\NN_{(p,0)}(\ep^{-2}) \ge -\ep$ and $(M,g)$ is a non-flat Ricci shrinker, it follows from Theorem \ref{thm:epr} that $r_{\text{Rm}}(p,0) \ge 1$.
 In particular, it implies that $|Rm(p,t)| \le 1$ for any $t \in [0,1)$. By the self-similarity of the flow, 
 we have $|Rm|(p,0)=|Rm|(p,t)(1-t)$ and hence $|Rm|(p,0)=0$, which contradicts the fact that $R>0$ for non-flat Ricci shrinkers.
\end{proof}
We conclude this section by stating the following two results, whose proofs are more or less standard. See~\cite[Theorem 10.3, Theorem 10.4]{Bam20a}.
\begin{thm} \label{thm:epr2}
For any $\ep > 0$ there is a $\delta=\delta (\ep) > 0$ such that the following holds.
Given $(x,t ) \in M \times (-\infty,1)$ and $r > 0$, if $\NN_{x,t} (r^2) \geq - \delta$, then 
\[ |Rm| \leq \ep r^{-2} \qquad \text{on} \quad P(x,t; \ep^{-1} r,- (1-\ep) r^2, \ep^{-1} r^2). \]
Moreover, we have $\NN^*_{t-r^2} \geq - \ep$ on $P(x,t; \ep^{-1} r, - (1-\ep) r^2)$.
\end{thm}

\begin{thm} \label{thm:epr3}
For any $\ep > 0$ and $Y < \infty$ there is a $\delta=\delta(\ep, Y) > 0$ such that the following holds.
Given $(x,t ) \in M \times (-\infty,1)$ and $r > 0$, suppose that $|Rm| \leq r^{-2}$ on $P^-(x,t; r)$ and $\NN_{x,t} (r^2) \geq - Y$.
Then $\NN_{x,t} ( \delta r^2) \geq - \ep$.
\end{thm}

\section{Metric flows and $\IF$-convergence}
In previous sections, we have generalized (or slightly improved) the theorems and tools in \cite{Bam20a}. 
Notice that these results also hold for Ricci flows induced by Ricci shrinkers (cf. Definition~\ref{dfn:RA19_2}) since most of them are scaling-invariant. In a few cases, one needs to modify the assumptions correspondingly. For instance, the conditions in Theorem \ref{thm:lower} and Theorem \ref{thm:heatupper3} need to be changed to $-\delta^{-1} \lambda \le t <s \le (1-\delta) \lambda$ and $d_t(x,p) \le K \lambda^{1/2}$, if the Ricci flow associated with a Ricci shrinker is parabolically rescaled by $\lambda>0$.

Based on these results and techniques, one can generalize the theory of $\IF$-convergence in \cite{Bam20b} and \cite{Bam20c} from compact Ricci flows to the setting of Ricci flows induced by Ricci shrinkers. 

Notice that the results in \cite{Bam20b} and \cite{Bam20c} are already generalized by Bamler to Ricci flows with complete time-slices and bounded curvature on compact time-intervals (cf. \cite{Bam21}). In \cite[Appendix A]{Bam21}, some issues in the non-compact case are addressed and can be resolved similarly in the setting of Ricci shrinkers by the results and techniques developed in previous sections. For instance, by Theorem \ref{thm:lower} and Theorem \ref{thm:heatupper3}, it is known that the conjugate heat kernel decays exponentially and the function $b$ induced by the conjugate heat kernel increases quadratically (cf. Remark \ref{rem:poten}). Therefore, the weak splitting maps (cf. \cite[Definition 5.6]{Bam20c}) constructed in \cite[Section 10]{Bam20c} have at most quadratic spatial growth. Moreover, it follows from \cite[Proposition 12.1, Remark 12.3]{Bam20c} that one can construct a bounded strong splitting map with bounded gradient from a given weak splitting map.

At various places in \cite{Bam20c}, one also needs to consider integral $\int u\phi^r $ instead of $\int u $, and take the limit for $r \to \infty$ after all the estimates (e.g., $u=\square |\omega_l|$ in \cite[Lemma 17.37]{Bam20c}). This technique has already appeared multiple times in previous sections. As a showcase, we generalize the integral estimates in \cite[Section 6]{Bam20c} to Ricci flows associated with Ricci shrinkers in Appendix \ref{app:A}. These estimates are frequently used in \cite{Bam20c} and are of independent interest.

Now, we recall the following definition of the metric flow from \cite[Definition 3.2]{Bam20b}.

\begin{defn}[Metric flow] \label{def:mf}
Let $I \subset \R$ be a subset.
A metric flow over $I$ is a tuple of the form 
\begin{equation*}
(\XX , \tf, (d_t)_{t \in I} , (v_{x;s})_{x \in \XX, s \in I, s \leq \tf (x)}) 
\end{equation*}
with the following properties:
\begin{enumerate}[label=(\arabic*)]
\item $\XX$ is a set consisting of points.

\item $\tf : \XX \to I$ is a map called time-function.
Its level sets $\XX_t := \tf^{-1} (t)$ are called time-slices and the preimages $\XX_{I'} := \tf^{-1} (I')$, $I' \subset I$, are called time-slabs.
\item $(\XX_t, d_t)$ is a complete and separable metric space for all $t \in I$.

\item $v_{x;s}$ is a probability measure on $\XX_s$ for all $x \in \XX$, $s \in I$, $s \leq \tf (x)$. For any $x \in \XX$ the family $(v_{x;s})_{s \in I, s \leq \tf (x)}$ is called the conjugate heat kernel at $x$.

\item $v_{x; \tf (x)} = \delta_x$ for all $x \in \XX$.

\item For all $s, t \in I$, $s< t$, $T \geq 0$ and any measurable function $u_s : \XX_s \to [0,1]$ with the property that if $T > 0$, then $u_s = \Phi \circ f_s$ for some $T^{-1/2}$-Lipschitz function $f_s : \XX_s \to \R$ \emph{(}if $T=0$, then there is no additional assumption on $u_s$\emph{)}, the following is true. 
The function
\begin{equation*}
u_t :\XX_t \longrightarrow \R, \qquad x \longmapsto \int_{\XX_s} u_s \, dv_{x;s} 
\end{equation*}
is either constant or of the form $u_t = \Phi \circ f_t$, where $f_t : \XX_t \to \R$ is $(t-s+T)^{-1/2}$-Lipschitz. Here, $\Phi$ is given by \eqref{eq_erf}.

\item For any $t_1,t_2,t_3 \in I$, $t_1 \leq t_2 \leq t_3$, $x \in \XX_{t_3}$ we have the reproduction formula
\[ v_{x; t_1} = \int_{\XX_{t_2}} v_{\cdot; t_1} dv_{x; t_2}, \]
meaning that for any Borel set $S\subset \XX_{t_1}$
\[ v_{x;t_1} (S) = \int_{\XX_{t_2}} v_{y ; t_1} (S) dv_{x; t_2}(y). \]
\end{enumerate}
\end{defn}

Given a metric flow $\XX$ over $I$, we recall the following definitions from \cite[Definition 3.20, 3.30]{Bam20b}.

\begin{defn}[Conjugate heat flow] 
A family of probability measures $(\mu_t \in \mathcal{P} (\XX_t))_{t \in I'}$ over $I' \subset I$ is called a conjugate heat flow if for all $s, t \in I'$, $s \leq t$ we have
\begin{equation*}
\mu_s = \int_{\XX_t} v_{x;s} \, d\mu_t (x). 
\end{equation*}
\end{defn}

\begin{defn}[$H$-Concentration] 
Given a constant $H>0$, a metric flow $\XX$ is called $H$-concentrated if for any $s \leq t$, $s,t \in I$, $x_1, x_2 \in \XX_t$
\begin{equation*}
\emph{\Var} (v_{x_1; s}, v_{x_2; s} ) \leq d^2_t (x_1, x_2) + H (t-s).
\end{equation*}
\end{defn}

Next, we recall the definition of the metric flow pair from \cite[Definition 5.1, 5.2]{Bam20b}. Roughly speaking, two metric flow pairs are equivalent if they are the same in the metric measure sense almost everywhere.

\begin{defn}[Metric flow pair]
A pair $(\XX, (\mu_t)_{t \in I'})$ is called a metric flow pair over $I \subset \R$ if:
\begin{enumerate}
\item $I' \subset I$ with $|I \setminus I'| = 0$.
\item $\XX$ is a metric flow over $I'$.
\item $(\mu_t)_{t \in I'}$ is a conjugate heat flow on $\XX$ with $\text{supp}\,\mu_t = \XX_t$ for all $t \in I'$.
\end{enumerate}
If $J \subset I'$, then we say that $(\XX, (\mu_t)_{t \in I})$ is fully defined over $J$. We denote by $\IF_I^J$ the set of equivalence classes of metric flow pairs over $I$ that are fully defined over $J$. Here, two metric flow pairs $(\XX^i, (\mu^i_t)_{t \in I^{\prime, i}})$, $i = 1,2$, that are fully defined over $J$ are equivalent if there exists an isometry $\phi : \XX^1_{I' } \to \XX^2_{ I' }$ \emph{(cf. \cite[Definition 3.13]{Bam20b})} such that $|I^{\prime, 1} \setminus I'| = |I^{\prime, 2} \setminus I'| = 0$, $(\phi_t )_* \mu^1_t = \mu^2_t$ for all $t \in I'$ and $J \subset I'$.
\end{defn}

We will only consider $I:=(-\infty,0]$ for simplicity. Then for any pointed Ricci flow $(M^n,g(t),x_0)_{t \in I}$ induced by a Ricci shrinker, one can define $(\XX, (\mu_t)_{t \in I})$ as follows. 
\begin{align} \label{eq:example}
\lc \XX:=M \times (I\setminus \{0\}) \sqcup x_0 \times \{0\}, \t:=\text{proj}_{I}, (d_t)_{t \in I}, (v_{x,t;s})_{(x,t) \in M \times I, s\in I, s\le t}, \mu_t:=v_{x_0,0;t}\rc.
\end{align}

Then we have

\begin{prop}
The pair $(\XX, (\mu_t)_{t \in I})$ defined in \eqref{eq:example} is an $H_n$-concentrated metric flow pair that is fully defined over $I$.
\end{prop}

\begin{proof}
The conditions (1)-(5) in the definition of the metric flow can be easily checked. Condition (6) follows from \eqref{thm:T303} and (7) from the semigroup property \eqref{E301}. The metric flow is $H_n$-concentrated by Proposition \ref{prop:302}.
\end{proof}

Next, we recall the definition of a correspondence between metric flows; see \cite[Definition 5.4]{Bam20b}.

\begin{defn}[Correspondence]
Let $(\XX^i, (\mu^i_t)_{t \in I^{\prime,i}})$ be metric flows over $I$, indexed by some $i \in \mathcal{I}$.
A correspondence between these metric flows over $I''$ is a pair of the form
\begin{equation*}
\CF := \big( (Z_t, d^Z_t)_{t \in I''},(\varphi^i_t)_{t \in I^{\prime\prime,i}, i \in \mathcal{I}} \big), 
\end{equation*}
where:
\begin{enumerate}
\item $(Z_t, d^Z_t)$ is a metric space for any $t \in I''$.
\item $I^{\prime\prime,i} \subset I'' \cap I^{\prime,i}$ for any $i \in \mathcal{I}$.
\item $\varphi^i_t : (\XX^i_t, d^i_t) \to (Z_t, d^Z_t)$ is an isometric embedding for any $i \in \mathcal{I}$ and $t \in I^{\prime\prime,i}$.
\end{enumerate}
If $J \subset I^{\prime\prime,i}$ for all $i \in \II$, we say that $\CF$ is fully defined over $J$.
\end{defn}

Given a correspondence, one can define the $\mathbb F$-distance, see \cite[Definition 5.6, 5.8]{Bam20b}.

\begin{defn}[$\IF$-distance within correspondence] \label{FDWC}
We define the $\IF$-distance between two metric flow pairs within $\CF$ (uniform over $J$),
\[ d_{\IF}^{\,\CF, J} \big( (\XX^1, (\mu^1_t)_{t \in I^{\prime,1}}), (\XX^2, (\mu^2_t)_{t \in I^{\prime,2}}) \big), \] 
to be the infimum over all $r > 0$ with the property that there is a measurable subset $E \subset I''$ with
\[ J \subset I'' \setminus E \subset I^{\prime\prime,1} \cap I^{\prime\prime,2} \]
and a family of couplings $(q_t)_{t \in I'' \setminus E}$ between $\mu^1_t, \mu^2_t$ such that:
\begin{enumerate}[label=(\arabic*)]
\item $|E| \leq r^2$.
\item For all $s, t \in I'' \setminus E$, $s \leq t$, we have
\[ \int_{\XX^1_t \times \XX_t^2} d_{W_1}^{Z_s} ( (\varphi^1_s)_* \nu^1_{x^1; s}, (\varphi^2_s)_* \nu^2_{x^2; s} ) dq_t (x^1, x^2) \leq r. \]
\end{enumerate}
\end{defn}

Notice that (2) above implies that for any $t \in I'' \setminus E$,
\begin{align} \label{App:GW}
d_{GW_1} \big( (\XX^1_t, d^1_t, \mu^1_t), (\XX^2_t, d^2_t, \mu^2_t) \big)
\leq d_{W_1}^{Z_t} ( (\varphi^1_t)_* \mu^1_t , (\varphi^2_t)_* \mu^2_t) \le r.
\end{align}
Here, $d_{GW_1}$ denotes the Gromov-$W_1$-Wasserstein distance, see \cite[Definition 2.11]{Bam20b} for the precise definition.

\begin{defn}[$\IF$-distance] 
The $\IF$-distance between two metric flow pairs (uniform over $J$),
\[ d_{\IF}^{ J} \big( (\XX^1, (\mu^1_t)_{t \in I^{\prime,1}}), (\XX^2, (\mu^2_t)_{t \in I^{\prime,2}}) \big), \] 
is defined as the infimum of
\[ d_{\IF}^{\,\CF, J} \big( (\XX^1, (\mu^1_t)_{t \in I^{\prime,1}}), (\XX^2, (\mu^2_t)_{t \in I^{\prime,2}}) \big), \] 
over all correspondences $\CF$ between $\XX^1, \XX^2$ over $I''$ that are fully defined over $J$.
\end{defn}

With all those definitions, it can be proved (cf. \cite[Theorem 5.13, 5.26]{Bam20b}) that $(\IF_I^J,d_{\IF}^J)$ is a complete metric space, with possible infinite distances.

In addition, $\IF$-convergence implies $\IF$-convergence within a correspondence; see \cite[Theorem 6.12]{Bam20b}. More precisely,

\begin{thm} 
Let $(\XX^i, (\mu^i_t)_{t \in I^{\prime,i}})$, $i \in \IN \cup \{ \infty \}$, be metric flow pairs over $I $ that are fully defined over some $J \subset I$.
Suppose that for any compact subinterval $I_0 \subset I$
\[ d_{\IF}^{J \cap I_0} \big( (\XX^i , (\mu^i_t)_{t \in I_0 \cap I^{\prime,i}}), (\XX^\infty , (\mu^\infty_t)_{t \in I_0\cap I^{\prime,\infty}}) \big) \to 0. \] 
Then there is a correspondence $\CF$ between the metric flows $\XX^i$, $i \in \IN \cup \{ \infty \}$, over $I$ such that
\begin{equation*} 
(\XX^i, (\mu^i_t)_{t \in I^{\prime,i}}) \xrightarrow[i \to \infty]{\quad \IF, \CF, J \quad} (\XX^\infty, (\mu^\infty_t)_{t \in I^{\prime,\infty}})
\end{equation*}
on compact time intervals, in the sense that
\begin{equation*} 
d_{\IF}^{\,\CF, J \cap I_0} \big( (\XX^i, (\mu^i_t)_{t \in I_0 \cap I^{\prime,i}}), (\XX^\infty, (\mu^\infty_t)_{t \in I_0\cap I^{\prime,\infty}}) \big) \to 0
\end{equation*}
for any compact subinterval $I_0 \subset I$.
\end{thm}

For a sequence of Ricci flows $(M^n_i,g_i(t)),x_i)_{t \in I}$ induced by Ricci shrinkers, one can use the $\IF$-compactness theorem for metric flow pairs \cite[Corollary 7.5, Theorem 7.6]{Bam20b} to obtain the following result. 

\begin{thm}[$\IF$-compactness] \label{thm:601}
Let $(M_i^n,g_i(t),x_i)_{t \in I}$ be a sequence of pointed Ricci flows induced by Ricci shrinkers with the corresponding metric flow pairs $(\XX^i,(\mu_t^i)_{t \in I})$ as described in \eqref{eq:example}.

After passing to a subsequence, there exists an $H_n$-concentrated metric flow pair $(\XX^\infty, (\mu^\infty_t)_{t \in I})$ for which $\XX^\infty$ is future continuous in the sense of \emph{\cite[Definition 4.25]{Bam20b}} such that the following holds.
There is a correspondence $\CF$ between the metric flows $\XX^i$, $i \in \IN \cup \{ \infty \}$, over $I$ such that on compact time-intervals
\begin{equation} \label{Fconv}
(\XX^i, (\mu^i_t)_{t \in I}) \xrightarrow[i \to \infty]{\quad \IF, \CF \quad} (\XX^\infty, (\mu^\infty_t)_{t \in I}) .
\end{equation}
Moreover, the convergence (\ref{Fconv}) is uniform over any compact $J \subset I$ that only contains times at which $\XX^\infty$ is continuous, see \emph{\cite[Definition 4.25]{Bam20b}}. Notice that $\XX^\infty$ is continuous everywhere except possibly at a countable set of times, by \emph{\cite[Corollary 4.35]{Bam20b}}.
\end{thm}

We sketch the main ideas and steps of Theorem \ref{thm:601} modulo all technical details.
\begin{enumerate} 

\item One needs a characterization of the compactness for a subset in $(\mathbb M, d_{GW_1})$, the isometry classes of all metric measure space $(X,d,\mu)$, where $\mu \in \mathcal P(X)$ with $\text{supp}\,\mu=X$ and $d_{GW_1}$ denotes the Gromov-$W_1$-Wasserstein distance (cf. \cite[Definition 2.11]{Bam20b}). Let $\IM_r(V,b) \subset \IM$ be the subset consisting of $(X,d,\mu)$ satisfying
\begin{equation} 
\text{Var}(\mu) \le Vr^2 \quad \text{and} \quad \mu\lc \left\{ x\in X \mid \mu(D(x,\ep r))<b(\ep) \right\} \rc \le \ep, \quad \forall \ep\in (0,1]. \label{E601}
\end{equation}
Here, $V,r$ are two positive constants and $b:(0,1] \to (0,1]$ is a function. Moreover, $D(x,\ep r)$ denotes a closed ball with center $x$ and radius $\ep r$. It is proved by \cite[Theorem 2.27]{Bam20b} that $\IM_r(V,b)$ is compact.

\item For any metric flow pair $(\XX, (\mu_t)_{t \in I})$ defined in \eqref{eq:example}. It is clear by $H_n$-concentration that $\text{Var}(\mu_t) \le H_n |t|$. It can be proved (cf. \cite[Proposition 4.1]{Bam20b} with $\tau=\frac{\ep^3}{8H_n}$) that for any $t<0$,
\begin{equation} 
(\XX_t,d_t,\mu_t) \in \IM_r(V,b), \label{E602}
\end{equation}
where $V=1/8$, $r=\sqrt{8H_n |t|}$, $b(\ep)=\Phi(\ep^{-2} \sqrt{8H_n})/2$ and $\Phi$ is given by \eqref{eq_erf}. The proof of \eqref{E602} uses Definition \ref{def:mf}(6)(7) in an essential way. Therefore, for any $t \le 0$, $(\XX^i,d^i_t,\mu^i_t)$ subconverges in $GW_1$ to a limit metric measure space.

\item To compare different time-slices of $(\XX, (\mu_t)_{t \in I})$, one considers the function
\begin{equation} 
D(t):=\int_{\XX_t} \int_{\XX_t} d_t \, d\mu_t d\mu_t \label{E603}
\end{equation}
for $t \in I$. It is not hard to prove (cf. \cite[Lemma 4.7]{Bam20b}) that for any $s \le t \in I$,
\begin{equation} 
-\sqrt{H_n(t-s)} \le D(t)-D(s) \le \sqrt{\text{Var}(\mu_t)-\text{Var}(\mu_s)+H_n(t-s)}+2\sqrt{H_n(t-s)}. \label{E604}
\end{equation}
It follows immediately from \eqref{E604} that $D(t)$ is continuous on a complement of a countable subset of $I$. In addition, it is proved (cf. \cite[Theorem 4.31]{Bam20b}) that for any $t_0 \le 0$, $D(t)$ is continuous at $t_0$ is equivalent to the statement that $(\XX_t,d_t,\mu_t)$ is continuous at $t_0$ in the $GW_1$ sense. In this case, one can construct an isometric embedding of $(\XX_t,d_t)$ and $(\XX_{t_0},d_{t_0})$ into a metric space $(Z_t,d^Z_t)$ with an explicit coupling $q_t$ between $\mu_t$ and $\mu_{t_0}$ for $t$ close to $t_0$. Therefore, one concludes that the metric flow $(\XX_t,\mu_t)$ is continuous on $I$ except at a countable set of times.

\item For the sequence $(\XX^i, (\mu^i_t)_{t \in I})$ in \eqref{Fconv}, we consider the limit
\begin{equation} 
D^{\infty}(t):=\lim_{t \to \infty}D^i(t), \label{E605}
\end{equation}
which exists for any $t \in I$ by taking a subsequence. Indeed, by \eqref{E604}, we may assume that $D^{\infty}(t)$ exists for $t \in I\cap \mathbb Q$ and $D^{\infty}(t)-D^{\infty}(s) \ge -\sqrt{H_n(t-s)}$ for any $s, t \in I\cap \mathbb Q$ with $s \le t$, after taking a subsequence if necessary. Therefore, there exists a countable set $S \subset I$ such that $D^{\infty}$ is continuous on $I \setminus S$, by extending the definition of $D^{\infty}$. Moreover, \eqref{E605} holds for any $t \in I \setminus S$. Now, \eqref{E605} also holds for $t \in S$, by further taking a subsequence.

The $\IF$-convergence of $(\XX^i, (\mu^i_t)_{t \in I})$ can be constructed as follows. We assume $D^{\infty}(t)$ is continuous at $I \setminus S$ for a countable set $S$. For a large $k>0$, we take a compact set $I_1 \subset [-k,0] \setminus S$ so that $|[-k,0] \setminus I_1|$ is small. Then $I_1$ is finitely covered by compact intervals $I_{t_i}$ centered at $t_i \in I_1$ such that $|I_{t_i}|$ and the oscillation of all $D^i$ and $D^{\infty}$ on each $I_{t_i}$ are sufficiently small. By steps 1 and 2 above, one can construct a correspondence $\CF_0$ that is fully defined on the finite set $I_0:=\{t_i\}$ between $\XX^i$, so that
\begin{equation} 
d_{\IF}^{\,\CF_0, I_0} \big( (\XX^i, (\mu^i_t)_{t \in I_0}), (\XX^j, (\mu^j_t)_{t \in I_0}) \big)<\ep \label{E606}
\end{equation}
for any $\ep>0$, if $i,j$ are sufficiently large. Then by using the small oscillation of $D^i$ on $I_{t_i}$, one can extend the correspondence $\CF_0$ to $\CF_1$ over $I_1$ so that $(\XX^i, (\mu^i_t)_{t \in I_1})$ forms a Cauchy sequence over $I_1$ in the sense of \eqref{E606} with respect to $d_{\IF}^{\,\CF_1, I_1}$ (cf. \cite[Lemma 7.24]{Bam20b}). By letting $k \to \infty$ and taking a diagonal sequence, we obtain from the completeness of $(\IF_I,d_{\IF})$ a limit metric flow pair $(\XX^\infty, (\mu^\infty_t)_{t \in I \setminus S})$, which has an extended definition for all $t \in I$ by the future completion (cf. \cite[Section 4.4]{Bam20b}) so that $(\XX^\infty, (\mu^\infty_t)_{t \in I})$ is right continuous for $t \in I$. Notice that Definition \ref{def:mf}(1)-(7) for $(\XX^\infty, (\mu^\infty_t)_{t \in I})$ are inherited from $(\XX^i, (\mu^i_t)_{t \in I})$. In addition, one can construct a correspondence $\CF$ so that
\begin{equation*}
(\XX^i, (\mu^i_t)_{t \in I}) \xrightarrow[i \to \infty]{\quad \IF, \CF \quad} (\XX^\infty, (\mu^\infty_t)_{t \in I})
\end{equation*}
on compact time intervals and the convergence is uniform over the set on which $(\XX^\infty, (\mu^\infty_t)_{t \in I})$ is continuous. Moreover, $(\XX^\infty_{t'}, d^{\infty}_{t'}, \mu^\infty_{t'}) \subset \IM_r(V,b)$ as \eqref{E602} and $\text{Var}(\mu^\infty_t) \le H_n |t|$ for any $t \in I$. Notice that $\XX^\infty_0$ consists of a single point from which $\mu^\infty_t$ is the conjugate heat measure. 
\end{enumerate} 

\begin{rem}
In \emph{\cite[Theorem 7.4]{Bam20b}}, a general compactness for a subset $\IF_{I}^J(H,V,b,r) \subset \IF_I^J$ is proved by the same method as described above.
\end{rem}

It follows from \cite[Theorem 8.2, 8.4]{Bam20b} that the limit metric flow pair $(\XX^\infty, (\mu^\infty_t)_{t \in I})$ obtained in \eqref{Fconv} is a length space for any $t \in I$. 
In general, further geometric information contained in $(\XX^\infty, (\mu^\infty_t)_{t \in I})$ is scarce. 
However, if $(M_i^n,g_i(t))$ are induced by Ricci shrinkers in $\MM(A)$, then, in particular, their Nash entropies are uniformly bounded by Corollary \ref{cor:302}. 
In this case, one obtains much more concrete structure theorem regarding the limit metric flow obtained in \eqref{Fconv}; see \cite[Theorem 2.3, 2.4, 2.5, 2.6, 2.46]{Bam20c} and \cite[Theorem 9.31]{Bam20b}.

\begin{thm} \label{thm:602}
Let $(M_i^n,g_i(t),x_i)_{t \in I}$ be a sequence of pointed Ricci flows induced by Ricci shrinkers in $\mathcal M(A)$ and $(\XX^\infty, (\mu^\infty_t)_{t \in I})$ the limit metric flow pair obtained in Theorem \ref{thm:601}. Then the following properties hold.

\begin{enumerate}[label=(\arabic*)]
\item There exists a decomposition
\begin{equation}
\XX^{\infty}_{0}=\{x_{\infty}\},\quad \XX^{\infty}_{t<0}=\mathcal R \sqcup \mathcal S, \label{eq:rsd}
\end{equation}
such that $\mathcal R$ is given by an $n$-dimensional Ricci flow spacetime $(\mathcal R, \tf, \partial^{\infty}_{\tf}, g^{\infty})$, in the sense of \emph{\cite[Definition 9.1]{Bam20b}} and $\emph{\text{dim}}_{\mathcal M^*}(\mathcal S) \le n-2$, where $\emph{\text{dim}}_{\mathcal M^*}$ denotes the $*$-Minkowski dimension in \emph{\cite[Definition 3.42]{Bam20b}}. Moreover, $\mu^{\infty}_t(\MS_t)=0$ for any $t<0$.

\item Every tangent flow $(\XX',(v_{x_{\infty}';t})_{t \le 0})$ at every point $x \in \XX^{\infty}$ is a metric soliton in the sense of \emph{\cite[Definition 3.57]{Bam20b}}. Moreover, $\XX'$ is the Gaussian soliton iff $x \in \mathcal R$. If $x \in \MS$, the singular set of $(\XX',(v_{x_{\infty}';t})_{t \le 0})$ on each $t <0$ has Minkowski dimension at most $n-4$. In particular, if $n=3$, the metric soliton is a smooth Ricci flow associated with a $3$-dimensional Ricci shrinker. If $n=4$, each slice of the metric soliton is a smooth Ricci shrinker orbifold with isolated singularities.

\item $\mathcal R_t=\mathcal R \cap \XX^{\infty}_t$ is open such that the restriction of $d_t$ on $\mathcal R_t$ agrees with the length metric of $g_t$.

\item The convergence (\ref{Fconv}) is smooth on $\mathcal R$, in the following sense. There exists an increasing sequence $U_1 \subset U_2 \subset \ldots \subset \mathcal R$ of open subsets with $\bigcup_{i=1}^\infty U_i = \mathcal R$, open subsets $V_i \subset M_i \times I$, time-preserving diffeomorphisms $\phi_i : U_i \to V_i$ and a sequence $\ep_i \to 0$ such that the following holds:
\begin{enumerate}[label=(\alph*)]
\item We have
\begin{align*}
\Vert \phi_i^* g^i - g^\infty \Vert_{C^{[\ep_i^{-1}]} ( U_i)} & \leq \ep_i, \\
\Vert \phi_i^* \partial^i_{\tf} - \partial^\infty_{\tf} \Vert_{C^{[\ep_i^{-1}]} ( U_i)} &\leq \ep_i, \\
\Vert w^i \circ \phi_i - w^\infty \Vert_{C^{[\ep_i^{-1}]} ( U_i)} &\leq \ep_i,
\end{align*}
where $g^i$ is the spacetime metric induced by $g_i(t)$, and $w^i$ is the conjugate heat kernel defined by $d\mu^i=w^i dg^i$, $i \in \IN \cup \{ \infty \}$.

\item Let $y_\infty \in \mathcal R$ and $y_i \in M_i \times (-\infty,0)$.
Then $y_i$ converges to $y_\infty$ within $\CF$ \emph{(cf. \cite[Definition 6.18]{Bam20b})} if and only if $y_i \in V_i$ for large $i$ and $\phi_i^{-1} (y_i) \to y_\infty$ in $\mathcal R$.

\item If the convergence (\ref{Fconv}) is uniform at some time $t \in I$, then for any compact subset $K \subset \RR_t$ and for the same subsequence we have
\[ \sup_{x \in K \cap U_i} d^Z_t (\varphi^i_t (\phi_i(x)), \varphi^\infty_t (x) ) \longrightarrow 0. \]
\end{enumerate}
\end{enumerate}
\end{thm}

Theorem~\ref{thm:602} is a flow version of the Cheeger-Colding theory (cf.~\cite{CC},~\cite{CN13} and~\cite{CN15}).  Its proof shares similar strategy as its elliptic counterparts. 
Many concepts also have counterparts. 
For example, tangent flow corresponds to tangent space,  metric soliton corresponds to metric cone.  We recall their definitions.  See \cite[Definition 6.55, 3.57]{Bam20b}.

\begin{defn}[Tangent flow] \label{def:tangentf}
Let $\XX$ be a metric flow over $I$ and $x_0 \in \XX_{t_0}$ a point.
We say that a metric flow pair $(\XX', (v'_{x_{\max};t})_{t \in I})$ is a tangent flow of $\XX$ at $x_0$ if there is a sequence of scales $\lambda_k> 0$ with $\lambda_k \to \infty$ such that for any $T > 0$ the parabolic rescalings \[ \big(\XX^{-t_0, \lambda_k}_{[-T,0]}, (v_{x_0;t}^{-t_0, \lambda_k})_{\lambda_k^{-2} t + t_0 \in I' , t \in [-T,0]}\big) \] $\IF$-converge to $(\XX'_{[-T,0]}, (v'_{x_{\max};t})_{t \in [-T,0]})$.
\end{defn}

\begin{defn}[Metric soliton] \label{def:met_soliton}
A metric flow pair $(\XX, (\mu_t)_{t \in I})$ is called a metric soliton if there is a tuple
\[ \big( X, d, \mu, (v'_{x;t} )_{x \in X; t \leq 0} \big) \]
and a map $\phi : \XX \to X$ such that the following holds:
\begin{enumerate}
\item For any $t \in I$, the map $\phi_t : (\XX_t, d_t, \mu_t) \to (X, \sqrt{t} d, \mu)$ is an isometry between metric measure spaces.
\item For any $x \in \XX_t$, $s \in I$ with $s \leq t$, we have $(\phi_s)_* v_{x;s} = v'_{\phi_t (x); \log (s/t)}$.
\end{enumerate}
\end{defn}
Roughly speaking, a metric soliton is a metric flow pair induced by a metric measure space in a shrinking way. In general, a tangent flow of a metric flow may not be a metric soliton. In the setting of Theorem \ref{thm:602}, every tangent flow of $(\XX^\infty, (\mu^\infty_t)_{t \in I})$ is also an $\IF$-limit of a sequence of Ricci flows induced by Ricci shrinkers in $\MM(A)$ (cf. \cite[Theorem 6.58]{Bam20b}).

Notice that the limit metric flow $(\XX^\infty, (\mu^\infty_t)_{t \in I})$ in \eqref{Fconv} always admits a regular-singular decomposition 
\begin{equation*}
\XX^{\infty}_{t<0}=\mathcal R \sqcup \mathcal S, 
\end{equation*}
so that $\mathcal R$ is given by a Ricci flow spacetime (cf. \cite[Definition 9.1]{Bam20b}). The key point is to control the size of the singular part in the appropriate sense. To avoid the distance distortion at different time-slices, one can redefine the Hausdorff and Minkowski dimensions (denoted by $\mathcal H^*$ and $\mathcal M^*$ respectively) by using the $P^*$-parabolic balls instead of the conventional ones; see \cite[Definition 3.41, 3.42]{Bam20b}.

One can control the size of $\mathcal S$ quantitatively. Let $(M^n,g(t))_{t \in I}$ be the Ricci flow induced by a Ricci shrinker in $\MM(A)$. We fix a point $(x_0,t_0) \in M \times I$ and define $\tau=t_0-t$ and $H(x_0,t_0,\cdot,\cdot)=(4\pi\tau)^{-\frac n 2} e^{-b}$. We next recall the following definitions from \cite[Definition 5.1, 5.5, 5.6, 5.7]{Bam20c}, which indicate the extent to which the local geometry around $(x_0,t_0)$ is a Ricci shrinker, Ricci flat space or splitting off an $\R^k$.

\begin{defn}[Almost self-similarity] \label{def:ss}
Let $(M^n,g(t))_{t \in I}$ be the Ricci flow induced by a Ricci shrinker. The point $(x_0,t_0) \in M \times I$ is called $(\ep, r)$-selfsimilar if the following holds:
\begin{align*}
& \int_{t_0 - \ep^{-1} r^2}^{t_0 - \ep r^2} \int_M \tau \Big| Rc + \nabla^2 b - \frac1{2\tau} g \Big|^2 dv_{x_0, t_0; t} dt \leq \ep, \\
& \int_M \big| \tau (2\triangle b - |\nabla b|^2 + R) + b - n - \NN_{x_0,t_0}(r^2) \big| dv_{x_0, t_0; t}
\leq \ep, \quad \forall t \in [t_0 - \ep^{-1} r^2, t_0 - \ep r^2].
\end{align*}
\end{defn}

\begin{defn}[Almost static] \label{def:static}
The point $(x_0,t_0)$ is called $(\ep, r)$-static if the following holds:
\begin{align*}
& r^2 \int_{t_0 - \ep^{-1} r^2}^{t_0 - \ep r^2} \int_M |Rc|^2 dv_{x_0, t_0; t} dt \le \ep, \\
& r^2\int_M R \, dv_{x_0, t_0;t} \leq \ep, \quad \forall t \in [t_0 - \ep^{-1} r^2, t_0 - \ep r^2].
\end{align*}
\end{defn}

\begin{defn}[Weak splitting] \label{def:wpm}
$(x_0, t_0)$ is called weakly $(k, \ep, r)$-split if there exists a vector-valued function $\vec y = (y_1, \ldots, y_k) : M \times [t_0 - \ep^{-1} r^2, t_0 - \ep r^2] \to \R^k$ with the following properties for all $i,j = 1, \ldots, k$:
\begin{enumerate}[label=(\arabic*)]
\item We have
\[ r^{-1} \int_{t_0 - \ep^{-1} r^2}^{t_0 - \ep r^2} \int_M |\square y_i | dv_{x_0, t_0;t} dt \leq \ep. \] \\
\item We have
\[ r^{-2} \int_{t_0 - \ep^{-1} r^2}^{t_0 - \ep r^2} \int_M |\nabla y_i \cdot \nabla y_j - \delta_{ij} | dv_{x_0, t_0;t} dt \leq \ep. \]
\end{enumerate} 
\end{defn}

\begin{defn}[Strong splitting] \label{def:spm}
$(x_0, t_0)$ is called strongly $(k, \ep, r)$-split if there exists a vector-valued function $\vec y = (y_1, \ldots, y_k) : M \times [t_0 - \ep^{-1} r^2, t_0 - \ep r^2] \to \R^k$ with the following properties for all $i,j = 1, \ldots, k$:
\begin{enumerate}[label=(\arabic*)]
\item $y_i$ solves the heat equation $\square y_i = 0$ on $M \times [t_0 - \ep^{-1} r^2, t_0 - \ep r^2]$.
\item We have
\[ r^{-2} \int_{t_0 - \ep^{-1} r^2}^{t_0 - \ep r^2} \int_M |\nabla y_i \cdot \nabla y_j - \delta_{ij} | dv_{x_0, t_0;t} dt \leq \ep. \]
\item \label{Def_strong_splitting_map_3} For all $t \in [t_0 - \ep^{-1} r^2, t_0 - \ep r^2]$ we have
\begin{equation*}
\int_M y_i \, dv_{x_0 t_0; t} = 0. 
\end{equation*}
\end{enumerate} 
\end{defn}

It can be proved (cf. \cite[Proposition 12.1]{Bam20c}) that if $(x_0,t_0)$ is weakly $(k, \ep, r)$-split, then it is strongly $(k, \delta(\ep), r)$-split. With these definitions, one can consider the following quantitative stratification.

\begin{defn} \label{Def_SS_weak}
For $\ep > 0$ and $0 < r_1 < r_2 \leq \infty$ the effective strata
\[ \td \MS^{\ep, 0}_{r_1, r_2} \subset \td \MS^{\ep, 1}_{r_1, r_2} \subset \td \MS^{\ep,2}_{r_1, r_2} \subset \ldots \subset \td \MS^{\ep,n+2}_{r_1, r_2} \subset M \times I \]
are defined as follows:
$(x',t') \in \td \MS^{\ep, k}_{r_1, r_2}$ if and only for all $r' \in (r_1, r_2)$ none of the following two properties hold:
\begin{enumerate}
\item $(x',t')$ is $(\ep, r')$-selfsimilar and weakly $(k+1, \ep, r')$-split.
\item $(x',t')$ is $(\ep, r')$-selfsimilar, $(\ep, r')$-static and weakly $(k-1, \ep, r')$-split.
\end{enumerate}
\end{defn}

By a delicate choice of the covering by $P^*$-parabolic balls, it can be proved, see \cite[Proposition 11.2]{Bam20c}, that for any $0<\sigma<\ep$, there are points $(x_1, t_1), \ldots, (x_N, t_N) \in \td\MS^{\ep,k}_{\sigma r, \ep r} \cap P^* (x_0, t_0; r)$ with $N \leq C(A, \ep) \sigma^{-k-\ep}$ and
\begin{equation}\label{eq:cover}
\td\MS^{\ep, k}_{\sigma r, \ep r} \cap P^* (x_0, t_0; r) \subset \bigcup_{i=1}^N P^* (x_i, t_i; \sigma r).
\end{equation}
Notice that \eqref{eq:cover} can be regarded as a parabolic version of the covering in~\cite{CN13} by Cheeger and Naber.

On the complement of $\td\MS^{\ep,n-2}_{\sigma r, \ep r}$, the following $\ep$-regularity theorem is proved (cf. \cite[Proposition 17.1]{Bam20c}), which can be viewed as a parabolic analogue of 
Cheeger-Naber's codimension $4$ theorem in \cite{CN15}. Roughly speaking, one needs to rule out the tangent flows which are Ricci-flat, and split off an $\R^{n-3}$.

\begin{prop}\label{prop:codim}
There exists a constant $\ep=\ep(n,A)>0$ such that the following holds.
Let $(M^n,g(t))_{t \in I}$ be the Ricci flow induced by a Ricci shrinker in $\MM(A)$. Suppose that $(x_0, t_0)$ is strongly $(n-1, \ep, r)$-split or strongly $(n-3, \ep, r)$-split and $(\ep, r)$-static.
Then $r_{\emph{Rm}} (x_0, t_0) \geq \ep r$.
\end{prop}

There are many implications of Proposition \ref{prop:codim}. Notice that one has the following decomposition:
\begin{equation*}
\XX^{\infty}_{t<0}=\mathcal R^* \sqcup \mathcal S^*, 
\end{equation*}
where $\mathcal R^* \subset \mathcal R$ is the set of points where the convergence \eqref{Fconv} is smooth as defined in \cite[Section 9.4]{Bam20b}. Since $\MS \subset \MS^*$, one can obtain the estimate of $*$-Minkowski dimension of $\MS$ by that of $\MS^*$ from \eqref{eq:cover} and Proposition \ref{prop:codim} (cf. \cite[Theorem 15.28 (a)]{Bam20c}). Moreover, it can be proved that $\MS^* \cap \XX^{\infty}_t$ has measure $0$ for any $t<0$ (cf. \cite[Theorem 15.28 (b)]{Bam20c}). Therefore, Theorem \ref{thm:602} (1) is obtained.

Since $\MS$ has measure $0$ on each time-slice, one can extend the definition of the Nash entropy on $\XX^{\infty}$. Therefore, the Nash entropy at the base point $x'$ of any tangent flow $(\XX',(v_{x';t})_{t \in I})$ of $(\XX^\infty, (\mu^\infty_t)_{t \in I})$ is a constant. By the relation between the Nash entropy and the almost self-similarity (cf. \cite[Proposition 7.1]{Bam20c}), one concludes that $(\XX',(v_{x';t})_{t \in I})$ is a metric soliton since its regular part admits an incomplete Ricci shrinker and the tangent flow itself is determined by its regular part due to the high codimension of the singular part (cf. \cite[Theorem 15.60, 15.69]{Bam20c}). Moreover, the singular set on each time-slice of $(\XX',(v_{x';t})_{t \in I})$ has Minkowski dimension $\le n-4$ (cf. \cite[Theorem 2.16]{Bam20c}). Furthermore, the fact that $x \in \mathcal R$ iff $\XX'$ is the Gaussian soliton follows from the $\ep$-regularity theorem \ref{thm:epr} and the convergence of the Nash entropies under \eqref{Fconv} (cf. \cite[Theorem 2.11, 2.14]{Bam20c}). Notice that if $n=4$, each time-slice of $(\XX',(v_{x';t})_{t \in I})$ is a smooth orbifold with isolated singularities since each tangent flow at any singular point of $(\XX',(v_{x';t})_{t \in I})$ is a flat cone (cf. \cite[Theorem 2.46]{Bam20c}). Therefore, we obtain Theorem \ref{thm:602} (2).

For Theorem \ref{thm:602} (3), the inequality $d_t \le d_{g_t}$ is clear. The opposite inequality is proved by showing that any $u \in C^0(\mathcal R_t)$ that is $1$-Lipschitz with respect to $d_{g_t}$ is also $1$-Lipschitz with respect to $d_t$ (cf. \cite[Theorem 15.28 (c)]{Bam20c}). The argument uses the high codimension of $\MS$, the fact that $\XX^\infty$ is future continuous at $t$, and the fact that $\mathcal R=\mathcal R^*$, which can be proved by using the $\ep$-regularity theorem and the convergence of the Nash entropies (cf. \cite[Corollary 15.47]{Bam20c}).

Once we know $\mathcal R=\mathcal R^*$, the diffeomorphisms in Theorem \ref{thm:603} (4) can be obtained by patching all local conventional Ricci flows into a Ricci flow spacetime by a center of mass construction (cf. \cite[Theorem 9.31]{Bam20b}). Notice that similar constructions are well-known for the Cheeger-Gromov convergence (cf. Remark 7.7 of~\cite{LLW21} and references therein).  All assertions Theorem \ref{thm:603} (4) are proved by smooth convergence. Therefore, Theorem \ref{thm:602} (4) is obtained.


As an application of the theory of $\IF$-convergence, we have the following backward pseudolocality theorem; see \cite[Theorem 2.47]{Bam20c}. 
Earlier backward pseudolocality can be found in~\cite[Corollary 11.6(b)]{Pe1}~\cite[Lemma 4.2]{CW12}~\cite[Theorem 4.7]{CW17B}~\cite[Theorem 1.5]{BZ17}.

\begin{thm}[Backward pseudolocality theorem] \label{thm:603}
For any $n \in \IN$ and $\alpha > 0$ there is an $\ep (n, \alpha) > 0$ such that the following holds.

Let $(M^n,g(t))_{t \in I}$ be a Ricci flow induced by a Ricci shrinker. Given $(x_0, t_0) \in M \times I$ and $r > 0$, if
\begin{equation*}
|B_{t_0}(x_0,r)| \geq \alpha r^n, \qquad |Rm| \leq (\alpha r)^{-2} \quad \text{on} \quad B_{t_0}(x_0,r),
\end{equation*}
then
\[ |Rm| \leq (\ep r)^{-2} \quad \text{on} \quad P(x_0,t_0; (1-\alpha) r, -(\ep r)^2,0). \]
\end{thm}

Note that the combination of the above theorem with the forward pseudolocality (cf. Theorem 24 of~\cite{LW20}), we arrive at the two-sided pseudolocality.
Thus Theorem~\ref{thm:106} is proved.

Combining Theorem \ref{thm:603} and \eqref{E507a}, we have
\begin{cor}[Comparison of the curvature radii] \label{cor:601}
There exists a constant $C (n, A) > 1$ such that the following holds.

Let $(M^n,g(t))_{t \in I'}$ be a Ricci flow induced by a Ricci shrinker in $\MM(A)$. Then for any $(x,t) \in M \times I'$,
\begin{equation*}
r_{\emph{Rm}}(x,t) \le r_{\emph{Rm}}^-(x,t) \le r_{\emph{Rm}}^s(x,t) \le C r_{\emph{Rm}}(x,t).
\end{equation*}
\end{cor}
Another application is the following integral estimate using the quantitative stratification; see \cite[Theorem 2.28]{Bam20c}.
\begin{thm} \label{thm:604}
Let $(M^n, g(t))_{t<1}$ be a Ricci flow associated with a Ricci shrinker in $\MM(A)$. Then for any $(x_0, t_0) \in M \times (-\infty,1)$, $r > 0$ and $\ep>0$, 
\begin{align}
& \int_{[t_0 - r^2, t_0 + r^2] \cap (-\infty,1)} \int_{P^* (x_0, t_0; r) \cap M \times \{ t \}} |Rm|^{2-\ep}\, dV_t dt \notag \\
\leq & \int_{[t_0 - r^2, t_0 + r^2] \cap (-\infty,1)} \int_{P^* (x_0, t_0; r) \cap M \times \{ t \}} r_{\emph{Rm}}^{-4+2\ep} \, dV_t dt \leq C(n,A, \ep) r^{n-2+2\ep}. \label{eq:int}
\end{align}
\end{thm}
As a corollary, we prove
\begin{cor} \label{cor:602}
Let $(M^n,g,f,p)$ be a Ricci shrinker in $\MM(A)$. Then
\begin{align} 
\int_{d(p,\cdot) \le r} |Rm|^{2-\ep} \,dV \le& \int_{d(p,\cdot) \le r} r_{\emph{Rm}}^{-4+2\ep} \,dV \le C r^{n+2\ep-2}, \label{eq:int2}\\
\int_{d(p,\cdot) \ge 1} \frac{|Rm|^{2-\ep}}{d^{n+2\ep-2}(p,\cdot)} \,dV \le& \int_{d(p,\cdot) \ge 1} \frac{r_{\emph{Rm}}^{-4+2\ep}}{d^{n+2\ep-2}(p,\cdot)} \,dV \le C \label{eq:int3}
\end{align}
for any $\ep>0$ and $r \ge 1$, where $r_{\emph{Rm}}(\cdot)=r_{\emph{Rm}}(\cdot,0)$ and $C=C(n,A,\ep)$.
\end{cor}
\begin{proof}
We consider the Ricci flow $(M,g(t))_{t<1}$ associated with the given Ricci shrinker.
It follows from Proposition \ref{prop:Hcent1a} that
\begin{align*}
Q(p,0;1,0,1) \subset P^*(p,0;C_1,0,1) \subset P^*(p,0;C_1)
\end{align*}
for some constant $C_1=C_1(n,A)>1$. Therefore, it follows from Theorem \ref{thm:604} that with $(x_0,t_0)=(p,0)$ and $r=C_1$ that
\begin{align}
& \int_0^1 \int_{d_t(p,\cdot)<1} |Rm|^{2-\ep}\, dV_t dt \leq C(n,A, \ep). \label{e601}
\end{align}
Since $g(t)=(1-t)(\psi^t)^*g$ and $\psi^t$ is defined by \eqref{E201a}, we have $d_t(x,p)=\sqrt{1-t}d(x^t,p)$ and $|Rm|(x,t)=|Rm|(x^t)/(1-t)$, where $x^t=\psi^t(x)$. Therefore, we have
\begin{align} \label{e602}
\int_{d_t(p,x)<1} |Rm|^{2-\ep}(x,t)\, dV_t(x)=(1-t)^{\frac{n}{2}-2+\ep}\int_{d(x,p)<\frac{1}{\sqrt{1-t}}} |Rm|^{2-\ep}(x)\,dV(x).
\end{align}
By a change of variable with $t=1-r^{-2}$, it follows from \eqref{e601} and \eqref{e602} that
\begin{align}
& \int_1^{\infty} r^{1-2\ep-n} m(r) \,dr \leq C(n,A, \ep), \label{e603}
\end{align}
where
\begin{align*}
m(r):=\int_{d(\cdot,p)<r} |Rm|^{2-\ep} \,dV.
\end{align*}

We claim that there exists a sequence $r_i \to \infty$ such that
\begin{align}
\lim_{i \to \infty} \frac{m(r_i)}{r_i^{n+2\ep-2}}=0. \label{e603xa}
\end{align}
Otherwise, there exists a constant $\delta>0$ such that $m(r) \ge \delta r^{n+2\ep-2}$ for sufficiently large $r$. However, it contradicts \eqref{e603}.

We apply the integration by parts to \eqref{e603} from $1$ to $r_i$ to \eqref{e603} and obtain
\begin{align*}
\int_{1 \le d(p,\cdot) \le r_i} \frac{|Rm|^{2-\ep}}{d^{n+2\ep-2}(p,\cdot)} \,dV \le C(n,A,\ep)+m(r_i) r_i^{2-2\ep+n}.
\end{align*}
By letting $i \to \infty$, we have from \eqref{e603xa} that
\begin{align}\label{e603xb}
\int_{d(p,\cdot) \ge 1} \frac{|Rm|^{2-\ep}}{d^{n+2\ep-2}(p,\cdot)} \,dV \le C(n,A,\ep).
\end{align}
In addition, we have for any $r \ge 1$,
\begin{align*}
r^{2-2\ep-n} \int_{1\le d(p,\cdot) \le r} |Rm|^{2-\ep} \,dV \le \int_{1 \le d(p,\cdot) \le r} \frac{|Rm|^{2-\ep}}{d^{n+2\ep-2}(p,\cdot)} \,dV \le C(n,A,\ep).
\end{align*}
Therefore, for any $r \ge 1$,
\begin{align} \label{e603xc}
\int_{d(p,\cdot) \le r} |Rm|^{2-\ep} \,dV \le C(n,A,\ep) r^{n+2\ep-2},
\end{align}
since $m(1)$ is bounded by \eqref{e603}. In sum, the inequalities involving $|Rm|$ in \eqref{eq:int2} and \eqref{eq:int3} are proved.

Notice that for any $(x_0,t_0) \in M \times (-\infty,1)$, it follows from the definition of $\psi^t$ \eqref{E201a} that
\begin{align*}
\psi^{\theta(t)} \circ \psi^{t_0}=\psi^t,
\end{align*}
where $\theta(t):=\frac{t-t_0}{1-t_0}$. Therefore, for any $t<1$,
\begin{align*}
g(t)=(1-t)(\psi^t)^* g=(1-t_0)(1-\theta(t))(\psi^{t_0})^* (\psi^{\theta(t)})^*g=(1-t_0)(\psi^{t_0})^*g(\theta(t)).
\end{align*}
Therefore,
\begin{align*}
r_{\text{Rm}}(x_0,t_0)=\sqrt{1-t_0} r_{\text{Rm}}(x_0^{t_0},0).
\end{align*}

Now, the conclusion regarding $r_{\text{Rm}}$ can be proved similarly. 
\end{proof}
We end this section by proving a gap property for the volume ratio at infinity.

\begin{cor} \label{cor:603}
Let $(M^n,g,f,p)$ be a Ricci shrinker in $\MM(A)$. Suppose
\begin{align} \label{e604a}
\liminf_{r \to \infty} \frac{|B(p,r)|}{r^n}=0.
\end{align}
Then
\begin{align}
|B(p,r)| \le Cr^{n-2+\ep} \label{e604b}
\end{align}
for any $r\ge 1$ and some $C=C(n,A,\ep)$.
\end{cor}
\begin{proof}
We claim that $r_{\text{Rm}}(x)<2$ for any $x$. Indeed, if $r_{\text{Rm}}(x) \ge 2$, we have
\begin{align*}
|Rm|(y,t)<1
\end{align*}
for any $ y\in B(x,1)$ and $t<1$. By the same argument as in \cite[Corollary 9]{LW20}, we obtain that
\begin{align} \label{e604c}
\psi^t \lc B(x,1) \rc \subset B \lc p,\frac{c_1}{\sqrt{1-t}} \rc \setminus B\lc p,\frac{c_2}{\sqrt{1-t}} \rc
\end{align}
for $c_1>c_2>0$, if $t$ is sufficiently close to $1$. From the standard distance distortion and Theorem \ref{thm:volume1}, we obtain that
\begin{align} \label{e604d}
|\psi^t \lc B(x,1) \rc| \ge c_3 (1-t)^{-\frac{n}{2}}.
\end{align}
However, \eqref{e604c} and \eqref{e604d} contradict \eqref{e604a}.
Thus the desired inequality~\eqref{e604b} follows immediately from \eqref{eq:int2}.
\end{proof}

\newpage

\appendixpage
\addappheadtotoc
\appendix
\section{Integral estimates for the conjugate heat kernel} 
\label{app:A}

In this appendix, we generalize some integral estimates regarding the conjugate heat kernel from \cite[Section 6]{Bam20c} to Ricci flows associated with Ricci shrinkers. These estimates also hold for Ricci flows induced by Ricci shrinkers since they are scaling-invariant.

Throughout this appendix, let $(M^n,g(t))_{t<1}$ be the Ricci flow associated with a Ricci shrinker in $\MM(A)$. We fix a spacetime point $(x_0,t_0) \in M \times (-\infty,1)$ and set $dv_t=dv_{x_0,t_0;t}$ and $\tau=t_0-t$. Moreover, we define $w=w(x,t)$ and $b=b(x,t)$ by $w=H(x_0,t_0,x,t)=(4\pi(t_0-t))^{-\frac{n}{2}}e^{-b}$.

\begin{lem} \label{lem:A1}
There exists a constant $C=C(n,A)>1$ such that for any $0<\tau_0<\tau_1$,
\begin{align} \label{eq:A1}
\int_{t_0-\tau_1}^{t_0-\tau_0} \int_M  \left\{ |Rc|^2+|\na^2 b|^2 \right\} dv_t\,dt \le C \tau_0^{-1}\lc 1+\log \frac{\tau_1}{ \tau_0} \rc.
\end{align}
\end{lem}

\begin{proof}
Without loss of generality, we assume $t_0=0$.

From Corollary \ref{cor:301}, we have for any $\tau>0$ that
\begin{align} \label{eq:A2}
\int_M |\na b|^2+R\,dv_{-\tau } \le \frac{n}{2\tau}.
\end{align}
Direct calculation shows that
\begin{align*} 
\partial_t\int_{M} R w\phi^r\,dV_t&=\int_{M} \left\{ (\square R) w\phi^r- R \square^{*} (w\phi^r) \right\} dV_t \notag\\
&=\int_{M} 2|Rc|^2 w\phi^r+R\left\{w(\Delta \phi^r+\phi^r_t)+2\la \na w,\na \phi^r \ra \right\} \,dV_t \notag\\
&=\int_{M} \left\{2|Rc|^2\phi^r+R(\Delta \phi^r+\phi^r_t)-2R\la \na b,\na \phi^r \ra \right\} \,dv_t. 
\end{align*}
Integrating the above equation from $-\tau_1$ to $-\tau_0$, we obtain
\begin{align} \label{eq:A3}
&\int_{-\tau_1}^{-\tau_0}\int_M 2|Rc|^2 \phi^r \,dv_t dt \notag \\
\le & \int_{M} R\phi^r\,dv_{-\tau_0}+\int_{-\tau_1}^{-\tau_0}\int_M  \left\{ R(|\Delta \phi^r|+|\phi_t^r|)+R^2 |\na \phi^r|+|\na b|^2 |\na \phi^r|  \right\} dv_t dt. 
\end{align}

From \eqref{E202c} and Lemma \ref{L201}, $R$ increases at most quadratically. Combining \eqref{E205a}, \eqref{E205b}, \eqref{E205c} and \eqref{eq:A2}, it follows that the last integral in \eqref{eq:A3} tends to $0$ as $r \to \infty$. Therefore, we obtain
\begin{align} \label{eq:A3x}
\int_{-\tau_1}^{-\tau_0}\int_M |Rc|^2 \,dv_t dt \le \frac{1}{2} \int_{M} R\,dv_{-\tau_0} \le \frac{n}{4 \tau_0}.
\end{align}
On the other hand, it follows from \eqref{E311a} and Corollary \ref{cor:302} that
\begin{align} 
\int_{-\tau_1}^{-\tau_0} 2\tau \int_M \left|Rc+\na^2 b-\frac{g}{2\tau} \right|^2 \,dv_tdt \le -\WW_{(x_0,t_0)}(\tau_1) \le A. \label{eq:A4} 
\end{align}
By virtue of the elementary identity $(x-y)^2 \ge x^2/2-y^2$, it follows from \eqref{eq:A4} that
\begin{align} 
\tau_0 \int_{-\tau_1}^{-\tau_0} \int_M |Rc+\na^2 b|^2 \,dv_tdt \le \int_{-\tau_1}^{-\tau_0} \tau \int_M |Rc+\na^2 b|^2 \,dv_tdt \le A+\frac{n}{2} \log \frac{\tau_1}{ \tau_0}. \label{eq:A5} 
\end{align}
Combining \eqref{eq:A3x} and \eqref{eq:A5}, the conclusion follows immediately.
\end{proof}

\begin{lem} \label{lem:A2}
There exists a constant $C=C(n,A)>1$ such that the following estimates hold for any $t<t_0$ and $0 \le s \le 1/4$.
\begin{align} \label{eq:A6a}
\int_M \left\{1+|b|+\tau(|\Delta b|+|\na b|^2+R) \right\} e^{s b} \,dv_t \le C.
\end{align}
\end{lem}

\begin{proof}
We compute 
\begin{align} 
\frac{d}{d s} \int_M e^{s b} \,dv_t= \int_M b e^{s b} \,dv_t. \label{eq:A6b} 
\end{align}
Here, the differentiation under the integral sign is allowed by Theorem \ref{thm:heatupper3} and Remark \ref{rem:poten}. By the differential Harnack inequality \cite[Theorem 21]{LW20}, we calculate
\begin{align} 
\int_M b e^{s b} \,dv_t \le& \int_M (\tau(-2\Delta b+|\na b|^2-R)+n) e^{s b} \,dv_t \notag \\
= & \int_M ( \tau ((2s-1)|\na b|^2-R)+n) e^{s b} \,dv_t \le n \int_M e^{s b} \,dv_t, \label{eq:A6c} 
\end{align}
where the integration by parts in the equality can be justified similarly as in Remark \ref{rem:ibp}. Combining \eqref{eq:A6b} and \eqref{eq:A6c}, we obtain
\begin{align*} 
\int_M e^{s b} \,dv_t \le e^{n s}.
\end{align*}
On the other hand, it follows from Theorem \ref{thm:heatupper1} that $b \ge -A$. Therefore, it follows from \eqref{eq:A6c} and the above inequality that
\begin{align}
\int_M \left\{ |b|+\tau(|\na b|^2+R) \right\} e^{s b} \,dv_t \le C(n,A).
\label{eqn:RA20_2}
\end{align}
Applying the differential Harnack inequality and the integration by parts again, we obtain
\begin{align*} 
\int_M 2\tau| \Delta b| e^{s b} \,dv_t \le& \int_M \left\{ |u|+\tau(|\na b|^2+R)+|b|+n \right\}e^{s b} \,dv_t \\
=& \int_M \left\{ -u+\tau(|\na b|^2+R)+|b|+n \right\}e^{s b} \,dv_t \\
\le & \int_M \left\{ 2s\tau |\na b|^2+2|b|+2n \right\}e^{s b} \,dv_t \le C(n,A),  
\end{align*}
where $u=\tau(2\Delta b-|\na b|^2+R)+b-n \le 0$.  It is clear that~\eqref{eq:A6a} follows from the combination of (\ref{eqn:RA20_2}) and the above inequality. 
\end{proof}

\begin{lem} \label{lem:A3}
There exists a constant $C=C(n)>1$ such that the following estimates hold for any $t<t_0$ and $0 \le s \le 1/4$.
\begin{align} \label{eq:Axx1}
\int_M |\na b|^4e^{s b} \,dv_t \le C \int_M |\na^2 b|^2 e^{s b} \,dv_t.
\end{align}
\end{lem}

\begin{proof}
In the proof, all constants $C>1$ depend only on $n$, which may be different line by line.

We compute for $s \le 1/4$ that
\begin{align} 
&\int_M |\na b|^4 \phi^r e^{s b} \,dv_t \notag \\
=& (4\pi\tau)^{-\frac n 2} \int_M |\na b|^4 \phi^r e^{(s-1) b} \,dV_t \notag\\
=& (4\pi\tau)^{-\frac n 2} (s-1)^{-1} \int_M |\na b|^2 \la \na b, \na e^{(s-1)b} \ra \phi^r \,dV_t \notag\\
=& (4\pi\tau)^{-\frac n 2} (1-s)^{-1} \int_M \lc 2\na^2 b(\na b,\na b)+|\na b|^2 \Delta b \rc \phi^r e^{(s-1)b} \,dV_t+Z\notag\\
\le & C (4\pi\tau)^{-\frac n 2} (1-s)^{-1} \int_M |\na^2 b| |\na b|^2 \phi^r e^{(s-1)b} \,dV_t+Z\notag \\
\le & \frac{1}{4} \int_M |\na b|^4 \phi^r e^{s b} \,dv_t+C \int_M |\na^2 b|^2 \phi^r e^{s b}\, dv_t+Z, \label{eq:Axx2}
\end{align}
where the remainder
\begin{align} 
Z:&=(1-s)^{-1} \int_M |\na b|^2 \la \na b,\na \phi^r \ra e^{s b} \,dv_t \le 2 \int_M |\na b|^3 |\na \phi^r|e^{s b} \,dv_t \notag \\
&\le \frac{1}{4} \int_M |\na b|^4 \phi^r e^{s b} \,dv_t+4\int_M |\na b|^2 |\na \phi^r|^2 (\phi^r)^{-1} e^{s b} \,dv_t. \label{eq:Axx3}
\end{align}

Applying Lemma \ref{lem:A2} and \eqref{E205a}, we conclude from \eqref{eq:Axx2} and \eqref{eq:Axx3} that
\begin{align} 
\int_M |\na b|^4e^{s b} \phi^r \,dv_t \le C \int_M |\na^2 b|^2 e^{s b} \phi^r \,dv_t+\ep(r) \label{eq:Axxr}
\end{align}
where $\ep(r) \to 0$ as $r \to \infty$.  Thus we arrive at \eqref{eq:Axx1} by letting $r \to \infty$ in the above inequality. 
\end{proof}

The main result of this section is the following spacetime integral estimate.

\begin{prop} \label{prop:A1}
There exists a constant $C=C(n,A)>1$ and $\bar s=\bar s(n) <1$ such that the following estimates hold for any $r>0$, $0< \theta <1/2$ and $ s\le \bar s$.
\begin{align} \label{eq:A7a}
\int_{t_0-\tau_0}^{t_0-\theta \tau_0} \int_M \tau(|Rc|^2+|\na^2 b|^2+|\na b|^4) e^{s b} \,dv_t dt \le C \log \theta^{-1}.
\end{align}
\end{prop}

\begin{proof}
In the proof, all constants $C$ depend on $n$, and $C'$ depend on $n$ and $A$. Moreover, we use $\ep(r)$ to denote a function independent of $t$ such that $\ep(r) \to 0$ if $r \to \infty$. Those terms may be different line by line. Without loss of generality, we assume $t_0=0$.

We set $u=\tau(2\Delta b-|\na b|^2+R)+b-n \le 0$. Recall that from \cite{Pe1}, we have the celebrated identity
\begin{align*} 
\square^*(uw)=-2\tau \left|Rc+\na^2 b-\frac{g}{2\tau} \right|^2 w.
\end{align*}
Moreover, we have
\begin{align*} 
\square b=-2\Delta b+|\na b|^2-R+\frac{n}{2\tau}=\tau^{-1}(b-u-n/2).
\end{align*}
Direct computation shows that
\begin{align*} 
&\partial_t\int_{M} uw e^{s b} \phi^r\,dV_t \notag \\
=&\int_{M} \left\{ \square (e^{sb} \phi^r) uw- e^{sb} \phi^r \square^{*} (uw) \right\} dV_t \notag\\
=&\int_{M} \left\{ \lc (\square e^{sb}) \phi^r+e^{sb} (\square \phi^r)-2\la \na \phi^r, \na e^{sb} \ra \rc uw+2\tau \left|Rc+\na^2 b-\frac{g}{2\tau} \right|^2 w e^{sb} \phi^r \right\} dV_t \\
=&\int_{M} \left\{\lc (s \square b-s^2|\na b|^2) e^{sb} \phi^r+e^{sb} (\square \phi^r)-2\la \na \phi^r, \na e^{sb} \ra \rc u+2\tau \left|Rc+\na^2 b-\frac{g}{2\tau} \right|^2  e^{sb} \phi^r \right\} wdV_t,  \\
=&\int_{M} \left\{  \lc (s\tau^{-1}(b-u-n/2)-s^2|\na b|^2)  \phi^r+\square \phi^r-2s \la \na \phi^r, \na b \ra \rc u+2\tau \left|Rc+\na^2 b-\frac{g}{2\tau} \right|^2  \phi^r  \right\} e^{sb} dv_t. 
\end{align*}
It follows that
\begin{align}
&\partial_t\int_{M} uw e^{s b} \phi^r\,dV_t \notag \\
\ge &\int_{M} \left\{ (s\tau^{-1}(b-u-n/2))u \phi^r+\lc \square \phi^r-2s \la \na \phi^r, \na b \ra \rc u+2\tau \left|Rc+\na^2 b-\frac{g}{2\tau} \right|^2 \phi^r \right\} e^{sb} dv_t \notag \\
\ge &\int_{M} \left\{ -C s\tau^{-1}(u^2+b^2+1) \phi^r+\lc \square \phi^r-2s \la \na \phi^r, \na b \ra \rc u+2\tau \left|Rc+\na^2 b-\frac{g}{2\tau} \right|^2 \phi^r \right\} e^{sb} dv_t \notag\\
\ge &\int_{M} \left\{ -C s \lc \tau((\Delta b)^2+|\na b|^4+R^2)+\tau^{-1}(b^2+1) \rc +2\tau \left|Rc+\na^2 b-\frac{g}{2\tau} \right|^2 \right\} \phi^r e^{sb} dv_t +X_t, \label{eq:A7b}
\end{align}
where 
\begin{align*} 
X_t := \int_{M} \lc \square \phi^r-2s \la \na \phi^r, \na b \ra \rc u e^{sb} dv_t.
\end{align*}
Define $X' :=\int_{-\tau_0}^{-\theta \tau_0} |X_t|\,dt$. Then it follows from Lemma \ref{lem:A2} and inequalities \eqref{E205a} to \eqref{E205d} that for any positive $\delta$ we have
\begin{align} 
|X'|\le& \int_{-\tau_0}^{-\theta \tau_0} \int_{M} (|\square \phi^r|+2s|\na \phi^r||\na b|) |u| e^{sb}\, dv_t dt \notag \\
\le& \ep(r)+\int_{-\tau_0}^{-\theta \tau_0} \int_{M} \lc \delta^{-1} |\na \phi^r|^2(\phi^r)^{-1}\tau |\na b|^2+\delta \tau^{-1}u^2 \phi^r \rc e^{sb}\, dv_t dt \notag \\
\le& \ep(r)+\delta \int_{-\tau_0}^{-\theta \tau_0} \int_{M} \tau^{-1}u^2 \phi^r e^{sb}\, dv_t dt. \label{eq:xt001}
\end{align}

It is clear from the definition of $u$ that $u^2 \le C\lc \tau^2(|\Delta b|^2+|\na b|^2+R))+b^2+1 \rc$. In addition, since $-A \le b \le -\tau(2\Delta b-|\na b|^2+R)+n$, we have
\begin{align} 
b^2 \le C' \lc \tau^2(|\Delta b|^2+|\na b|^4+R^2)+1 \rc. \label{eq:xt002}
\end{align}
Combining these facts with \eqref{eq:Axxr}, we may choose $\delta$ in \eqref{eq:xt001} sufficiently small such that
\begin{align} 
|X'| \le \ep(r)+\frac{1}{10} \int_{-\tau_0}^{-\theta \tau_0} \int_{M} \tau(|\na^2 b|^2+|Rc|^2) \phi^r e^{sb} \, dv_t dt+ \log \theta^{-1}. \label{eq:xt003}
\end{align}

Similarly, we compute
\begin{align} 
&\partial_t\int_{M} \tau R e^{s b}w \phi^r\,dV_t \notag \\
=&\int_{M} \left\{ \square (\tau R e^{s b}) \phi^r w- \tau R e^{s b} \square^{*} (\phi^r w) \right\} dV_t \notag\\
=&\int_{M} \left\{ \lc \square (\tau R) e^{sb}+\tau R \square e^{sb}-2\tau \la \na R, \na e^{sb} \ra \rc\phi^r w+ \tau R e^{s b} \lc (\Delta \phi^r+\phi^r_t) w+2 \la \na w, \na \phi^r \ra \rc \right\} dV_t \notag\\
=&\int_{M} \left\{ \lc \square (\tau R) e^{sb}+\tau R \square e^{sb}\rc\phi^r w+2\tau R \lc \Delta e^{sb}-\la \na e^{sb},\na b \ra \rc \phi^r w \right\} dV_t +Y_t \notag\\
= &\int_{M} \left\{ 2\tau|Rc|^2-R+\tau R \lc s \square b+(s^2-2s) |\na b|^2 +2s \Delta b \rc \right\} \phi^r e^{sb} dv_t+Y_t \notag \\
\ge & \int_{M} \left\{ 2\tau|Rc|^2-R-C s \lc \tau(R^2+|\na b|^4+(\Delta b)^2)+R \rc \right\} \phi^r e^{sb} dv_t+Y_t, \label{eq:A7c}
\end{align}
where
\begin{align*} 
Y_t:&= \int_{M}\tau R e^{s b} \lc (\Delta \phi^r+\phi^r_t) w+2 \la \na w, \na \phi^r \ra+2s \la \na b, \na \phi^r \ra \rc \, dV_t \\
&=\int_{M}\tau R \lc \Delta \phi^r+\phi^r_t+(2s-2) \la \na b, \na \phi^r \ra \rc e^{s b}\, dv_t.
\end{align*}
We define similarly $Y:=\int_{-\tau_0}^{-\theta \tau_0} |Y_t|\,dt$. Then it follows from Lemma \ref{lem:A2} and inequalities \eqref{E205a} to \eqref{E205d} that
\begin{align} 
|Y'| \le& \ep(r)+C\int_{-\tau_0}^{-\theta \tau_0} \int_{M} \lc \delta^{-1} |\na \phi^r|^2(\phi^r)^{-1}\tau |\na b|^2+\delta \tau R^2 \phi^r \rc e^{sb}\, dv_t dt \notag \\
\le& \ep(r)+C\int_{-\tau_0}^{-\theta \tau_0} \int_{M} \delta \tau R^2 \phi^r e^{sb}\, dv_t dt \notag \\
\le& \ep(r)+\frac{1}{10} \int_{-\tau_0}^{-\theta \tau_0} \int_{M} \tau(|\na^2 b|^2+|Rc|^2) \phi^r e^{sb} \, dv_t dt, \label{eq:xt004}
\end{align}
for $\delta$ sufficiently small.

Combining \eqref{eq:A7b} and \eqref{eq:A7c}, we obtain
\begin{align} 
&\partial_t\int_{M} (\tau R+u) e^{s b} \phi^r\,dv_t \notag \\
\ge &\int_{M} \left\{ -C s \lc \tau(|\na b|^4+|\na^2 b|^2+R^2)+\tau^{-1}+R \rc +2\tau \left|Rc+\na^2 b-\frac{g}{2\tau} \right|^2+2\tau|Rc|^2-R \right\} \phi^r e^{sb} dv_t \notag \\
& +X_t+Y_t \notag \\
\ge &\int_{M} \left\{ -C s \lc \tau(|\na b|^4+|\na^2 b|^2+R^2) \rc +\tau (|\na^2 b|^2+|Rc|^2) \right\} \phi^r e^{sb} dv_t+X_t+Y_t-C' \tau^{-1}, \label{eq:A7e}
\end{align}
where we have used Lemma \ref{lem:A2}.

If $s$ is sufficiently small, it follows from \eqref{eq:A7e} and \eqref{eq:Axxr} that
\begin{align} 
&\partial_t\int_{M} (\tau R+u) e^{s b} \phi^r\,dv_t \notag \\
\ge & \frac{1}{2}\int_{M} \tau (|\na^2 b|^2+|Rc|^2) e^{sb} \phi^r\, dv_t+X_t+Y_t-C' \tau^{-1}+\ep(r) \label{eq:A7f}.
\end{align}
By integration from $-\tau_0$ to $-\theta \tau_0$, we obtain from \eqref{eq:A7f}, Lemma \ref{lem:A2}, \eqref{eq:xt003} and \eqref{eq:xt004} that
\begin{align*} 
\int_{-\tau_0}^{-\theta \tau_0} \int_{M} \tau (|\na^2 b|^2+|Rc|^2) e^{sb} \phi^r\, dv_t dt \le C' \log \theta^{-1}+\ep(r). 
\end{align*}
Letting $r \to \infty$, we obtain
\begin{align} 
\int_{-\tau_0}^{-\theta \tau_0} \int_{M} \tau (|\na^2 b|^2+|Rc|^2) e^{sb}\, dv_t dt \le C' \log \theta^{-1}. \label{eq:A8a}
\end{align}
Thus the inequality \eqref{eq:A7a} follows from the combination of \eqref{eq:A8a} and Lemma \ref{lem:A3}.

\end{proof}

\vskip10pt

Yu Li, Institute of Geometry and Physics, University of Science and Technology of China, No. 96 Jinzhai Road, Hefei, Anhui Province, 230026, China; yuli21@ustc.edu.cn.\\

Bing Wang, Institute of Geometry and Physics, School of Mathematical Sciences, University of Science and Technology of China, No. 96 Jinzhai Road, Hefei, Anhui Province, 230026, China; topspin@ustc.edu.cn.\\

\end{document}